\documentclass[11pt, a4paper]{article}

\usepackage{amsmath}
\usepackage{amssymb}
\usepackage{amsthm}
\usepackage{appendix}

\newtheorem{theorem}{Theorem}[section]
\newtheorem{corollary}{Corollary}[section]

\newtheorem{definition}{Definition}[section]
\newtheorem{lemma}{Lemma}[section]
\newtheorem{remark}{Remark}[section]
\newtheorem{example}{Example}[section]
\newtheorem*{thmErdos}{Theorem A}
\newtheorem*{thmJMPA}{Theorem B}
\newtheorem*{thmc}{Theorem C}
\newtheorem*{thmd}{Theorem D}
\newtheorem*{thme}{Theorem E}
\newtheorem*{propa}{Proposition A}
\newtheorem*{lema}{Lemma A}

\begin{document}
\title{The closed span of some Exponential system $E_{\Lambda}$ in the spaces $L^p(\gamma,\beta)$,
properties of a Biorthogonal family to $E_{\Lambda}$ in $L^2(\gamma,\beta)$,
Moment problems, and a differential equation of Carleson}
\author{Elias Zikkos\\
Khalifa University, Abu Dhabi, United Arab Emirates\\
email address:  elias.zikkos@ku.ac.ae and eliaszikkos@yahoo.com}

\maketitle

\begin{abstract}
A set of complex numbers $\Lambda=\{\lambda_n,\mu_n\}_{n=1}^{\infty}$ with multiple terms
\[
\{\lambda_n,\mu_n\}_{n=1}^{\infty}:=
\{\underbrace{\lambda_1,\lambda_1,\dots,\lambda_1}_{\mu_1 - times},
\underbrace{\lambda_2,\lambda_2,\dots,\lambda_2}_{\mu_2 - times},\dots,
\underbrace{\lambda_k,\lambda_k,\dots,\lambda_k}_{\mu_k - times},\dots\}
\]
is said to belong to the $\bf ABC$ class if it satisfies three conditions:
$\bf (A)$ $\sum_{n=1}^{\infty}\mu_n/|\lambda_n|<\infty$,
$\bf (B)$ $\sup_{n\in\mathbb{N}}|\arg\lambda_n|<\pi/2$,
$\bf (C)$ $\Lambda$ is an interpolating variety for the space of entire functions of exponential type zero.
Assuming that $\Lambda\in\bf ABC$, we characterize
in the spirit of the M\"{u}ntz-Sz\'{a}sz theorem, the closed span of its associated exponential system
\[
E_{\Lambda}:=\{x^k e^{\lambda_n x}:\, n\in\mathbb{N},\,\, k=0,1,2,\dots,\mu_n-1\}
\]
in the Banach spaces $L^p(\gamma,\beta)$,  where $-\infty<\gamma<\beta<\infty$ and $p\ge 1$.
Related to $E_{\Lambda}$, we explore the properties of its unique biorthogonal sequence
\[
r_{\Lambda}=\{r_{n,k}:\, n\in\mathbb{N},\, k=0,1,\dots,\mu_n-1\}\subset\overline{\text{span}}(E_{\Lambda})
\]
in $L^2(\gamma,\beta)$. We prove that
the closed spans of $r_{\Lambda}$ and $E_{\Lambda}$ in $L^2(\gamma, \beta)$ are equal.
In addition, we show that the system $E_{\Lambda}$ is hereditarily complete in this closure. Moreover, we prove that
every $f\in \overline{\text{span}}(E_{\Lambda})$ in $L^2(\gamma,\beta)$ admits the Fourier-type series representation
\[
f(t)=\sum_{n=1}^{\infty}\left(\sum_{k=0}^{\mu_n-1} \langle f, r_{n,k} \rangle \cdot t^k\right)e^{\lambda_n t},
\quad \text{almost everywhere on}\,\, (\gamma,\beta),
\]
with the series converging uniformly on closed subintervals of $(\gamma, \beta)$,
and extended analytically in some sector of the complex plane.
Furthermore, we obtain a sharp upper bound for the norm of each $r_{n,k}$ by showing that
for every $\epsilon>0$ there is a positive constant $m_{\epsilon}$, independent of $n$ and $k$,
but depending on $\Lambda$ and $(\beta-\gamma)$, so that
\[
||r_{n,k}||_{L^2 (\gamma,\beta)}\le m_{\epsilon}e^{(-\beta+\epsilon)\Re\lambda_n}\qquad \forall\,\, n\in\mathbb{N},\quad\text{and}\quad k=0,1,\dots,\mu_n-1.
\]
The above results allow us to conclude that given a set of numbers $\{d_{n,k}\}$
satisfying $d_{n,k}=O(e^{a\Re\lambda_n})$ for $a<\beta$,
there exists a unique Fourier-type series $f\in L^2 (\gamma,\beta)$,
such that $f$ is a solution to the Moment problem
\[
\int_{\gamma}^{\beta} f(t)\cdot t^k e^{\overline{\lambda_n} t}\, dt=d_{n,k},\qquad \forall\,\, n\in\mathbb{N}\quad
\text{and}\quad k=0,1,\dots ,\mu_n-1.
\]
Finally, we characterize the solution space of a differential equation of infinite order, studied by L. Carleson.

We point out that our results depend heavily on finding a sharp lower bound for the Distance
between an element of $E_{\Lambda}$ and the closed span  of the remaining elements in $L^p(\gamma, \beta)$. The crucial tool
employed is a certain entire function introduced by Luxemburg and Korevaar.

\end{abstract}

Keywords: Exponential Systems, Closed Span, Hereditary completeness, Distances, Biorthogonal Families, Moment Problems,
Taylor-Dirichlet Series, Bessel and Riesz-Fischer Sequences,
Differential Equations Of Infinite Order.

AMS 2010 Mathematics Subject Classification. 30B60, 30B50, 46E15, 46E20, 46B20, 41A30, 34A35.

\section{Introduction and the Main Results}
\setcounter{equation}{0}

\subsection{Motivation}

The classical M\"{u}ntz-Sz\'{a}sz theorem answers the following question posed by S. N. Bernstein.

``Find necessary and sufficient conditions on a strictly increasing sequence $\{\lambda_n\}_{n=1}^{\infty}$
of positive real numbers diverging to infinity,
so that the span of the system $\{1\}\cup\{x^{\lambda_n}\}_{n=1}^{\infty}$ is dense in the space of continuous functions $C[0,1]$.''

M\"{u}ntz and Sz\'{a}sz proved that
\[
\overline{\text{span}}(\{1\}\cup\{x^{\lambda_n}\}_{n=1}^{\infty})=C[0,1]\quad \text{if\,\, and\,\, only\,\, if}\quad
\sum_{n=1}^{\infty}\frac{1}{\lambda_n}=\infty.
\]

This result was extended later on by J. A. Clarkson, P. Erd\H{o}s, L. Schwartz,
W. A. J. Luxemburg, J. Korevaar, P. Borwein,  T. Erdelyi, and W. B. Johnson (see \cite{CE,LK,BE,BE1,EJ,E1,E2}).
In their papers one finds that the above condition is still both necessary and sufficient
when one replaces the interval $[0,1]$ by an interval $[a,b]$ away from the origin with $0<a<b<\infty$
or even by a compact set $K\subset [0,\infty)$  of positive Lebesgue measure. The results also hold for
the $L^p (a,b)$ and $L^p(K)$ spaces, $p\ge 1$. For further reading, one may also consult
the works in \cite{Redheffer,Dzh,Khabibullin}, the survey articles \cite{Pinkus,Almira}
as well as the book \cite{BE}. Moreover, we remark that there is an ongoing research
on M\"{u}ntz-Sz\'{a}sz type problems (see \cite{Lefevre,Jaming}).\\

If $\sum_{n=1}^{\infty}1/\lambda_n<\infty$ then
$\overline{\text{span}}(\{1\}\cup\{x^{\lambda_n}\}_{n=1}^{\infty})$ is a proper subspace of $C[0,1]$ and
this raises the question of describing the closure. The first to deal with this problem were
Clarkson and Erd\H{o}s in \cite{CE} as well as L. Schwartz in his Ph.D Thesis,
who proved that if the $\lambda_n$ are positive $\bf integers$, then any function $f$
belonging to $\overline{\text{span}}(\{1\}\cup\{x^{\lambda_n}\}_{n=1}^{\infty})$
in $C[0,1]$, is extended to an analytic function
throughout the interior of the unit disk $\cal D$=$\{z:\, |z|<1\}$, admitting a power series representation
of the form
\[
f(z)=\sum_{n=0}^{\infty} a_n z^{\lambda_n},
\]
which converges uniformly on compact subsets of $\cal D$.

The above result bears the name ``Clarkson-Erd\H{o}s-Schwartz Phenomenon''.
Several generalizations were given later on \cite{LK,Dzh,BE1,EJ,E1,E2} where the terms of the sequence $\{\lambda_n\}_{n=1}^{\infty}$
were permitted to be positive real numbers or even complex numbers.
We mention here a result by Luxemburg and Korevaar.
\begin{thmErdos}(\cite[Theorem 8.2]{LK})
Let $\{\lambda_n\}_{n=1}^{\infty}$ be a sequence of distinct complex numbers satisfying the condition
\begin{equation}\label{LKcondition}
\sum_{n=1}^{\infty}\frac{1}{|\lambda_n|}<\infty,\qquad
\sup_{n\in\mathbb{N}}|\arg \lambda_n|<\pi /2,\qquad\text{and}\qquad  |\lambda_n-\lambda_k|\ge c|n-k|,\,\, c>0,\quad \forall\,\, n\not= k.
\end{equation}
Let $f$ be in $C[a,b]$ (or in $L^p(a,b)$) for $0\le a<b<\infty$.
Then $f$ belongs to the closed span of the system $\{x^{\lambda_n}\}_{n=1}^{\infty}$
in $C[a,b]$ (or in $L^p(a,b)$) if and only if $f$ admits the series representation
\[
f(x)=\sum_{n=1}^{\infty} a_n x^{\lambda_n},\qquad \forall \,\,x\in (a,b)\qquad (or\,\, almost\,\, everywhere\,\, in\,\, (a,b))
\]
with the series converging uniformly on compact subsets of $(a,b)$.
\end{thmErdos}

Now, by a change of variables the system $\{x^{\lambda_n}\}_{n=1}^{\infty}$ becomes an exponential system
$\{e^{-\lambda_n x}\}_{n=1}^{\infty}$. In several cases,
the non-dense span of such a system or even of more general ones of the form
\[
\{x^ke^{-\lambda_n x}:\,\, n\in\mathbb{N}, \,\, k=0,1,2,\dots,m\}\qquad m\in\mathbb{N},
\]
in the space $L^2(0,T)$, has led mathematicians to obtain a lower bound for the distance between an element
$x^ke^{-\lambda_n x}$ of the system and the closed span of the remaining ones in $L^2(0,T)$. This in turn yields
a sharp upper bound for the norm of the elements of a family biorthogonal to exponential systems in $L^2(0,T)$,
as for example in the following result.
\begin{thmJMPA}\cite[Theorem 1.2]{2011JPA}
Let $\{\lambda_n\}_{n=1}^{\infty}$ be a sequence of distinct complex numbers satisfying the condition $(\ref{LKcondition})$
and let $m\in\mathbb{N}$.
Fix some $T>0$: then there exists a biorthogonal family
\[
\{r_{n,k}:\, n\in\mathbb{N},\, k=0,1,\dots, m\}\subset L^2(0,T)
\]
to the exponential system $\{x^ke^{-\lambda_n x}:\,\, n\in\mathbb{N}, \,\, k=0,1,2,\dots,m\}$ in $L^2 (0,T)$ and
belonging to its closed span in $L^2 (0,T)$, such that for every $\epsilon>0$ there is a constant $m_{\epsilon}>0$,
independent of $n\in\mathbb{N}$ and $k=0,1,\dots,m-1$, but depending on $T$, so that
\begin{equation}\label{rnkboundJMPA}
||r_{n,k}||_{L^2 (0,T)} \le  m_{\epsilon}e^{\epsilon\Re\lambda_n},\qquad \forall\,\, n\in\mathbb{N},\quad k=0,1,\dots,m-1.
\end{equation}
\end{thmJMPA}
It is well known that upper bounds as above yield solutions $f$ to Moment Problems of type
\[
\int_{0}^{T} f(t)\cdot t^k e^{-\lambda_n t} \, dt=
d_{n,k}\qquad \forall\,\, n\in\mathbb{N}\quad k=0,1,2,\dots, m,
\]
and these proved to be crucial in Control Theory for Partial Differential Equations,
starting with the pioneering work of Fattorini and Russell \cite{1971FR} and followed by a vast amount of work done after that.
We mention here several papers of the last decade on the above topics
\cite{2010JFA,Glass2010,2011JPA,2011MicuZuazua,2011Micu,2012Micu,2014JFA,2014Control,
Lissy2014,2014Bugariu,2014Bugariua,2014Bugariub,2014Bugariuc,
2016JMAA,2016Bugariu,2016Micu,Lissy2017,Lissy2017b,Cannarsa2017,2018Micu,2018Allonsius,2019JPA,Lissy2019,2020AHL,
2021Allonsius,Cannarsa2020a,Cannarsa2020b,
2021Bhandari,2021Bhandarib,2022Gonzalez-Burgos}.\\

Motivated by all these, our main goals in this article are to obtain substantial generalizations of Theorems $\bf A$ and $\bf B$
as well as to derive solutions $f$ to Moment Problems, with $f$ not only belonging in $L^2 (0,T)$ but also extended
analytically in some sector of the complex plane. In addition, we will characterize the solution space of
a differential equation of infinite order studied in the past by L. Carleson \cite{Carleson}.

We will describe in detail our various goals in the upcoming subsections and state our results as well.
But first, let us introduce some notations and definitions.

\subsection{Notation}

For $p\ge 1$ and real numbers $\gamma$ and $\beta$ with $\gamma<\beta$,
let $L^p (\gamma,\beta )$ be the Banach space of complex-valued measurable functions defined
on the bounded interval $(\gamma,\beta)$ on the real line such that
$\int_{\gamma}^{\beta}|f(x)|^p\, dx<\infty$. The $L^p (\gamma,\beta )$ space is equipped with the norm
\[
||f||_{L^p (\gamma,\beta)}:=\left(\int_{\gamma}^{\beta}|f(x)|^p\, dx\right)^{\frac{1}{p}},
\]
and $L^2(\gamma,\beta)$ is a Hilbert space once endowed with the inner product
\[
\langle f,g \rangle:=\int_{\gamma}^{\beta} f(x)\overline{g(x)}\, dx.
\]

Next, for a set of complex numbers having multiple terms
\[
\{\lambda_n,\mu_n\}_{n=1}^{\infty}:=\{\underbrace{\lambda_1,\lambda_1,\dots,\lambda_1}_{\mu_1 - times},
\underbrace{\lambda_2,\lambda_2,\dots,\lambda_2}_{\mu_2 - times},\dots,
\underbrace{\lambda_k,\lambda_k,\dots,\lambda_k}_{\mu_k - times},\dots\},
\]
we refer to it by the name  multiplicity sequence $\Lambda=\{\lambda_n,\mu_n\}_{n=1}^{\infty}$,
where
\begin{itemize}
\item
$\{\lambda_n\}_{n=1}^{\infty}$ is a sequence of distinct complex or real numbers  diverging to infinity, enumerated
so that $0<|\lambda_{n}|\le |\lambda_{n+1}|$ for all $n\in \mathbb{N}$ and $-\pi<\arg\lambda_n<\arg\lambda_{n+1}\le \pi$ whenever $|\lambda_n|=|\lambda_{n+1}|$.
\item
$\{\mu_n\}_{n=1}^{\infty}$ is a sequence of positive integers, not necessarily bounded, that is $\mu_n\not= O(1)$ is possible.
\end{itemize}
Each term $\lambda_n$ in $\Lambda$ is repeated exactly $\mu_n$ times.
If $\mu_n=1$ for all $n\in\mathbb{N}$ we simply write $\Lambda=\{\lambda_n\}_{n=1}^{\infty}$.

To a multiplicity sequence $\Lambda=\{\lambda_n,\mu_n\}_{n=1}^{\infty}$, we associate the exponential system
\[
E_{\Lambda}:=\{e_{n,k}:\,\, n\in\mathbb{N}, \,\, k=0,1,2,\dots,\mu_n-1\}\qquad
\text{where}\qquad e_{n,k}(x):=x^ke^{\lambda_n x}.
\]
We denote by $\text{span}(E_{\Lambda})$ the set of all finite linear combinations of elements from $E_{\Lambda}$,
that is the set of all exponential polynomials of the form
\[
P(x)=\sum_{n=1}^{m}\left(\sum_{k=0}^{\mu_n-1}c_{m,n,k}x^k\right)e^{\lambda_n x},\quad
c_{m,n,k}\in\mathbb{C}.
\]
We say that a function $f:(\gamma, \beta)\mapsto \mathbb{C}$ belongs to the closed span of the system $E_{\Lambda}$ in $L^p(\gamma,\beta)$,
if for every $\epsilon>0$ there is an exponential polynomial $P$ so that $||f-P||_{L^p(\gamma,\beta)}<\epsilon$.

We also say that a family of functions
\begin{equation}\label{BioFamily}
r_{\Lambda}:=\{r_{n,k}:\, n\in\mathbb{N},\, k=0,1,\dots,\mu_n-1\}\subset L^2(\gamma, \beta)
\end{equation}
is a biorthogonal sequence to the exponential system $E_{\Lambda}$ in $L^2 (\gamma, \beta)$ if
\[
\langle r_{n,k}, e_{n,k} \rangle = \int_{\gamma}^{\beta}r_{n,k}(x) x^l e^{\overline{\lambda_j} x}\, dx =
\begin{cases} 1, & j=n,\,\,
l=k, \\ 0,  & j=n,\,\,
l\in\{0,1,\dots,\mu_n-1\}\setminus\{k\}, \\ 0, &
j\not=n,\,\, l\in\{0,1,\dots,\mu_j-1\}.\end{cases}
\]

And finally, we denote by $A^0_{|z|}$ the space of entire functions of exponential type zero.
An entire function $F$ belongs to $A^0_{|z|}$,
if for every $\epsilon>0$ there is a positive constant $M_{\epsilon}$,
depending only on $\epsilon$ and $F$, such that
$|F(z)|\le M_{\epsilon}e^{\epsilon |z|}$ for all $z\in \mathbb{C}$.

\subsection{The $ABC$ class}

We remark that our various results in this article hold under the assumption that $\Lambda$ belongs to
a certain class denoted by $ABC$.

\begin{definition}\label{ABC}
$ABC$ is the class of the multiplicity sequences $\Lambda=\{\lambda_n,\mu_n\}_{n=1}^{\infty}$ that satisfy the following three conditions.

$(A):$ The M\"{u}ntz-Sz\'{a}sz convergence condition
\begin{equation}\label{convergencecondition}
\sum_{n=1}^{\infty}\frac{\mu_n}{|\lambda_n|}<\infty.
\end{equation}

$(B):$ The $\lambda_n$ are in some sector of the right half-plane such that for some $\eta\in [0,\pi/2)$ we have
\begin{equation}\label{lessthan}
\sup_{n\in\mathbb{N}}|\arg\lambda_n|\le\eta.
\end{equation}

$(C):$ $\Lambda$ is an $interpolating\,\, variety$  for the space $A^0_{|z|}$.
\end{definition}
\begin{remark}
For $\beta\in\mathbb{R}$, due to $(\ref{lessthan})$ and if $\eta>0$, we will be considering the sector
\begin{equation}\label{opensector}
\Theta_{\eta ,\beta}:=\left\{z:\left|\frac{\Im z}{\Re (z-\beta)}\right|\le \frac{1}{\tan\eta},\,\,\Re z<\beta\right\},
\qquad \eta=\sup_{n\in\mathbb{N}} |\arg\lambda_n |.
\end{equation}
If $\eta=0$ then $\Theta_{0,\beta}$ will be the half-plane $\Re z<\beta$.
\end{remark}

Now, regarding condition $(C)$, the topic of such interpolating varieties was investigated thoroughly in \cite{BerBaoVidras}.
We will recall the definitions and results of that paper in Section 2 accompanied by several examples.
We point out that in \cite[Lemmas 3.1 and 3.2]{Z2015JMAA} we found a connection between interpolating varieties for $A^0_{|z|}$
and the Krivosheev characteristic  $S_{\Lambda}$ (see \cite[Section 3]{Kriv}),
a notion very similar in nature, but more general, to the Bernstein Condensation Index
(see \cite[Definition 3.1]{2014JFA}  and \cite[Definition 1.1]{2019JPA}).
Given a multiplicity sequence $\Lambda=\{\lambda_n,\mu_n\}_{n=1}^{\infty}$,
$S_{\Lambda}$ measures in some sense how close the $\lambda_n$'s are to each other, whereas the condensation index measures
the closeness too but it is for $\Lambda$ having simple terms $\lambda_n$.
We note that recently, this index has played an important role in papers dealing with Control Theory for PDE's \cite{2014JFA,2019JPA}.

Let us present for now just two examples of $\Lambda\in ABC$, where the multiplicities $\mu_n$ can be bounded as in Example $\ref{exb}$,
or might diverge to infinity as in Example $\ref{exinf}$.

\begin{example}\label{exb}
Let $\Lambda=\{\lambda_n, k\}_{n=1}^{\infty}$ where $k$ is a positive integer and $\{\lambda_n\}_{n=1}^{\infty}$
is a sequence of distinct complex numbers satisfying the condition $(\ref{LKcondition})$.
\end{example}

\begin{example}\label{exinf}
Let $\Lambda=\{\lambda_n,\mu_n\}_{n=1}^{\infty}$ be so that
\[
\sup_{n\in\mathbb{N}}|\arg\lambda_n|<\pi/2,\qquad\inf_{n\in\mathbb{N}}\frac{|\lambda_{n+1}|}{|\lambda_n|}>1,
\qquad\text{and}\qquad \mu_n=O(|\lambda_n|^{\alpha})\quad 0\le \alpha<1.
\]
\end{example}
(e.g, $\Lambda=\{3^n,2^n\}_{n=1}^{\infty}$)

\subsection{Our Goals}

Assuming that a multiplicity sequence $\Lambda$ belongs to the $ABC$ class, our goals are as follows:

\begin{itemize}

\item Generalize Theorem $\bf A$ by characterizing the closed span of the system $E_{\Lambda}$
in the $L^p(\gamma,\beta)$ space in terms of a Fourier-type series called Taylor-Dirichlet series
(see Theorems $\ref{theorem1}$, $\ref{converse}$, and $\ref{IFF}$).

\item Generalize Theorem $\bf B$ by
proving that there exists a family of functions $r_{\Lambda}$ $(\ref{BioFamily})$
biorthogonal to the system $E_{\Lambda}$ in $L^2(\gamma,\beta)$,
such that the closed span of $r_{\Lambda}$ is not just a subspace of the closed span of $E_{\Lambda}$ in $L^2(\gamma,\beta)$,
but in fact the two closures are equal. In addition, we show that the system $E_{\Lambda}$ is hereditarily complete in this closure.
Moreover, we obtain sharp upper bounds as in $(\ref{rnkboundJMPA})$
for the norm of the elements of $r_{\Lambda}$ as well as Fourier-type series representations for them
(see Theorem $\ref{biorthogonalsystem}$).

\item
Given a sequence of non-zero complex numbers $\{d_{n,k}:\,\, n\in\mathbb{N}, \,\, k=0,1,2,\dots, \mu_n-1\}$
subject to the condition $d_{n,k}=O(e^{a\Re\lambda_n})$ for $a<\beta$,
we find a solution $f\in L^2 (\gamma, \beta)$ to the Moment Problem
\begin{equation}\label{mp1}
\langle f, e_{n,k} \rangle =\int_{\gamma}^{\beta}f(t)\cdot t^k e^{\overline{\lambda_n} t}\, dt=d_{n,k},
\qquad \forall\,\, n\in\mathbb{N}\quad\text{and}\quad k=0,1,\dots ,\mu_n-1,
\end{equation}
(see Theorem $\ref{MomentProblem}$). In fact the solution extends analytically in the sector
$\Theta_{\eta ,\beta}$ $(\ref{opensector})$.

\item
And finally, we characterize the solution space of a differential equation of infinite order on a bounded interval $(\gamma, \beta)$
(see Theorem $\ref{CarlesonTheorem}$), studied by L. Carleson and A. F. Leontev.

\end{itemize}

\subsection{The Fundamental Result}

In order to achieve our goals, we need Theorem $\ref{Distances}$, proved in Section 4,
which we call Our Fundamental Result.
Given a multiplicity sequence $\Lambda=\{\lambda_n, \mu_n\}_{n=1}^{\infty}$ in the $ABC$ class, then
for every fixed $n\in\mathbb{N}$ and $k=0,1,\dots, \mu_n-1$, we denote by $E_{\Lambda_{n,k}}$
the exponential system $E_{\Lambda}$ excluding the element $e_{n,k}(x)=x^ke^{\lambda_n x}$, that is
\[
E_{\Lambda_{n,k}}:=E_{\Lambda}\setminus e_{n,k}.
\]
We then denote by $D_{\gamma,\beta,p,n,k}$ the $\bf Distance$ between $e_{n,k}$
and the closed span of $E_{\Lambda_{n,k}}$ in $L^p(\gamma,\beta)$,
\[
D_{\gamma,\beta,p,n,k}:=\inf_{g\in \overline{\text{span}} (E_{\Lambda_{n,k}})} ||e_{n,k}-g||_{L^p (\gamma,\beta)}.
\]
We will derive the very important lower bound $(\ref{distancelowerbounds})$ for $D_{\gamma,\beta,p,n,k}$.
\begin{theorem}\label{Distances}
Let the multiplicity sequence $\Lambda=\{\lambda_n,\mu_n\}_{n=1}^{\infty}$ belong to the $ABC$ class
and consider a bounded interval $(\gamma,\beta)$.
For every $\epsilon>0$ there is a constant $u_{\epsilon}>0$, independent of $p\ge 1$, $n\in\mathbb{N}$ and $k=0,1,\dots,\mu_n-1$,
but depending on $\Lambda$ and $(\beta-\gamma)$, so that
\begin{equation}\label{distancelowerbounds}
D_{\gamma,\beta,p,n,k}\ge u_{\epsilon}e^{(\beta-\epsilon)\Re\lambda_n}.
\end{equation}
\end{theorem}

\subsection{First goal: characterizing the Closed Span of the system $E_{\Lambda}$ in $L^p(\gamma, \beta)$}

Theorem $\ref{IFF}$ generalizes Theorem $\bf A$ and is split into two parts. First we prove:
\begin{theorem}\label{theorem1}
Let the multiplicity sequence $\Lambda=\{\lambda_n,\mu_n\}_{n=1}^{\infty}$ belong to the $ABC$ class.
Then for every function $f$ belonging to the closed span of $E_{\Lambda}$ in $L^p (\gamma,\beta)$,
there exists an analytic function $g(z)$ in the open sector $\Theta_{\eta ,\beta}$ $(\ref{opensector})$,
admitting a unique Taylor-Dirichlet series representation of the form
\begin{equation}\label{uniqueTDseries}
g(z)=\sum_{n=1}^{\infty}\left(\sum_{k=0}^{\mu_n-1} c_{n,k}z^k\right)e^{\lambda_n z},
\quad c_{n,k}\in\mathbb{C},
\end{equation}
converging uniformly on compact subsets of $\Theta_{\eta , \beta}$, so that
\[
f(x)=g(x)\quad \text{almost everywhere on}\quad (\gamma,\beta).
\]
The coefficients $c_{n,k}$ satisfy the upper bound
\begin{equation}\label{cnkbound}
\forall\,\,\epsilon>0\quad \exists\,\, m_{\epsilon}>0:\quad |c_{n,k}|\le  m_{\epsilon} e^{(-\beta+\epsilon)\Re\lambda_n},\quad\forall\,\,n\in\mathbb{N}\quad\text{and}\quad \forall\,\, k=0,1,\dots,\mu_n-1.
\end{equation}
\end{theorem}
\begin{remark}
In \cite{Z2011JAT}, the above result was proved under the assumption that $\Lambda$
belonged to a certain class denoted by $U_{\eta}$, a smaller class compared to $ABC$.
\end{remark}
\begin{remark}
For $p\ge 2$, the $c_{n,k}$ coefficients are in fact equal to the inner product $\langle f, r_{n,k}\rangle$ (see $(\ref{representationf})$)
where $\{r_{n,k}\}$  is a family biorthogonal to $E_{\Lambda}$ in $L^2 (\gamma, \beta)$.
\end{remark}

Theorem $\ref{theorem1}$ is then supplemented by its converse which reads as follows.

\begin{theorem}\label{converse}
Let the multiplicity sequence $\Lambda=\{\lambda_n,\mu_n\}_{n=1}^{\infty}$ belong to the $ABC$ class.
Suppose that the series
\[
f(z)=\sum_{n=1}^{\infty}\left(\sum_{k=0}^{\mu_n-1} c_{n,k} z^k\right) e^{\lambda_n z},
\quad c_{n,k}\in\mathbb{C},
\]
is an analytic function in the sector $\Theta_{\eta,\beta}$ $(\ref{opensector})$
and $f\in L^p(\gamma,\beta)$ for some $p\ge 1$.
Then $f\in\overline{span}(E_{\Lambda})$ in $L^p(\gamma,\beta)$.
\end{theorem}

Combining our two results, both proved in Section 6, gives the following Clarkson-Erd\H{o}s-Schwartz Phenomenon
for the closed span of the system $E_{\Lambda}$ in $L^p(\gamma, \beta)$.

\begin{theorem}\label{IFF}
Let the multiplicity sequence $\Lambda=\{\lambda_n,\mu_n\}_{n=1}^{\infty}$ belong to the $ABC$ class
and let $f\in L^p (\gamma, \beta)$. Then $f$ belongs to the closed span of $E_{\Lambda}$ in $L^p (\gamma,\beta)$
if and only if there is a Taylor-Dirichlet series $g(z)$ $(\ref{uniqueTDseries})$ analytic in the open sector
$\Theta_{\eta , \beta}$ $(\ref{opensector})$,
converging uniformly on compacta, so that $f(x)=g(x)$ almost everywhere on $(\gamma,\beta)$.
\end{theorem}

\subsection{Second goal: properties of a Biorthogonal Family $r_{\Lambda}$ to $E_{\Lambda}$ in $L^2 (\gamma, \beta)$}

Our goal here is to give a very strong generalization of Theorem $\bf B$ (see Theorem $\ref{biorthogonalsystem})$.
Given $\Lambda\in ABC$, we first derive the sharp upper bound $(\ref{rnkbound})$ for the norms of the elements $r_{n,k}$
of a biorthogonal family $r_{\Lambda}$ to the exponential system $E_{\Lambda}$ in $L^2(\gamma, \beta)$,
similar to $(\ref{rnkboundJMPA})$.  We then obtain
the Fourier-type series representation $(\ref{representation})$ for these elements and a more general one
$(\ref{representationf})$ for every element in the closure of the span $E_{\Lambda}$ in $L^2(\gamma, \beta)$.
This results in showing that the closed spans of $r_{\Lambda}$ and $E_{\Lambda}$ in $L^2(\gamma, \beta)$ are equal:
that is, the system $E_{\Lambda}$ is a $\bf{Markushevich\,\, basis}$ in its closed span in $L^2(\gamma, \beta)$. In fact,
$E_{\Lambda}$ is also a $\bf{strong\,\,Markushevich\,\, basis}$ in this closure.
This means that if the set
\[
\{(n,k):\,\, n\in\mathbb{N},\,\, k=0,1,\dots,\mu_n-1\}
\]
is written as an $\bf arbitrary$ disjoint union of two sets $N_1$ and $N_2$, that is,
\[
\{(n,k):\,\, n\in\mathbb{N},\,\, k=0,1,\dots,\mu_n-1\}=N_1\cup N_2, \qquad N_1\cap N_2=\emptyset,
\]
then the closed span of the mixed system
\[
\{e_{n,k}:\,\, (n,k)\in N_1\}\cup\{r_{n,k}:\,\, (n,k)\in N_2\},\qquad (e_{n,k}=x^ke^{\lambda_n x})
\]
in $L^2(\gamma, \beta)$, coincides with the closed spans of $r_{\Lambda}$ and $E_{\Lambda}$ in $L^2(\gamma, \beta)$.
In this case we also say that the system $E_{\Lambda}$ is $\bf{hereditarily\,\, complete}$ in its closed span in $L^2(\gamma, \beta)$.
This notion is closely related to the spectral synthesis problem for linear operators \cite{Baranov2013,Baranov2015,Baranov2022}.
In these articles one also finds, amongst other results, examples of exponential systems $\{e^{i\lambda_n t}\}$,
with real $\lambda_n$, which are M-bases but not hereditarily complete in $L^2 (\gamma, \beta)$.

\begin{theorem}\label{biorthogonalsystem}
Let the multiplicity sequence $\Lambda=\{\lambda_n,\mu_n\}_{n=1}^{\infty}$ belong to the $ABC$ class.
Given a bounded interval $(\gamma, \beta)$, there exists a family of functions
\[
r_{\Lambda}=\{r_{n,k}:\,\, n\in\mathbb{N},\,\, k=0,1,\dots,\mu_n-1\}\subset L^2(\gamma, \beta)
\]
so that it is the unique biorthogonal sequence to $E_{\Lambda}$ in $L^2(\gamma,\beta)$ which belongs to the closed span
of the system $E_{\Lambda}$ in $L^2(\gamma,\beta)$, with $r_{\Lambda}$ and $E_{\Lambda}$ having the following properties:

$(I)$ For every $\epsilon>0$ there is a constant $m_{\epsilon}>0$, independent of $n\in\mathbb{N}$ and $k=0,1,\dots,\mu_n-1$, but depending on $\Lambda$ and $(\beta-\gamma)$, so that
\begin{equation}\label{rnkbound}
||r_{n,k}||_{L^2 (\gamma,\beta)} \le  m_{\epsilon}e^{(-\beta+\epsilon)\Re\lambda_n},\qquad \forall\,\, n\in\mathbb{N},\quad k=0,1,\dots,\mu_n-1.
\end{equation}

$(II)$ For $p\ge 2$ and each $f$ in the closed span of the system $E_{\Lambda}$ in $L^p (\gamma, \beta)$,
there exists an analytic function $g$ in the sector $\Theta_{\eta,\beta}$ $(\ref{opensector})$,
so that
\[
f(x)=g(x)\qquad \text{almost everywhere on}\quad (\gamma,\beta),
\]
with $g$ admitting the Fourier-type Taylor-Dirichlet series representation
\begin{equation}\label{representationf}
g(z)=\sum_{n=1}^{\infty}\left(\sum_{k=0}^{\mu_n-1} \langle f, r_{n,k} \rangle \cdot z^k\right)e^{\lambda_n z},
\end{equation}
converging uniformly on compact subsets of $\Theta_{\eta,\beta}$.

$(III)$ The closed span of $r_{\Lambda}$ in $L^2 (\gamma, \beta)$ is not merely a subspace of
the closed span of the exponential system $E_{\Lambda}$ in $L^2 (\gamma, \beta)$, but
the two are equal, that is,
\[
\overline{\text{span}}(r_{\Lambda})=\overline{\text{span}}(E_{\Lambda})\quad \text{in}\quad L^2(\gamma, \beta).
\]

$(IV)$ The exponential system $E_{\Lambda}$ is hereditary complete in its closed span in $L^2 (\gamma, \beta)$.
\end{theorem}
The proof of the above result occupies Section 7.

\begin{remark}
Each $r_{n,k}\in r_{\Lambda}$, admits the Taylor-Dirichlet series representation
\begin{equation}\label{representation}
r_{n,k}(x)=\sum_{j=1}^{\infty}\left(\sum_{l=0}^{\mu_j-1} \langle r_{n,k}, r_{j,l} \rangle \cdot x^l\right)e^{\lambda_j x},\quad
\text{almost everywhere on}\quad (\gamma, \beta)
\end{equation}
with the series extending analytically in the sector $\Theta_{\eta,\beta}$ and converging uniformly on its compact subsets.
Hence, the elements of the family
$r_{\Lambda}$  are connected to the elements of the exponential system $E_{\Lambda}$ via the Gram matrix
whose entries are the inner products $\langle r_{n,k}, r_{j,l}\rangle$ (see $\bf Appendix$, Section A).
\end{remark}

\begin{corollary}
Let the multiplicity sequence $\Lambda=\{\lambda_n,\mu_n\}_{n=1}^{\infty}$ belong to the $ABC$ class.
Let $H(\mathbb{C}, \Lambda)$ be the subspace of entire functions that admit a Taylor-Dirichlet series representation
in the complex plane. Clearly if $f\in H(\mathbb{C}, \Lambda)$, then all
its derivatives $f^{(k)}$  for $k=1,2,\dots$ belong to $H(\mathbb{C}, \Lambda)$ as well.
Fix a bounded interval $(\gamma, \beta)$, and consider the biorthogonal family $r_{\Lambda}$
to the system $E_{\Lambda}$ in $L^2 (\gamma, \beta)$.
Then for every non-negative integer $k$, one has
\[
f^{(k)}(z)=\sum_{n=1}^{\infty}\left(\sum_{k=0}^{\mu_n-1} \langle f^{(k)}, r_{n,k} \rangle \cdot z^k\right)e^{\lambda_n z},
\quad \forall\,\, z\in \mathbb{C}.
\]
\end{corollary}

\subsection{Third goal: a Moment Problem}

As a consequence of our previous results, we find an analytic solution to the Moment Problem $(\ref{mp1})$
and we do hope that researchers in control theory for PDE's will find it both interesting and useful.

\begin{theorem}\label{MomentProblem}
Let the multiplicity sequence $\Lambda=\{\lambda_n,\mu_n\}_{n=1}^{\infty}$ belong to the $ABC$ class.
Consider a bounded interval $(\gamma,\beta)$ and let
$\{d_{n,k}:\, n\in\mathbb{N},\, k=0,1,\dots,\mu_n-1\}$ be a doubly-indexed sequence of non-zero complex numbers such
that for some $a\in [-\infty,\beta)$ we have
\begin{equation}\label{interpolationrnkcomplex}
\limsup_{n\to\infty}
\frac{\log A_n}{\Re\lambda_n}=a<\beta\quad \text{where}\quad
A_n=\max\{|d_{n,k}|:\, k=0,1,\dots,\mu_n-1\}.
\end{equation}
Then there exists a unique function $f$ in the closed span of the system $E_{\Lambda}$ in $L^2(\gamma,\beta)$ so that
\begin{equation}\label{momentproblemquestion}
\langle f, e_{n,k} \rangle =\int_{\gamma}^{\beta}f(t)\cdot t^k e^{\overline{\lambda_n} t}\, dt=d_{n,k},
\qquad \forall\,\, n\in\mathbb{N}\quad \text{and}\quad k=0,1,2,\dots ,\mu_n-1.
\end{equation}
The solution $f$ extends analytically in the sector $\Theta_{\eta,\beta}$ $(\ref{opensector})$ as a Taylor-Dirichlet series
\[
\sum_{n=1}^{\infty}\left(\sum_{k=0}^{\mu_n-1} \langle f, r_{n,k} \rangle \cdot z^k\right) e^{\lambda_n z}
\]
converging uniformly on compacta. Moreover, it is the only such series in $L^2(\gamma,\beta)$ which is
a solution of $(\ref{momentproblemquestion})$.
\end{theorem}

We provide two proofs in Section 8 and in both we utilize the $r_{\Lambda}$ family of Theorem $\ref{biorthogonalsystem}$.
The first proof is classical in character in the sense that the solution is given in terms of an infinite series
\[
\sum_{n=1}^{\infty}\left(\sum_{k=0}^{\mu_n-1}d_{n,k} r_{n,k}(t)\right)
\]
with the series converging in $L^2 (\gamma, \beta)$.
The second proof uses notions from Nonharmonic Fourier Series such as Bessel sequences and Riesz-Fischer sequences.

\begin{corollary}
Let $\{\lambda_n\}_{n=1}^{\infty}$ be a sequence of positive real numbers diverging to infinity so that
$\sum_{n=1}^{\infty}1/\lambda_n <\infty$ and
$\lambda_{n+1}-\lambda_n>c>0$ for all $n\in\mathbb{N}$. Consider a positive integer $m$ and a positive real number $T$.
Let $\{d_{n,k}:\, n\in\mathbb{N},\, k=0,1,\dots,m-1\}$  be a doubly-indexed sequence of non-zero complex numbers such
that
\[
\limsup_{n\to\infty}
\frac{\log A_n}{\lambda_n}<0\quad \text{where}\quad
A_n=\max\{|d_{n,k}|:\, k=0,1,\dots,m-1\}.
\]
Then there exists a unique Taylor-Dirichlet series
\[
f(z)=\sum_{n=1}^{\infty}\left(\sum_{k=0}^{m-1}c_{n,k}z^k\right) e^{-\lambda_n z}
\]
analytic in the right half-plane $\{z:\,\, \Re z>0\}$, with $f\in L^2(0,T)$, so that
\[
\int_{0}^{T} f(t)\cdot t^k e^{-\lambda_n t}f(t)\, dt=d_{n,k},\qquad \forall\,\, n\in\mathbb{N}\quad \text{and}\quad k=0,1,2,\dots ,m-1.
\]
\end{corollary}

\subsection{Fourth goal: the solution space of a differential equation of infinite order on a bounded interval}

Our fourth and final topic deals with a differential equation of infinite order on a bounded interval $(\gamma, \beta)$ studied by
Carleson \cite {Carleson} as well as by Leont' ev \cite{Leontev}.
Let us first describe this problem
and then present our result which is a complete characterization of the solution space of the differential equation
in case $\Lambda$ belongs to the class $ABC$.

\subsubsection{A Carleson differential equation}
Suppose that a multiplicity sequence $\Lambda=\{\lambda_n, \mu_n\}_{n=1}^{\infty}$
satisfies conditions $A$ $(\ref{convergencecondition})$ and $B$ $(\ref{lessthan})$.
We associate to $\Lambda$ the entire functions of exponential type zero
\begin{equation}\label{F}
F(z)=\prod_{n=1}^{\infty}\left(1-\frac{z}{\lambda_n}\right)^{\mu_n}\quad
\text{and}
\quad
G(z)=\prod_{n=1}^{\infty}\left(1+\frac{z}{|\lambda_n|}\right)^{\mu_n}.
\end{equation}
We also consider their Taylor series expansions about zero
\[
F(z)=\sum_{n=0}^{\infty}\frac{F^{(n)}(0)}{n!}z^n\qquad\text{and}\qquad
G(z)=\sum_{n=0}^{\infty}\frac{G^{(n)}(0)}{n!}z^n.
\]

\begin{remark}\label{allpositive}
We note that $G^{(n)}(0)$ is positive for all $n\ge 0$.
\end{remark}

Carleson  \cite {Carleson} and Leont' ev \cite{Leontev} introduced the following class of functions.
\begin{definition}\label{The Class}
Let $\Lambda$=$\{\lambda_n,\mu_n\}_{n=1}^{\infty}$ be a multiplicity sequence that satisfies
conditions $A$ $(\ref{convergencecondition})$ and $B$ $(\ref{lessthan})$.
Let $(\gamma,\beta)$ be an open bounded interval on the real line.
A function $f(x)$ belongs to the class $C(\gamma, \beta, \{G_n\})$, where $G_n=\frac{G^{(n)}(0)}{n!}$,
if $f$ is infinitely differentiable on $(\gamma, \beta)$ and the series
\[
\sum_{n=0}^{\infty}\frac{G^{(n)}(0)}{n!}\cdot |f^{(n)}(x)|
\]
converges uniformly in  $(\gamma+\epsilon, \beta-\epsilon)$ for every $\epsilon>0$.
\end{definition}

For $f\in C(\gamma,\beta,\{G_n\})$,
they investigated the infinite order differential equation:
\begin{equation}\label{CarlesonLeontevEquation}
F(D)f(x)=0\qquad \forall\,\, x\in (\gamma,\beta)
\end{equation}
where
\[
D=\frac{d}{dx}\qquad \text{and}\qquad F(D)f(x):=\sum_{n=0}^{\infty}\frac{F^{(n)}(0)}{n!}\cdot f^{(n)}(x).
\]

A description of their result follows: since $F$ is an entire function of exponential type zero,
and $\sup_{n\in\mathbb{N}}|\arg\lambda_n|<\pi/2$, then there are rays $l_1$ and $l_2$ in the right half-plane $\Re z>0$,
emerging from the Origin, so that for any $\epsilon>0$ one has $|F(z)|>e^{-\epsilon |z|}$
finally for all $z$ on the two rays. A similar bound holds on a sequence
of circles $\{|z|=r_m\}_{m=1}^{\infty}$ with $r_m$ diverging to infinity.
Denote by $D_m$ the region in the complex plane bounded by the two rays and the circle $|z|=r_m$.
Carleson \cite[Theorem 4]{Carleson} and  Leont' ev \cite[Theorem 2]{Leontev}
\footnote{The Leont' ev article considers even more general results}
proved that every solution $f$ of $(\ref{CarlesonLeontevEquation})$
extends analytically in the sector $\Theta_{\eta, \beta}$ $(\ref{opensector})$ and $f$ admits $a\,\, series\,\, representation\,\,
with\,\, groupings$:
\begin{equation}\label{solutionTDS}
f(z)=\lim_{m\to\infty}f_m(z),\qquad f_m(z):=\sum_{\lambda_n\in D_m}\left(\sum_{k=0}^{\mu_n-1} c_{n,k} z^k\right)
e^{\lambda_n z}\qquad  \forall\,\, z\in \Theta_{\eta,\beta},
\end{equation}
with uniform convergence on compact subsets of the sector and with the coefficients $c_{n,k}$ derived in some special way.

\subsubsection{Our result: removing the groupings}

Assuming that $\Lambda$ belongs to the $ABC$ class, we will prove in Section 9 that
every solution of $(\ref{CarlesonLeontevEquation})$ extends analytically as a Taylor-Dirichelt series,
thus the $groupings$ in $(\ref{solutionTDS})$ can be dropped.

\begin{theorem}\label{CarlesonTheorem}
Let $\Lambda\in ABC$ class and consider an interval $(\gamma, \beta)$.

(A) Suppose that a function $f\in C(\gamma,\beta,\{G_n\})$ is a solution of the equation
$(\ref{CarlesonLeontevEquation})$. Then $f$ extends analytically in the sector $\Theta_{\eta ,\beta}$ as a Taylor-Dirichlet series
\[
f(z)=\sum_{n=1}^{\infty}\left(\sum_{k=0}^{\mu_n-1} c_{n,k} z^k\right)
e^{\lambda_n z},
\]
converging uniformly on compact subsets of the sector.

(B) Let $f$ be a Taylor-Dirichlet series as above, analytic in the sector $\Theta_{\eta ,\beta}$. Then
$f\in C(\gamma,\beta,\{G_n\})$ and $f$ is a solution of the equation
$(\ref{CarlesonLeontevEquation})$.
\end{theorem}

\begin{remark}
We will show by a counterexample, that if $\Lambda$ satisfies only conditions
$A$ $(\ref{convergencecondition})$ and $B$ $(\ref{lessthan})$,
then the groupings in $(\ref{solutionTDS})$ cannot in general be removed.
\end{remark}

\subsection{ The Main tool for the Fundamental result and Distances in $L^p(-\infty, \beta)$}

As mentioned earlier, our goals are achieved based on the Fundamental Result, Theorem $\ref{Distances}$, which is on the
lower bound $(\ref{distancelowerbounds})$ for the
distance in $L^p (\gamma, \beta)$ between an element of the system $E_{\Lambda}$ and the closed span of the remaining elements.
We note that from $(\ref{distancelowerbounds})$ it is straight forward to derive a lower bound for distances in
$L^p(-\infty, \beta)$ as well (see $(\ref{distancehalflinezero})$).
We introduce below the tool needed for Theorem $\ref{Distances}$ and make comparisons with the $classical\,\,approach$ where
a Blaschke product for a half-plane is used instead.

\subsubsection{Our Main Tool}

Our Fundamental Result is proved by employing an entire function $G$ introduced in the past by
Luxemburg and Korevaar \cite{LK}. They obtained the following result which appears also in
\cite[Theorem 3.3.3]{Sed1} and restated here as follows.

\begin{thmc}\cite[Theorem 5.2]{LK})\footnote{
If one reads carefully \cite[Sections 4 and 5]{LK},  the authors consider a sequence of complex numbers
$\{w_n\}_{n=1}^{\infty}$ satisfying $\sum_{n=1}^{\infty}1/|w_n|<\infty$.
This allows for the sequence to have terms with multiplicities involved. We note that in the statement of
\cite[Theorem 5.2]{LK} the same condition applies.
We remark that the $\lambda_n$ are chosen to be distinct, and in particular
satisfying $(\ref{LKcondition})$, only in \cite[Sections 7 and 8]{LK} where
the authors describe the closed span of the system $\{x^{\lambda_n}\}_{n=1}^{\infty}$.}

Suppose that a multiplicity sequence $\Lambda=\{\lambda_n,\mu_n\}_{n=1}^{\infty}$ satisfies condition $A$ $(\ref{convergencecondition})$.
For fixed $-\infty <\gamma<\beta<\infty$, let $\sigma=(\beta+\gamma)/2$, and $\tau=(\beta-\gamma)/2$.
Then, by $\bf {properly\,\, choosing}$ a decreasing sequence $\{\epsilon_n>0\}_{n=1}^{\infty}$
so that $\sum_{n=1}^{\infty}\epsilon_n=\tau$,
\[
G(z):=e^{-i\sigma z}\prod_{n=1}^{\infty}\left(1+\frac{z^2}{\lambda_n^2}\right)^{\mu_n}\prod_{n=1}^{\infty}
\cos(\epsilon_n z)
\]
is an entire function of exponential type, $G(x)\in L^2(\mathbb{R})\cap L^1(\mathbb{R})$ and
\[
G(z)=\frac{1}{\sqrt{2\pi}}\int_{\gamma}^{\beta}e^{-izt}h(t)\, dt,\quad
h\in C[\gamma,\beta]
\]
where $h$ is a continuous function on $[\gamma, \beta]$, vanishing outside this interval.
\end{thmc}

Assuming now that $\Lambda$ belongs to the $ABC$ class, we will extend Theorem $\bf C$ by proving in Theorem $\ref{Zikkostheorem}$ that
for every fixed $\epsilon>0$ there is a system of disjoint disks $\{P_{n,\epsilon}\}_{n=1}^{\infty}$ with respective circles $\partial P_{n,\epsilon}$,
and a constant $M_{\epsilon}$ independent of $n$ but depending on $\tau$, such that
\begin{equation}\label{lowerboundforthefunctionGz}
|G(z)|\ge M_{\epsilon} e^{(-\epsilon + \beta) \Re\lambda_n},\qquad \forall\,\, z\in\partial P_{n,\epsilon}, \quad n=1,2,\dots.
\end{equation}
Together with some auxiliary results this lower bound will eventually yield the distance bound $(\ref{distancelowerbounds})$.

\subsubsection{Distances in $L^p(-\infty,0)$}

Now, one may use $(\ref{distancelowerbounds})$ to derive the lower bound $(\ref{distancehalflinezero})$
for the distance between $e_{n,k}(x)=x^ke^{\lambda_n x}$
and the closed span of $E_{\Lambda_{n,k}}=E_{\Lambda}\setminus e_{n,k}$ in $L^p(-\infty,0)$, with
the distance denoted by
\[
D_{-\infty,0,p,n,k}:=\inf_{g\in \overline{\text{span}} (E_{\Lambda_{n,k}})} ||e_{n,k}-g||_{L^p (-\infty,0)}
\]
and
\[
L^p(-\infty,0):=\left\{f:\,\,\int_{-\infty}^{0}|f(x)|^p\, dx<\infty \right\},\qquad
||f||_{L^p (-\infty,0)}=\left(\int_{-\infty}^{0}|f(x)|^p\, dx\right)^{\frac{1}{p}},\qquad p\ge 1.
\]

Letting $g\in \overline{\text{span}} (E_{\Lambda_{n,k}})$ in $L^p (-\infty,0)$ and choosing
any real number $\gamma<0$, then clearly one has
\[
||e_{n,k}-g||_{L^p (-\infty,0)}\ge ||p_{n,k}-g||_{L^p (\gamma,0)}\ge D_{\gamma,0,p,n,k}.
\]
It then follows readily from $(\ref{distancelowerbounds})$ that
for every $\epsilon>0$, there is a positive constant $u_{\epsilon}$
independent of $p\ge 1$, $n\in\mathbb{N}$ and $k=0,1,\dots,\mu_n-1$, but depending on $\Lambda$, so that
\begin{equation}\label{distancehalflinezero}
D_{-\infty,0,p,n,k}\ge u_{\epsilon}e^{-\epsilon\Re\lambda_n}.
\end{equation}

We point out that in the $\bf Appendix$ (Section C), we derive $(\ref{distancehalflinezero})$
by another method, employing the meromorphic function
\begin{equation}\label{meromorphicfunctionfz}
f(z)=\frac{1}{(4+z)^2}\cdot \prod_{n=1}^{\infty}\left( \frac{1-z/\lambda_n}{1+z/(\overline{\lambda_n}+4)}\right)^{\mu_n},
\qquad \text{$\Lambda$\,\, belongs\,\, to\,\, the\,\, class\,\, ABC}.
\end{equation}

\begin{remark}
We find it necessary to mention this alternative tool because the authors of several papers dealing with Control Theory for PDE's
where the Method of Moments is the key for proving various results, use a similar meromorphic function (Blaschke product) as above.
Their initial goal is to derive a lower bound for distances in $L^2 (-\infty,0)$ and then based on ideas of L. Schwartz they obtain a lower bound
for distances in $L^2 (-T,0)$ for $0<T<\infty$. The interested readers may consult
the monograph by E. Zuazua \cite[Theorem 2.6.6]{Zuazua} as well as
\cite[Corollary 4.6]{2011JPA} and \cite[Lemma A.1]{2019JPA} on this matter.
However, in our opinion, using the entire function G of Theorem $\bf C$ is more effective
because we can estimate the distances in $L^p(\gamma, \beta)$
$\bf directly$ without first estimating the distances in $L^p(-\infty, 0)$. But this is our own subjective argument.
\end{remark}

\subsection{Organization of this article}

\begin{itemize}

\item
In Section 2 we recall definitions and results from \cite{BerBaoVidras} regarding interpolating varieties
for the space of entire functions of exponential type zero. We also present many examples and prove some new results
regarding these varieties. Moreover, we make a connection between interpolating varieties and the
condensation index of a sequence $\Lambda$ in case it has simple terms $\lambda_n$ (Lemma $\ref{condensation}$).

\item
In Section 3 we extend the Luxemburg-Korevaar result (Theorem $\bf C$)
by obtaining the lower bound $(\ref{lowerboundforthefunctionGz})$.

\item
Section 4 is devoted to the proof of our Fundamental Result.

\item
The region of holomorphy of Taylor-Dirichlet series is discussed in Section 5.

\item
In Section 6 we prove Theorems $\ref{theorem1}$ and $\ref{converse}$ which are on
characterizing the closed span of the exponential system $E_{\Lambda}$ in the space $L^p(\gamma,\beta)$.

\item
Section 7 is devoted in proving Theorem $\ref{biorthogonalsystem}$ which is
on the existence and properties of a biorthogonal sequence $r_{\Lambda}$
to the system $E_{\Lambda}$ in $L^2(\gamma,\beta)$.

\item
In Section 8 we present two proofs of Theorem $\ref{MomentProblem}$ dealing with the Moment Problem.

\item
In Section 9 we prove Theorem $\ref{CarlesonTheorem}$, on the solution space of
the Carleson equation.

\item
And finally in Section 10, based on our Fundamental result, an old result by G. Valiron,
and a  Bernstein-type inequality by  A. Brudnyi \cite[Theorem 1.5]{Brudnyi},
we prove the amusing and remarkable result that if $\Lambda\in ABC$
and satisfies some additional condition, then
\[
\text{the pointwise convergence of a Taylor-Dirichlet series on $(\gamma,\beta)$}
\]
yields
\[
\text{uniform convergence on closed subintervals of $(\gamma,\beta)$.}
\]
Consequently, we revisit Theorems $\ref{converse}$ and $\ref{CarlesonTheorem}$, and we obtain Theorem $\ref{CarlesonTheorem2}$.

\end{itemize}

\section{Interpolating varieties for $A^0_{|z|}$ and the $ABC$ class}
\setcounter{equation}{0}

In this section we first recall the results of \cite{BerBaoVidras} derived on interpolating varieties for the space of
entire functions of exponential type zero $A^0_{|z|}$.
We then give examples of multiplicity sequences $\Lambda=\{\lambda_n, \mu_n\}_{n=1}^{\infty}$
which are such interpolating varieties and belong to the class $ABC$ as well.
We also find a necessary gap condition for interpolation between the $\lambda_n$'s (see Lemma $\ref{zeroandgapcondition}$).
Then, for infinite products which vanish either on the set $\pm \Lambda$ or on
$\pm i \Lambda$, we derive a lower bound on a system of circles (see Lemma $\ref{anevenfunction}$).
Lastly, we state and prove Lemma $\ref{condensation}$ which is on the condensation index of a sequence.

\subsection{Interpolating varieties for the space $A^0_{|z|}$}

Following Berenstein et al. \cite{BerBaoVidras}, we say that a multiplicity sequence $\Lambda=\{\lambda_n,\mu_n\}_{n=1}^{\infty}$
is an interpolating\ variety for the space $A^0_{|z|}$ if for an arbitrary
doubly-indexed sequence $a=\{a_{n,k}:\,\, n\in \mathbb{N},\,\, k=0,1,\dots,\mu_n-1\}$, such that
\[
\forall\epsilon>0,\quad \sup_{n\in\mathbb{N}}
\sum_{k=0}^{\mu_n-1}|a_{n,k}|e^{-\epsilon |\lambda_n|}<\infty,
\]
there exists some function $f\in A^0_{|z|}$ such that
\[
\frac{f^{(k)}(\lambda_n)}{k!}=a_{n,k}.
\]

\subsubsection{Necessary and Sufficient Conditions}

The authors in \cite{BerBaoVidras}
obtained the following geometric and analytic conditions which were both necessary and sufficient
in order for $\Lambda$ to be an interpolating variety for $A^0_{|z|}$.

$The\,\, Analytic\,\, Condition$, \cite[Theorem 4.1]{BerBaoVidras}:
A multiplicity sequence $\Lambda$ is an interpolating variety for the space $A_{|z|} ^0$
if and only if there exists an entire function $f\in A^0_{|z|}$ such that $\Lambda$ is a subset of the zero set of $f$, and
for every $\epsilon>0$ there is a positive constant $u_{\epsilon}$, independent of $n\in\mathbb{N}$,  such that
\[
\frac{|f^{(\mu_n)}(\lambda_n)|}{\mu_n!}\ge u_{\epsilon}e^{-\epsilon |\lambda_n|}\qquad\forall\,\, n\in\mathbb{N}.
\]

$The\,\, Geometric\,\, Conditions$, \cite[Theorem 3.1]{BerBaoVidras}:
Consider a multiplicity sequence $\Lambda$ and
the counting functions of $\Lambda$ about $0$ and a point $z_0$, given respectively by
\[
n_{\Lambda}(t):=\sum_{|\lambda_n|\le t}\mu_n\qquad\text{and}\qquad
n_{\Lambda}(t,z_0):=\sum_{|\lambda_n-z_0|\le t}\mu_n.
\]
Let
\[
N(r,\Lambda):=\int_{0}^{r}\frac{n_{\Lambda}(t)-n_{\Lambda}(0)}{t}\, dt+n_{\Lambda}(0)\log r,
\]
and
\[
N(r,z_0,\Lambda):=\int_{0}^{r}\frac{n_{\Lambda}(t,z_0)-n_{\Lambda}(0,z_0)}{t}\, dt+n_{\Lambda}(0,z_0)\log r.
\]
Then $\Lambda$ is an interpolating variety for the space $A_{|z|} ^0$ if and only if
\[
(I)\,\,\, N(r,\Lambda)=o(r)\,\, \text{as}\,\, r\to\infty
\qquad\text{and}\qquad
(II)\,\,\, N(|\lambda_n|,\lambda_n,\Lambda)=o(|\lambda_n|)\,\, \text{as}\,\, n\to\infty.
\]

\subsubsection{A necessary condition for Interpolation}

Observe that
\begin{eqnarray}
N(|\lambda_n|,\lambda_n,\Lambda) & =& \int_{0}^{|\lambda_n|}\frac{n_{\Lambda}(t,\lambda_n)-n_{\Lambda}(0,\lambda_n)}{t}\, dt +  n_{\Lambda}(0,\lambda_n)\log |\lambda_n|\nonumber\\
& = &
\sum_{0<|\lambda_n-\lambda_k|\le|\lambda_n|}\mu_k\log \left|\frac{\lambda_n}{\lambda_n-\lambda_k}\right|+ \mu_n\log |\lambda_n|\label{relation}\label{boundmun}.
\end{eqnarray}
In particular
\begin{equation}\label{mun=1}
N(|\lambda_n|,\lambda_n,\Lambda)=
\sum_{0<|\lambda_n-\lambda_k|\le|\lambda_n|}\log \left|\frac{\lambda_n}{\lambda_n-\lambda_k}\right|+\log |\lambda_n|\qquad \text{if}\qquad \mu_n=1\quad \forall\,\, n\in\mathbb{N}.
\end{equation}
It follows from $(\ref{boundmun})$ and $Geometric\,\, Condition\,\, (II)$, that if $\Lambda$ is an interpolating variety, then
\begin{equation}\label{conditionn}
\frac{\mu_n\log |\lambda_n|}{|\lambda_n|}\to 0\qquad \text{as}\qquad n\to\infty.
\end{equation}

\begin{remark}
Therefore $(\ref{conditionn})$ is Necessary for Interpolation.
\end{remark}

\subsubsection{A particular sufficient condition for Interpolation}

The following result is a sufficient condition for interpolation
and provides us with a large class of examples of such varieties (see $\bf {Example\,\, v}$ later on).

\begin{lema}\cite[Remark 2.3 and Lemma 3.2]{Z2015JMAA}

Suppose that a multiplicity sequence $\Lambda=\{\lambda_n,\mu_n\}_{n=1}^{\infty}$ satisfies the following three condtions:

$(1)$ $\Lambda$ has density zero, that is $\frac{\sum_{|\lambda_n|\le t}\mu_n}{t}\to 0$ as $t\to\infty$.

$(2)$ Relation $(\ref{conditionn})$ holds,

$(3)$ There is some $\delta\in (0,1/10)$ such that
for every fixed $k\in \mathbb{N}$ the inequality $|\lambda_n-\lambda_k|\le \delta |\lambda_k|$ is true only for $n=k$.

Then $\Lambda$ is an interpolating variety for the space $A_{|z|}^0$.
\end{lema}

\subsection{Examples of $\Lambda$ in the $ABC$ class}

Next, we present examples where some $\Lambda$'s are interpolating varieties for the space $A_{|z|} ^0$
while some others are not. In all the examples given,
conditions $A$ $(\ref{convergencecondition})$ and $B$ $(\ref{lessthan})$ hold.
Therefore, those $\Lambda$'s which are interpolating varieties, belong to the $ABC$ class as well.
We point out however that conditions $A$ $(\ref{convergencecondition})$ and $B$ $(\ref{lessthan})$
are neither sufficient nor necessary for interpolation. \\

\smallskip

$\bf Case \,\, 1$, $\mu_n=1$ for all $n\in\mathbb{N}$:

$(\bf Example\,\, i):$ It follows from \cite{Vidras} (see also \cite[Theorem 5.1]{BerBaoVidras})
that if a sequence $\{\lambda_n\}_{n=1}^{\infty}$ satisfies condition $(\ref{LKcondition})$
then $\Lambda=\{\lambda_n,1\}_{n=1}^{\infty}$  is an interpolating  variety. In particular,
this is the case when $\{\lambda_n\}_{n=1}^{\infty}$ is a sequence
of distinct positive real numbers such that $\sum_{n=1}^{\infty}1/\lambda_n<\infty$
with uniformly separated terms, that is
\[
\liminf_{n\in\mathbb{N}} (\lambda_{n+1}-\lambda_n)>0.
\]

\begin{remark}
But what happens if $\liminf_{n\in\mathbb{N}} (\lambda_{n+1}-\lambda_n)=0$? In $\bf Examples\,\, ii$ and $\bf iii$,
$\Lambda$ might or might not be an interpolating variety.
\end{remark}

$(\bf Example\,\, ii):$  From $(\ref{mun=1})$ and the Geometric Conditions $(I)$ and $(II)$,
one can show that if
\[
\lambda_{2n-1}=n^2\quad \text{and}\quad \lambda_{2n}=n^2+e^{-n},
\]
then $\Lambda=\{\lambda_n\}_{n=1}^{\infty}$ is an interpolating variety.

$(\bf Example\,\, iii):$ On the other hand, if
\[
\lambda_{2n-1}=n^2\quad \text{and}\quad \lambda_{2n}=n^2+e^{-n^2},
\]
it then follows from $(\ref{mun=1})$ and Geometric Condition $(II)$,
that $\Lambda=\{\lambda_n\}_{n=1}^{\infty}$ is not an interpolating variety.\\

$\bf Case \,\, 2$,  $\mu_n=O(1)$:

If a sequence $\{\lambda_n\}_{n=1}^{\infty}$ with distinct terms is an interpolating variety, it then follows from the Geometric Conditions
and $(\ref{boundmun})$ that the multiplicity sequence
$\Lambda=\{\lambda_n,\mu_n\}_{n=1}^{\infty}$  with $\sup_{n\in\mathbb{N}}\mu_n<\infty$ is an interpolating variety as well.

$(\bf Example\,\, iv):$ Let a sequence $\{\lambda_n\}_{n=1}^{\infty}$ satisfy the condition $(\ref{LKcondition})$.
Then $\Lambda=\{\lambda_n,\mu_n\}_{n=1}^{\infty}$ with $\sup_{n\in\mathbb{N}}\mu_n<\infty$ is an interpolating variety.\\

$\bf Case \,\, 3$, $\sup_{n\in\mathbb{N}}\mu_n=\infty$:

$(\bf Example\,\, v):$  Let $\Lambda=\{\lambda_n,\mu_n\}_{n=1}^{\infty}$ be so that
\[
\sup_{n\in\mathbb{N}}|\arg\lambda_n|\le\eta<\pi/2,\qquad\inf_{n\in\mathbb{N}}\frac{|\lambda_{n+1}|}{|\lambda_n|}\ge q>1\qquad
\text{and}\qquad \mu_n=O(|\lambda_n|^{\alpha})\quad 0\le \alpha<1.
\]
Then $\Lambda$ is an interpolating variety. For example, let
$\Lambda=\{(k+1)^n, k^n\}_{n=1}^{\infty}$ where $k\ge 2$ is a positive integer, such as
$\Lambda=\{3^n,2^n\}_{n=1}^{\infty}$, $\Lambda=\{4^n,3^n\}_{n=1}^{\infty}$ e.t.c.

In order to justify this, we will show that Conditions $(1)$, $(2)$ and $(3)$ of Lemma $\bf A$ hold.

First observe that Condition $(2)$ is obvious.

Next we show that Condition $(1)$ holds, that is $\Lambda$ has density zero.
From above we get $|\lambda_n|\ge q^{n-1}|\lambda_1|$ and there is $M>0$ so that $\mu_n\le M |\lambda_n|^{\alpha}$
for all $n\in\mathbb{N}$ . Thus
\[
\sum_{n=1}^{\infty}\frac{\mu_n}{|\lambda_n|}\le \sum_{n=1}^{\infty}\frac{M}{|\lambda_n|^{1-\alpha}}\le
\sum_{n=1}^{\infty}\frac{M}{(q^{n-1})^{1-\alpha}|\lambda_1|^{1-\alpha}}=\frac{M}{|\lambda_1|^{1-\alpha}}\sum_{n=1}^{\infty}\frac{1}{(q^{1-\alpha})^{n-1}}<\infty
\]
since $q>1$ and $1-\alpha>0$. The convergence implies that $\Lambda$ has density zero.

Finally we have to show that Condition $(3)$ is true also.  For every fixed $k\in\mathbb{N}$ we will obtain a lower bound
for $|\lambda_{n}-\lambda_k|$ when $n\not= k$. We keep in mind the lacunary relation $|\lambda_{n+1}|\ge q|\lambda_n|$ and
we consider two cases:

$(a)$ $|\lambda_{k+j}-\lambda_k|$ for all $j\in\mathbb{N}$

$(b)$ $|\lambda_{k-j}-\lambda_k|$ for all $j\in\{1,2,\dots,k-1\}$.

For $(a)$ we get
\[
|\lambda_{k+j}-\lambda_k|\ge |\lambda_{k+j}|-|\lambda_k|\ge (q^j-1)|\lambda_k|\ge (q-1)|\lambda_k|
\]
and for $(b)$ we get
\[
|\lambda_{k-j}-\lambda_k|\ge |\lambda_k|-|\lambda_{k-j}|\ge |\lambda_k|-\frac{|\lambda_k|}{q^j}= \left(1-\frac{1}{q^j}\right)|\lambda_k|\ge
\frac{q-1}{q}|\lambda_k|.
\]
Using the above lower bounds, we see that if we choose $0<\delta<\min\left\{\frac{q-1}{q}, \frac{1}{10}\right\}$, then
for every fixed $k\in \mathbb{N}$ the inequality $|\lambda_n-\lambda_k|\le \delta |\lambda_k|$ is true only for $n=k$.
Thus Condition (3) of the Lemma $\bf A$ is satisfied.

$(\bf Example\,\, vi):$ Let
$\Lambda=\{\lambda_n,\mu_n\}_{n=1}^{\infty}$ be so that
\[
\lambda_n=n^2\cdot 10^n\quad\text{and}\quad \mu_n=10^n.
\]
and let $\Lambda^{\prime}=\{\lambda^{\prime}_n,\mu^{\prime}_n\}_{n=1}^{\infty}$ be so that
\[
\lambda^{\prime}_n=n^2\cdot 10^{n^2}\quad\text{and}\quad \mu^{\prime}_n=10^{n^2}.
\]
It is interesting to note that $\sum_{n=1}^{\infty}\mu_n/\lambda_n=\sum_{n=1}^{\infty}\mu^{\prime}_n/\lambda^{\prime}_n.$

Now, $\Lambda$ is an interpolating variety (see \cite[Example 3.2]{Z2015JMAA}).
Observe that in contrast to  $(\bf Example\,\, v)$, this time
\[
\mu_n\not=O(\lambda_n^{\alpha})\quad \text{for\,\, any}\quad 0\le \alpha <1.
\]
On the other hand, $\Lambda^{\prime}$ is not interpolating variety since $(\ref{conditionn})$  does not hold.

\subsection{A necessary gap condition for interpolation}

In $\bf{Examples}$ $\bf ii$ and $\bf iii$, we saw that if $\liminf_{n\in\mathbb{N}} |\lambda_{n+1}-\lambda_n|=0$, then
$\Lambda$ might or might not be an interpolating variety. Now, if $\Lambda$ is such a variety,
how close can these frequencies be to each other?
In the following result we provide a lower bound for the distance between them.

\begin{lemma}\label{zeroandgapcondition}
Let the multiplicity sequence $\Lambda=\{\lambda_n,\mu_n\}_{n=1}^{\infty}$ be an interpolating variety for the space $A^0_{|z|}$.
Then, for every $\epsilon>0$  there is a positive constant $m_{\epsilon}$, independent of $n\in\mathbb{N}$, so that
\begin{equation}\label{separationcondition}
\forall\,\, n\in\mathbb{N},\qquad \inf_{k\not= n}|\lambda_n-\lambda_k|\ge m_{\epsilon}\exp\left\{{-\frac{\epsilon |\lambda_n|}{\mu_n}}\right\}.
\end{equation}
\end{lemma}
\begin{proof}
Suppose that the gap relation $(\ref{separationcondition})$ is false.
Then there is some positive constant $\rho$ and a subsequence
$\{\lambda_{n_j}\}_{j=1}^{\infty}$  such that for each $\lambda_{n_j}$ there is some $\lambda_{m(j)}$ so that
\[
|\lambda_{n_j}-\lambda_{m(j)}|\le \exp\left\{{-\frac{\rho |\lambda_{n_j}|}{\mu_{n_j}}}\right\}<1.
\]
Equivalently,
\begin{equation}\label{separation}
-\mu_{n_j}\log |\lambda_{n_j}-\lambda_{m(j)}|\ge \rho |\lambda_{n_j}|.
\end{equation}
On the other hand, from $(\ref{boundmun})$ and $Geometric\,\, Condition\,\, (II)$,
we have
\[
\frac{1}{|\lambda_{m(j)}|}
\sum_{0<|\lambda_{m(j)}-\lambda_k|\le|\lambda_{m(j)}|}\mu_k\log \left|\frac{\lambda_{m(j)}}{\lambda_{m(j)}-\lambda_k}\right|\to 0,\qquad   j\to\infty.
\]
Hence for the same $\rho$ as above, there is $j(\rho)\in\mathbb{N}$ so that
\[
\sum_{0<|\lambda_{m(j)}-\lambda_k|\le|\lambda_{m(j)}|}\mu_k\log \left|\frac{\lambda_{m(j)}}{\lambda_{m(j)}-\lambda_k}\right|<
\frac{\rho}{2}|\lambda_{m(j)}|,\qquad \forall\,\, j\ge j(\rho).
\]
Since $0<|\lambda_{m(j)}-\lambda_{n_j}|<1<|\lambda_{m(j)}|$, it follows from the above relation that
\[
\mu_{n_j}\log \left|\frac{\lambda_{m(j)}}{\lambda_{m(j)}-\lambda_{n_j}}\right|<\frac{\rho}{2}|\lambda_{m(j)}|\quad \forall\,\, j\ge j(\rho).
\]
We rewrite this as
\[
-\mu_{n_j}\log |\lambda_{m(j)}-\lambda_{n_j}|+\mu_{n_j}\log |\lambda_{m(j)}|<\frac{\rho}{2}|\lambda_{m(j)}|\quad \forall\,\, j\ge j(\rho).
\]
Combining with $(\ref{separation})$  gives
\[
\rho |\lambda_{n_j}|\le -\mu_{n_j}\log |\lambda_{n_j}-\lambda_{m(j)}|< \frac{\rho}{2}|\lambda_{m(j)}|<\frac{\rho}{2}(|\lambda_{n_j}|+1).
\]
But then one has
\[
\frac{\rho}{2} |\lambda_{n_j}|\le \frac{\rho}{2},
\]
which is false since $|\lambda_{n_j}|\to \infty$ as $j\to\infty$.
Hence the gap relation $(\ref{separationcondition})$ is indeed true.
\end{proof}

Therefore we can state the following.
\begin{remark}\label{disjointdisks}
Let the multiplicity sequence $\Lambda=\{\lambda_n,\mu_n\}_{n=1}^{\infty}$ be an interpolating variety for the space $A_{|z|}^0$.
Then for every $\epsilon>0$
there exists a positive constant $m_{\epsilon}$, independent of $n\in\mathbb{N}$, so that each one of the three sets
\[
\bigcup_{n=1}^{\infty}D_{n,\epsilon},\qquad \bigcup_{n=1}^{\infty}C_{n,\epsilon},\qquad\text{and}\qquad \bigcup_{n=1}^{\infty}P_{n,\epsilon},
\]
is a union of disjoint open disks, where  for all $n\in\mathbb{N}$ we have
\[
D_{n,\epsilon}:=\left\{z:\,\,|z-\lambda_n|< \frac{m_{\epsilon}}{2} \exp\left\{{-\frac{\epsilon |\lambda_n|}{\mu_n}}\right\}\right\},
\]
\[
C_{n,\epsilon}:=\left\{z:\,\,|z-\lambda_n|< \frac{m_{\epsilon}}{6} \exp\left\{{-\frac{\epsilon |\lambda_n|}{\mu_n}}\right\}\right\},
\]
\[
P_{n,\epsilon}:=\left\{z:\,\,|z-i\lambda_n|< \frac{m_{\epsilon}}{6} \exp\left\{{-\frac{\epsilon |\lambda_n|}{\mu_n}}\right\}\right\}.
\]
Observe that $C_{n,\epsilon}\subset D_{n,\epsilon}$ and the ratio of their radii is 1:3.
\end{remark}

\subsection{Infinite products vanishing on $\pm\Lambda$ or on $\pm i\Lambda$}

For every fixed $\epsilon>0$, let $\partial D_{n,\epsilon}$, $\partial C_{n,\epsilon}$, and $\partial P_{n,\epsilon}$
be their respective circles of the disks $D_{n,\epsilon}$, $C_{n,\epsilon}$, and $P_{n,\epsilon}$ for $n=1,2,\dots$.
In Lemma $\ref{anevenfunction}$ we will derive sharp lower bounds on
the circles $\partial C_{n,\epsilon}$ and $\partial P_{n,\epsilon}$, for the modulus of infinite products
vanishing on $\pm\Lambda$ or on $\pm i\Lambda$.
To do that, we will need the following Carath\'{e}odory inequality
(see \cite[page 19 Theorem 9]{Levin} and  \cite[page 3 Theorem 1.3.2]{Boas}).
\begin{thmd}
If the function $f$ is holomorphic in the disk $|z|\le R$ and
has no zeros in this disk, and if $f(0)=1$, then its modulus in the disk $|z|\le r<R$
satisfies the inequality
\[
\log |f(z)|\ge -\frac{2r}{R-r}\log \max_{|z|=R}|f(z)|.
\]
\end{thmd}
In particular, if $r=R/3$  then
\[
\forall\,\, z:|z|\le R/3,\quad \log |f(z)|\ge -\log \max_{|z|=R}|f(z)|.
\]

\begin{lemma}\label{anevenfunction}
Let the multiplicity sequence $\Lambda=\{\lambda_n,\mu_n\}_{n=1}^{\infty}$ be an interpolating variety for the space $A^0_{|z|}$.
Suppose that $\sup_{n\in\mathbb{N}}|\arg\lambda_n|<\pi/2$.
Consider the entire functions of exponential type zero
\[
F(z)=\prod_{n=1}^{\infty}\left(1-\frac{z^2}{\lambda_n^2}\right)^{\mu_n}\quad\text{and}\quad
L(z)=\prod_{n=1}^{\infty}\left(1+\frac{z^2}{\lambda_n^2}\right)^{\mu_n}.
\]
For fixed $\epsilon>0$, consider the disks and circles of Remark $\ref{disjointdisks}$.
Then there are positive constants $m_{\epsilon ,1}$ and $m_{\epsilon ,2}$, independent of $n\in\mathbb{N}$, so that

\begin{equation}\label{lowerboundforFderivative}
\frac{|F^{(\mu_n)}(\lambda_n)|}{\mu_n !}\ge m_{\epsilon ,1}e^{-\epsilon |\lambda_n|}\qquad \forall\,\, n\in\mathbb{N},
\end{equation}

\begin{equation}\label{lowerboundforF}
|F(z)|\ge m_{\epsilon ,2}e^{-\epsilon |\lambda_n|}\qquad \forall\,\, z\in \partial C_{n,\epsilon},\quad n=1,2,\dots,
\end{equation}

\begin{equation}\label{lowerboundforL}
|L(z)|\ge m_{\epsilon ,2}e^{-\epsilon |\lambda_n|}\qquad \forall\,\, z\in \partial P_{n,\epsilon},\quad n=1,2,\dots.
\end{equation}
\end{lemma}

\begin{proof}

First we prove $(\ref{lowerboundforFderivative})$: since $\Lambda$ is such a variety then it satisfies
$The\,\, Geometric\,\, Conditions$ $(I)$ and $(II)$. If we let
$\Lambda^{\prime}:=\{\lambda_n,\mu_n\}_{n=1}^{\infty}\cup\{-\lambda_n,\mu_n\}_{n=1}^{\infty}$ and since
$\sup_{n\in\mathbb{N}}|\arg\lambda_n|<\pi/2$, then $\Lambda^{\prime}$ is also an interpolating
variety for the space $A^0_{|z|}$. Hence by $The\,\, Analytic\,\, Condition$ there exists an entire function $G\in A_{|z|}^0$ vanishing on $\Lambda^{\prime}$
so that for every $\epsilon>0$ there is a positive constant $u_{\epsilon}$, independent of $n\in\mathbb{N}$, such that
\begin{equation}\label{lowerforG}
\frac{|G^{(\mu_n)}(\lambda_n)|}{\mu_n !}\ge u_{\epsilon}e^{-\epsilon |\lambda_n|}\qquad \forall\,\, n\in\mathbb{N}.
\end{equation}

Let $\Lambda^{\prime\prime}=\{\lambda_n ^{\prime\prime},\mu_n ^{\prime\prime}\}_{n=1}^{\infty}$ be the zero set of $G$ and write
$\Lambda^{\prime\prime}=\Lambda^{\prime}\cup \Omega$
where $\Omega:=\{w_n,k_n\}_{n=1}^{\infty}$ is the set of the rest of the non-zero zeros of $G$, if there are any.
Suppose also that $G$ vanishes at $z=0$ exactly $m$ times, for some $m\ge 0$.
Now, since $G\in A_{|z|}^0$  then its zero set has density zero, that is
$\lim_{t\to\infty}\frac{\sum_{|\lambda_n^{\prime\prime}|\le t}\mu_n^{\prime\prime}}{t}=0$.
Moreover, by the properties of entire functions of exponential type zero
(see \cite[Definition 2.5.4 and Theorem 2.10.3]{Boas}, either

$(I)$  $G$ has order less than 1, which means that

\[
\sum_{n=1}^{\infty}\frac{\mu_n ^{\prime\prime}}{|\lambda_n ^{\prime\prime}|}<\infty,\quad\text{hence}\quad
\sum_{n=1}^{\infty}\frac{k_n}{|w_n|}<\infty\quad\text{also},
\]

or

$(II)$ $G$ has order 1 and type 0, which means that

\[
(a)\quad \sum_{n=1}^{\infty}\frac{\mu_n ^{\prime\prime}}{|\lambda_n ^{\prime\prime}|}=\infty,
\]
\[
(b)\quad \sum_{n=1}^{\infty}\left(\frac{\mu_n ^{\prime\prime}}{|\lambda_n ^{\prime\prime}|}\right)^{1+\epsilon}<\infty\quad  \forall\,\, \epsilon>0,\quad\text{and}
\]
\[
(c)\quad \sum_{n=1}^{\infty}\frac{\mu_n ^{\prime\prime}}{\lambda_n ^{\prime\prime}}=\alpha\in\mathbb{C}\quad\text{hence}\quad
\sum_{n=1}^{\infty}\frac{k_n}{w_n}=\alpha\quad\text{since}\quad \Lambda^{\prime}\quad \text{is\,\, even}.
\]

By the Hadamard Factorization Theorem it then follows that
\begin{equation}\label{HadFac}
G(z)=cF(z)W(z),\qquad c\in \mathbb{C}\setminus\{0\}
\end{equation}
where if $(I)$ is true then
\[
W(z)=z^m\prod_{n=1}^{\infty}\left(1-\frac{z}{w_n}\right)^{k_n},
\]
and if $(II)$ holds then
\[
W(z)=z^me^{-\alpha z}\prod_{n=1}^{\infty}\left(1-\frac{z}{w_n}\right)^{k_n}e^{zk_n/w_n}.
\]
In both cases, $W\in A_{|z|}^0$ hence for every $\epsilon>0$ there is some $U_{\epsilon}>0$ so that
\begin{equation}\label{upperforW}
|W(z)|\le U_{\epsilon}e^{\epsilon |z|}\qquad \forall\,\, z\in\mathbb{C},\qquad \text{thus}\qquad
|W(\lambda_n)|\le U_{\epsilon}e^{\epsilon |\lambda_n|}\qquad \forall\,\, n\in\mathbb{N}.
\end{equation}
Now, from $(\ref{HadFac})$ we get
\[
\frac{|F^{(\mu_n)}(\lambda_n)|}{\mu_n !}=\frac{|G^{(\mu_n)}(\lambda_n)|}{\mu_n !}\cdot\frac{1}{|cW(\lambda_n)|}.
\]
Then combining $(\ref{upperforW})$ with $(\ref{lowerforG})$ yields  $(\ref{lowerboundforFderivative})$.\\

\smallskip

Next we prove $(\ref{lowerboundforF})$: for each $n\in\mathbb{N}$ define
\begin{equation}\label{Fn}
F_n(z):=\frac{F(z)}{(z-\lambda_n)^{\mu_n}}\cdot\frac{\mu_n!}{F^{(\mu_n)}(\lambda_n)}.
\end{equation}
We then get
\[
F_n(\lambda_n)=1.
\]
For the fixed $\epsilon>0$, consider the two disks $D_{n,\epsilon}$ and  $C_{n,\epsilon}$ as in Remark $\ref{disjointdisks}$
and observe that $F_n$ has no zeros on these two disks and neither on their respective circles
$\partial D_{n,\epsilon},\,\, \partial C_{n,\epsilon}$.
Since the radius of $C_{n,\epsilon}$ is $1/3$ of the one of $D_{n,\epsilon}$
and $F_n(\lambda_n)=1$, from Theorem $\bf D$ we get
\begin{equation}\label{twodisks}
\forall\,\, z\in \overline{C_{n,\epsilon}},\quad \log |F_n(z)|\ge -\log \max_{z\in \partial D_{n,\epsilon}}|F_n(z)|.
\end{equation}
So let us now estimate $F_n$ from above on the circle $\partial D_{n,\epsilon}$.
Since $F\in A_{|z|}^0$ then for the above fixed $\epsilon>0$ there is some $t_{\epsilon}>0$ so that
$|F(z)|\le t_{\epsilon}e^{\epsilon |z|}$ for all $z\in\mathbb{C}$.
Then, based on the size of the radius of $D_{n,\epsilon}$ and relations $(\ref{lowerboundforFderivative})$ and $(\ref{Fn})$,
we get
\[
\forall\,\, z\in \partial D_{n,\epsilon},\quad
|F_n(z)|\le  t_{\epsilon}e^{\epsilon|\lambda_n|}\cdot\left(\frac{2}{m_{\epsilon}}e^{\epsilon |\lambda_n|/\mu_n}\right)^{\mu_n}\cdot\frac{e^{\epsilon |\lambda_n|}}{m_{\epsilon ,1}}
= t_{\epsilon}e^{\epsilon|\lambda_n|}\cdot\left(\frac{2}{m_{\epsilon}}\right)^{\mu_n}e^{\epsilon |\lambda_n|}\cdot \frac{e^{\epsilon |\lambda_n|}}{m_{\epsilon ,1}}.
\]
Combining this with the relation $\mu_n/\lambda_n\to 0$ as $n\to\infty$, shows that
there is some $r_{\epsilon}>0$, independent of $n\in\mathbb{N}$,  so that
\[
\max_{z\in \partial D_{n,\epsilon}}|F_n(z)|\le r_{\epsilon}e^{4\epsilon |\lambda_n|}.
\]
Therefore from $(\ref{twodisks})$ we get
\begin{equation}\label{newneweq}
\min_{z\in \partial C_{n,\epsilon}}|F_n(z)|\ge \frac{e^{-4\epsilon |\lambda_n|}}{r_{\epsilon}}.
\end{equation}
Then rewrite $(\ref{Fn})$ as
\[
F(z)=F_n(z)\cdot (z-\lambda_n)^{\mu_n}\cdot \frac{F^{(\mu_n)}(\lambda_n)}{\mu_n!}.
\]
Combining relations $(\ref{newneweq})$ and $(\ref{lowerboundforFderivative})$ with
the length of the radius  of $C_{n,\epsilon}$, yield relation $(\ref{lowerboundforF})$.

Finally, by rotation we get $(\ref{lowerboundforL})$ for the function $L(z)$ and our proof is now complete.
\end{proof}

\subsection{On the condensation index of a sequence $\Lambda$}

We now want to add a small remark which connects the topic of $interpolating\,\, varieties$
for the space $A^0 _{|z|}$, with the condensation index $c(\Lambda)$ of
a sequence $\Lambda=\{\lambda_n\}_{n=1}^{\infty}$.

Suppose that  $\Lambda=\{\lambda_n\}_{n=1}^{\infty}$ has distinct non-zero complex numbers that satisfies
conditions $A$ $(\ref{convergencecondition})$ and $B$ $(\ref{lessthan})$.
The condensation index $c(\Lambda)$ is defined as

\begin{equation}\label{cond}
c(\Lambda):=\limsup_{n\to\infty}\frac{-\log |F^{'}(\lambda_n)|}{|\lambda_n|},\qquad F(z)=
\prod_{n=1}^{\infty}\left(1-\frac{z^2}{\lambda_n^2}\right).
\end{equation}

We prove below the following.
\begin{lemma}\label{condensation}
Consider a sequence $\Lambda=\{\lambda_n\}_{n=1}^{\infty}$  of distinct non-zero complex numbers such that
$\sum_{n=1}^{\infty}1/|\lambda_n|<\infty$ and $\sup_{n\in\mathbb{N}}|\arg\lambda_n|<\pi/2$.
Then its condensation index $c(\Lambda)$ is equal to zero if and only if $\Lambda$ is an interpolating variety for the space $A^0 _{|z|}$.
\end{lemma}

\begin{proof}
Clearly the function $F(z)$ in $(\ref{cond})$ as well as its derivative function $F^{'}(z)$ are entire functions of exponential type zero.
Therefore for every $\epsilon>0$ there is a positive constant
$m_{\epsilon}$ so that
\begin{equation}\label{condensationabove}
|F^{'}(\lambda_n)|<m_{\epsilon}e^{\epsilon |\lambda_n|}, \qquad \forall\,\, n\in\mathbb{N}.
\end{equation}

Now, if such a sequence $\Lambda$ is an interpolating variety for the space $A^0 _{|z|}$, it follows from Lemma $\ref{anevenfunction}$ that
for every $\epsilon>0$ there is a positive constant $u_{\epsilon}$ so that
\begin{equation}\label{condensationbelow}
|F^{'}(\lambda_n)|>u_{\epsilon}e^{-\epsilon |\lambda_n|}, \qquad \forall\,\, n\in\mathbb{N}.
\end{equation}
Combining $(\ref{condensationabove})$ with $(\ref{condensationbelow})$ shows that $c(\Lambda)=0$.

On the other hand, if $c(\Lambda)=0$ then $(\ref{condensationbelow})$ holds, hence by
$The\,\, Analytic\,\, Condition$, $\Lambda$ is an interpolating variety for the space $A^0 _{|z|}$.
\end{proof}

\section{A lower bound for the Luxemburg-Korevaar function $G$}
\setcounter{equation}{0}

We are now ready to extend Theorem $\bf C$ by deriving the lower bound $(\ref{lowerboundforthefunctionGz})$.

\begin{theorem}\label{Zikkostheorem}
Let the multiplicity sequence $\Lambda=\{\lambda_n,\mu_n\}_{n=1}^{\infty}$ belong to the $ABC$ class and
let $-\infty <\gamma<\beta<\infty$. Let also
\[
\sigma=\frac{(\beta+\gamma)}{2},\qquad \tau=\frac{(\beta-\gamma)}{2},\qquad \text{hence}\qquad \beta=\tau+\sigma.
\]
Then, by properly choosing a decreasing sequence $\{\epsilon_n>0\}_{n=1}^{\infty}$
so that $\sum_{n=1}^{\infty}\epsilon_n=\tau$, the entire function
\[
G(z)=e^{-i\sigma z}\prod_{n=1}^{\infty}\left(1+\frac{z^2}{\lambda_n^2}\right)^{\mu_n}\prod_{n=1}^{\infty}
\cos(\epsilon_n z)
\]
belongs to $L^2(\mathbb{R})\cap L^1(\mathbb{R})$ such that
\begin{equation}\label{GzZikkos}
G(z)=\frac{1}{\sqrt{2\pi}}\int_{\gamma}^{\beta}e^{-izt}g(t)\, dt
\end{equation}
for some $g\in C[\gamma,\beta]$ with $g$ vanishing outside the interval $[\gamma,\beta]$.
Moreover, for every fixed $\epsilon>0$, let $P_{n,\epsilon}$ for $n=1,2,\dots$
be the disks as in  Remark $\ref{disjointdisks}$ and let $\partial P_{n,\epsilon}$ be their respective circles.
Then there is a positive constant $M_{\epsilon}$ independent of $n$ but depending on $\tau$, such that
relation $(\ref{lowerboundforthefunctionGz})$ is true.
\end{theorem}

\begin{proof}

First write
\[
G(z)=e^{-i\sigma z}\cdot H(z)\cdot L(z)
\]
where
\[
H(z):=\prod_{n=1}^{\infty}\cos (\epsilon_n z)\qquad\text{and}\qquad
L(z):=\prod_{n=1}^{\infty}\left(1+\frac{z^2}{\lambda_n^2}\right)^{\mu_n}.
\]

For $0\le \eta<\pi/2$ and $\chi >0$, consider the region
\begin{equation}\label{omegaeta}
\Omega_{\eta}:=\left\{z:\frac{\pi}{2}-\eta\le\arg z\le \frac{\pi}{2}+\eta,\,\, |z|\ge \chi\right\}.
\end{equation}

In \cite [Lemma 4.1]{Z2011JAT}, we proved
that for every $\epsilon>0$ there is a positive constant $M_{\epsilon}$, depending on $\eta$  and $\tau$, so that
\begin{equation}\label{lowerboundforproductcos}
|H(z)|\ge M_{\epsilon}e^{-\epsilon |z|}e^{\tau\cdot \Im z},\quad \forall\,\,z\in\Omega_{\eta}.
\end{equation}

Since $\sigma+\tau=\beta$, then
for every $\epsilon>0$ there is a positive constant $M_{\epsilon}$, depending on $\eta$  and $\tau$, so that
\begin{equation}\label{upperomega}
|H(z)|\cdot|e^{-i\sigma z}|\ge M_{\epsilon}e^{-\epsilon |z|}e^{\tau\cdot \Im z}\cdot e^{\sigma\cdot \Im z}=
M_{\epsilon}e^{-\epsilon |z|}e^{\beta\cdot \Im z},\quad \forall\,\,z\in\Omega_{\eta}.
\end{equation}

Fix such a positive $\epsilon$ and consider the $P_{n,\epsilon}$ disks and their circles $\partial P_{n,\epsilon}$.
Clearly they are all subsets of the region $\Omega_{\eta}$. For any $z\in \partial P_{n,\epsilon}$, $z\approx i\lambda_n$ and
$|z|\approx |\lambda_n|<A\Re\lambda_n$ for some $A>0$ since $\sup_{n\in\mathbb{N}}|\arg\lambda_n|<\pi/2$. Replacing in
$(\ref{upperomega})$ gives
\[
|H(z)|\cdot|e^{-i\sigma z}|\ge M_{\epsilon}e^{(\beta-A\epsilon)\cdot \Re\lambda_n},
\qquad \forall\,\, z\in \partial P_{n,\epsilon},\quad n=1,2,\dots.
\]
Combining the above with the lower bound of $L(z)$ in $(\ref{lowerboundforL})$, shows that  $(\ref{lowerboundforthefunctionGz})$ is true.
\end{proof}

\section{The Distance between $x^ke^{\lambda_n x}$ and the closed span of $E_{\Lambda}\setminus x^ke^{\lambda_n x}$
in $L^p(\gamma,\beta)$: Proof of Theorem $\ref{Distances}$}
\setcounter{equation}{0}

This section is devoted to the proof of our Fundamental result.
In what follows, we suppose that $\Lambda\in ABC$, $(\gamma,\beta)$ is a fixed bounded interval,
and $G$ is the entire function of Theorem $\ref{Zikkostheorem}$.
We also note that since $\Lambda\in ABC$, then $\mu_n/\Re\lambda_n\to 0$, thus
\begin{equation}\label{munlambdan}
\forall\,\,\epsilon>0,\quad \exists\quad \text{a\,\, positive\,\, constant}\,\, m_{\epsilon}:\quad
\mu_n\le \epsilon\Re\lambda_n\le m_{\epsilon}e^{\epsilon\Re\lambda_n},
\end{equation}
a relation to be used several times in this paper. In addition, for a function $f$, by $f^{(m)}$ we mean the $m^{th}$ derivative function of $f$.

\subsection{Auxiliary results}

\begin{lemma}\label{FirstLemma}
There exist entire functions $\{G_{n,k}(z):\,\, n\in\mathbb{N},\,\,  k=0,1,\dots,\mu_n-1\}$ so that
\begin{equation}\label{phi1a}
G_{n,k}^{(l)}(i\lambda_j)
=\begin{cases} 1, & j=n,\,\,  l=k, \\ 0,  & j=n,\,\,
l\in\{0,1,\dots,\mu_n-1\}\setminus\{k\}, \\ 0, &
j\not=n,\,\, l\in\{0,1,\dots,\mu_j-1\}.\end{cases}
\end{equation}
Moreover, from Remark $\ref{disjointdisks}$ consider for $\bf fixed$ $\epsilon>0$ and all $n\in\mathbb{N}$
the disks $P_{n,\epsilon}$ and their respective circles $\partial P_{n,\epsilon}$.
Then there is a constant $M_{\epsilon ,2}>0$, independent of $n$ and $k$ but depending on $\Lambda$ and $(\beta-\gamma)$,
so that  for every fixed $n\in\mathbb{N}$ and $k\in\{0,1,\dots,\mu_n-1\}$ we have
\begin{equation}\label{upperboundforGnkz}
|G_{n,k}(z)|\le |G(z)| M_{\epsilon ,2} e^{(-\beta+\epsilon)\Re\lambda_n},\quad \forall\,\, z\in\mathbb{C}\setminus P_{n,\epsilon},
\end{equation}
and
\begin{equation}\label{upperboundsrealline}
|G_{n,k}(x)|\le |G(x)| M_{\epsilon ,2} e^{(-\beta+\epsilon)\Re\lambda_n},
\quad \forall\,\, x\in\mathbb{R}.
\end{equation}
\end{lemma}

\begin{proof}
We note that the idea of constructing the family $\{G_{n,k}\}$ comes from \cite[page 4312]{Sed1}.

Obviously
\[
\frac{1}{G(z)}
\]
is a meromorphic function and at each point $z=i\lambda_n$ it has a pole of order $\mu_n$. This pole is the center of the
disk $P_{n,\epsilon}$. Now, let
\[
A_{n,j}:=\frac{1}{2\pi i}\int_{\partial P_{n,\epsilon}} \frac{(z-i\lambda_n)^{j-1}}{G(z)}\, dz,
\qquad j=1,\dots,\mu_n.
\]
From $(\ref{lowerboundforthefunctionGz})$ and the small radius of the disk $P_{n,\epsilon}$, one deduces
that for the $\bf fixed$ $\epsilon>0$
there is $M_{\epsilon ,1}>0$, independent of $n$ and $j$ but depending on $(\beta-\gamma)$, so that

\begin{equation}\label{upperboundforAnj}
|A_{n,j}|\le M_{\epsilon ,1}e^{(-\beta+\epsilon)\Re\lambda_n},\qquad \forall\,\, n\in\mathbb{N},\quad j=1,\dots,\mu_n.
\end{equation}

Let us also consider the punctured disk $P^*_{n,\epsilon}$: it is the disk $P_{n,\epsilon}$ excluding the point $i\lambda_n$, that is
\[
P^*_{n,\epsilon}:=\left\{z:\,\, 0<|z-i\lambda_n|< \frac{m_{\epsilon}}{6} \exp\left\{{-\frac{\epsilon |\lambda_n|}{\mu_n}}\right\}\right\}.
\]

With $A_{n,j}$ as above, we write down the Laurent series representation of $1/G$ in $P^*_{n,\epsilon}$,
\begin{equation}\label{oneover}
\frac{1}{G(z)}=\sum_{j=1}^{\mu_n}\frac{A_{n,j}}{(z-i\lambda_n)^j}+p_n(z),\qquad \text{for\,\, all}\,\, z\in P^*_{n,\epsilon}
\end{equation}
such that $p_n(z)$ is the regular part.

Next, for every positive integer $n$ and every $k\in\{0,1,2,\dots,\mu_n-1\}$ let
\begin{equation}\label{firstformulaforGnk}
G_{n,k}(z): = \frac{G(z)}{k!}\sum_{l=1}^{\mu_n-k}\frac{A_{n,k+l}}{(z-i\lambda_n)^l}.
\end{equation}
Obviously each $G_{n,k}$ is an entire function. We show below that it satisfies
$(\ref{phi1a})$ and then $(\ref{upperboundforGnkz})-(\ref{upperboundsrealline})$.\\

\smallskip

First suppose that $k=0$, thus
\[
G_{n,0}(z) = G(z)\sum_{l=1}^{\mu_n}\frac{A_{n,l}}{(z-i\lambda_n)^l}.
\]
Then we get $G_{n,0} ^{(l)}(i\lambda_j)=0$ for $j\not= n$ and $l=0,1,\dots, \mu_j -1$.
Also, from $(\ref{oneover})$ and since $G_{n,0}(z)$ is continuous at $z=i\lambda_n$, then
\[
G_{n,0}(z)=G(z)\left[ \frac{1}{G(z)}-p_{n}(z)\right]=1-G(z)p_{n}(z)\qquad\forall\,\, z\in P_{n,\epsilon}.
\]
Hence, $G_{n,0}(i\lambda_n)=1$ and  $G_{n,0}^{(l)}(i\lambda_n)=0$ for $l\in\{1,\dots, \mu_n -1\}$.
Thus, $G_{n,0}(z)$ satisfies $(\ref{phi1a})$.

Next, suppose that $k\in\{1,2,\dots ,\mu_n-1\}$. Clearly from $(\ref{firstformulaforGnk})$ we get $G_{n,k} ^{(l)}(i\lambda_j)=0$ for $j\not= n$ and $l=0,1,\dots, \mu_j -1$.
Then, from $(\ref{oneover})$ and since $G_{n,k}(z)$ is continuous at $z=i\lambda_n$,
we rewrite $G_{n,k}(z)$ for all $z$ in $P_{n,\epsilon}$ as
\begin{eqnarray}
G_{n,k}(z) & = & \frac{G(z) (z-i\lambda_n)^k}{k!}\sum_{l=k+1}^{\mu_n}\frac{A_{n,l}}{(z-i\lambda_n)^l}\nonumber\\
& = & \frac{G(z) (z-i\lambda_n)^k}{k!}\left[\frac{1}{G(z)}-p_n(z)- \sum_{j=1}^{k}\frac{A_{n,j}}{(z-i\lambda_n)^j}\right]\nonumber\\
& = & \frac{(z-i\lambda_n)^k}{k!}-\frac{G(z) (z-i\lambda_n)^k p_n(z)}{k!}-\frac{G(z)}{k!}\sum_{j=1}^{k}A_{n,j}(z-i\lambda_n)^{k-j}.
\label{secondformulaforGnk}
\end{eqnarray}
From this relation we get  $G_{n,k} ^{(k)}(i\lambda_n)=1$ and $G_{n,k} ^{(l)}(i\lambda_n)=0$ for $l\in\{0,1,\dots, \mu_n -1\}\setminus\{k\}$.
Thus, $G_{n,k}(z)$ satisfies $(\ref{phi1a})$ for $k\not=0$ as well.\\

\smallskip

Next, by combining relations $(\ref{munlambdan})$, $(\ref{upperboundforAnj})$, $(\ref{firstformulaforGnk})$,
and the small radius of $P_{n,\epsilon}$,
yields the upper bound $(\ref{upperboundforGnkz})$ for all $z$ which lie outside the disk $P_{n,\epsilon}$.
Finally, the upper bound $(\ref{upperboundsrealline})$ on $\mathbb{R}$ follows from $(\ref{upperboundforGnkz})$ and the fact that the points
$i\lambda_n$ do not lie on $\mathbb{R}$ since $\sup_{n\in\mathbb{N}}|\arg\lambda_n|<\pi/2$.
\end{proof}

\begin{lemma}\label{SecondLemma}
There exist continuous functions $\{g_{n,k}(t):\,\, n\in\mathbb{N},\,\, k=0,1,\dots,\mu_n-1\}$
on the interval $[\gamma,\beta]$, with $g_{n,k}(t)=0$ outside $[\gamma,\beta]$
so that
\begin{equation}\label{orthogonal}
\frac{1}{\sqrt{2\pi}}\int_{\gamma}^{\beta}g_{n,k}(t)(-it)^{l}e^{\lambda_j t}\,
dt=\begin{cases} 1, & j=n,\,\,  l=k, \\ 0,  & j=n,\,\,
l\in\{0,1,\dots,\mu_n-1\}\setminus\{k\}, \\ 0, &
j\not=n,\,\, l\in\{0,1,\dots,\mu_j-1\}.\end{cases}
\end{equation}
Furthermore, for every $\epsilon>0$ there is a constant $M_{\epsilon ,3}>0$ independent of $n$ and $k$
but depending on $\Lambda$ and $(\beta-\gamma)$, so that
\begin{equation}\label{upperboundforHnk}
|g_{n,k}(t)|\le M_{\epsilon ,3} e^{(-\beta+\epsilon)\Re\lambda_n}\qquad \forall\,\, t\in[\gamma,\beta],\quad  n\in\mathbb{N}, \quad k\in\{0,1,\dots,\mu_n-1\}.
\end{equation}
\end{lemma}
\begin{proof}
From $(\ref{GzZikkos})$, $(\ref{upperboundforGnkz})$, and $(\ref{upperboundsrealline})$,
it follows that $G_{n,k}(z)$ is an entire function of exponential type, $G_{n,k}\in L^2(\mathbb{R})\cap L^1(\mathbb{R})$,
and it is the Fourier Transform
of a continuous function compactly  supported on the interval $[\gamma,\beta]$.
Thus $G_{n,k} (z)$ admits the representation
\[
G_{n,k}(z)=\frac{1}{\sqrt{2\pi}}\int_{\gamma}^{\beta}e^{-izt}g_{n,k}(t)\, dt,
\]
for some $g_{n,k}\in C[\gamma,\beta]$ with $g_{n,k}(t)=0$ outside $[\gamma,\beta]$.
By differentiation one also has
\[
G_{n,k} ^{(l)} (z)=\frac{1}{\sqrt{2\pi}}\int_{\gamma}^{\beta}(-it)^{l}e^{-izt}g_{n,k}(t)\, dt.
\]
It now follows from $(\ref{phi1a})$ that $(\ref{orthogonal})$ is valid.

Moreover, since $G_{n,k}\in L^1(\mathbb{R})$, by Fourier Inversion we have
\[
g_{n,k}(t)=\frac{1}{\sqrt{2\pi}}\int_{-\infty}^{\infty}e^{ixt}G_{n,k}(x)\, dx\quad \forall\,\, t\in [\gamma,\beta].
\]
Finally,  since $G\in L^1 (\mathbb{R})$ then from  $(\ref{upperboundsrealline})$ we get $(\ref{upperboundforHnk}$).
\end{proof}

\subsection{Proof of the Fundamental Result (Theorem $\ref{Distances}$)}

Fix some $n\in\mathbb{N}$ and $k\in\{0,1,\dots, \mu_n-1\}$ and
let $g_{n,k}$ be the function as in Lemma $\ref{SecondLemma}$.
Consider also the exponential system $E_{\Lambda_{n,k}}=E_{\Lambda}\setminus x^ke^{\lambda_n x}$.

Since $g_{n,k}\in C[\gamma,\beta]$
then $g_{n,k}\in L^q (\gamma,\beta)$ for all $q\ge 1$ and
it follows from $(\ref{upperboundforHnk})$ that for every $\epsilon>0$ there is some $M_{\epsilon,3}>0$,
independent of $n\in\mathbb{N}$, $k=0,1,\dots,\mu_n-1$, so that

\begin{equation}\label{Lqgnk}
||g_{n,k}(t)||_{L^q (\gamma,\beta)}=\left(\int_{\gamma}^{\beta}|g_{n,k}(t)|^q\, dt\right)^{1/q}\le
M_{\epsilon,3}\cdot e^{(-\beta+\epsilon)\Re\lambda_n}\cdot \max\{(\beta-\gamma),1\}.
\end{equation}

\begin{remark}\label{ind}
This bound is independent of $q\in [1,\infty)$.
\end{remark}

Suppose now that $f\in\overline{\text{span}}(E_{\Lambda_{n,k}})$ in the space $L^p(\gamma,\beta)$ for some $p>1$.
Hence for every $\epsilon>0$ there is an exponential polynomial
$P_{\epsilon}\in \text{span}(E_{\Lambda_{n,k}})$ such that
\[
||f-P_{\epsilon}||_{L^p(\gamma,\beta)}<\epsilon.
\]

From $(\ref{orthogonal})$ we have
\[
\int_{\gamma}^{\beta}g_{n,k}(t)P_{\epsilon}(t)\, dt=0,\qquad\text{thus}\qquad
\int_{\gamma}^{\beta}g_{n,k}(t)f(t)\, dt=\int_{\gamma}^{\beta}g_{n,k}(t)\cdot \left(f(t)-P_{\epsilon}(t)\right)\, dt.
\]
Let $q$ be the conjugate of $p$, that is, $1/p +1/q=1$.
Then from the H\"{o}lder inequality we get
\[
\left|\int_{\gamma}^{\beta}g_{n,k}(t)\cdot\left(f(t)-P_{\epsilon}(t)\right)\, dt\right|
\le ||g_{n,k}||_{L^q(\gamma,\beta)}\cdot\epsilon.
\]

Since $\epsilon$ is arbitrary we conclude that
\[
\int_{\gamma}^{\beta}g_{n,k}(t)\cdot f(t)\, dt=0.
\]

Together with $(\ref{orthogonal})$ gives
\[
\frac{\sqrt{2\pi}}{(-i)^k}=\int_{\gamma}^{\beta}g_{n,k}(t)\cdot t^k e^{\lambda_n t}\, dt=
\int_{\gamma}^{\beta}g_{n,k}(t)\left(t^k e^{\lambda_n t}-f(t)\right)\, dt.
\]

Letting  $e_{n,k}(t)=t^k e^{\lambda_n t}$, then from $(\ref{Lqgnk})$ and the H\"{o}lder inequality we get
\begin{eqnarray}
\sqrt{2\pi} & \le & ||g_{n,k}(t)||_{L^q (\gamma,\beta)} \cdot ||e_{n,k}-f||_{L^p (\gamma,\beta)}
\nonumber\\& \le &  M_{\epsilon,3}\cdot e^{(-\beta+\epsilon)\Re\lambda_n}\cdot \max\{(\beta-\gamma),1\}\cdot
||e_{n,k}-f||_{L^p (\gamma,\beta)}.\nonumber
\end{eqnarray}

Thus
\[
||e_{n,k}-f||_{L^p (\gamma,\beta)}\ge \frac{\sqrt{2\pi}}{M_{\epsilon,3}\cdot \max\{(\beta-\gamma),1\}}\cdot e^{(\beta-\epsilon)\Re\lambda_n}.
\]

Since this is true for all $f\in\overline{\text{span}} (E_{\Lambda_{n,k}})$ in $L^p (\gamma,\beta)$, and letting
\[
u_{\epsilon}=\frac{\sqrt{2\pi}}{M_{\epsilon,3}\cdot \max\{(\beta-\gamma),1\}}
\]
we get the Distance lower bound
\[
D_{\gamma,\beta,p,n,k}\ge u_{\epsilon}e^{(\beta-\epsilon)\Re\lambda_n}
\]
with $u_{\epsilon}$ clearly independent of $n,\, k,$ and also independent of $p>1$ due to Remark $\ref{ind}$.

Similarly one gets $D_{\gamma,\beta,1,n,k}\ge u_{\epsilon}e^{(\beta-\epsilon)\Re\lambda_n}$.
The proof of Theorem $\ref{Distances}$ is now complete.

\subsection{An important corollary}

The following result is crucial for proving Theorem $\ref{theorem1}$.

\begin{corollary}\label{corbound}
Let the multiplicity sequence $\Lambda=\{\lambda_n,\mu_n\}_{n=1}^{\infty}$ belong to the class $ABC$
and let $(\gamma, \beta)$ be a bounded interval. Consider two exponential polynomials
\[
P_1(x)=\sum_{n=1}^{M_1}\left(\sum_{k=0}^{\mu_n-1} c_{n,k,1}x^{k}\right) e^{\lambda_n x},\qquad c_{n,k,1}\in\mathbb{C}
\]
and
\[
P_2(x)=\sum_{n=1}^{M_2}\left(\sum_{k=0}^{\mu_n-1} c_{n,k,2}x^{k}\right) e^{\lambda_n x},\qquad c_{n,k,2}\in\mathbb{C},
\]
such that $M_2\ge M_1$. Then, for every $\epsilon>0$ there is a constant $m_{\epsilon}>0$
which depends only on $\Lambda$ and $(\beta-\gamma)$, and not on $P_1(x)$, $P_2(x)$, $p\ge 1$,
$n\in\mathbb{N}$ and $k=0,1,\dots,\mu_n-1$, so that
\[
(I)\qquad |c_{n,k,1}|\le m_{\epsilon}e^{(-\beta+\epsilon)\Re\lambda_n}||P_1||_{L^p (\gamma, \beta)}
\qquad \forall\,\, n=1,2,\dots, M_1\quad \text{and}\quad k=0,1,\dots,\mu_{n}-1
\]
and
\[
(II)\qquad |c_{n,k,1}-c_{n,k,2}|\le m_{\epsilon} e^{(-\beta+\epsilon)\Re\lambda_n}||P_1-P_2||_{L^p (\gamma,\beta)}\quad
\forall\,\, n=1,\dots ,M_1 \quad\text{and}\quad k=0,1,\dots,\mu_n-1.
\]
\end{corollary}

\begin{proof}
Fix $m\in \{1,\dots , M_1\}$ and $l\in \{0,1,\dots, \mu_m-1\}$.
Let $e_{m,l}(x)=x^l e^{\lambda_m x}$, $E_{\Lambda_{m,l}}=E_{\Lambda}\setminus e_{m,l}$, and write
\[
||P_1||_{L^p (\gamma, \beta)}=|c_{m,l,1}|\cdot ||e_{m,l}+Q_{m,l,1}||_{L^p (\gamma, \beta)},
\]
where
\[
Q_{m,l,1}(x):=\left(\sum_{k=0,k\not=l}^{\mu_m-1} \frac{c_{m,k,1}}{c_{m,l,1}}x^{k}\right) e^{\lambda_m x}+\sum_{n=1,n\not=m}^{M_1}
\left(\sum_{k=0}^{\mu_n-1} \frac{c_{n,k,1}}{c_{m,l,1}}x^{k}\right) e^{\lambda_n x}.
\]
Clearly $Q_{m,l,1}$ belongs to the span of $E_{\Lambda_{m,l}}$, thus $(I)$ follows from Theorem $\ref{Distances}$.\\

Next, $(II)$ obviously holds if $c_{n,k,1}=c_{n,k,2}$ for some $n\in\{1,2,\dots , M_1\}$ and $k\in\{0,1,\dots ,\mu_n-1\}$.
Suppose now that $c_{m,l,1}\not=c_{m,l,2}$
for $m\in \{1,\dots ,M_1\}$ and $l\in \{0,1,\dots, \mu_m-1\}$. We then write
\[
|P_1(x)-P_2(x)| = |c_{m,l,1}-c_{m,l,2}|\cdot \left|x^{l}e^{\lambda_m x}+Q_{m,l,1,2}(x)\right|
\]
where
\begin{eqnarray*}
Q_{m,l,1,2}(x): & = & \left(\sum_{k=0,k\not=l}^{\mu_m-1} \frac{c_{m,k,1}-c_{m,k,2}}{c_{m,l,1}-c_{m,l,2}}x^k\right)
e^{\lambda_m x}+ \sum_{n=1,n\not=m}^{M_1}
\left(\sum_{k=0}^{\mu_n-1} \frac{c_{n,k,1}-c_{n,k,2}}{c_{m,l,1}-c_{m,l,2}}x^k\right) e^{\lambda_n x}
\\ & - &  \sum_{n=M_1+1}^{M_2}\left(\sum_{k=0}^{\mu_n-1} \frac{c_{n,k,2}}{c_{m,l,1}-c_{m,l,2}}x^k\right) e^{\lambda_n x}.
\end{eqnarray*}
Hence
\[
||P_1-P_2||_{L^p (\gamma, \beta)} = |c_{m,l,1}-c_{m,l,2}|\cdot ||e_{m,l}+Q_{m,l,1,2}||_{L^p (\gamma, \beta)}.
\]
Obviously $Q_{m,l,1,2}$ belongs to the span of $E_{\Lambda_{m,l}}$, hence
$(II)$ follows from Theorem $\ref{Distances}$.

\end{proof}

\section{A result on the region of holomorphy of  Taylor-Dirichlet series}
\setcounter{equation}{0}

In this section we state and prove a result on Taylor-Dirichlet series which as we have seen
appear in the statement of our results in Introduction.

Suppose that a  multiplicity sequence $\Lambda=\{\lambda_n,\mu_n\}_{n=1}^{\infty}$,
not necessarily in the $ABC$ class, satisfies the following two conditions:
\begin{equation}\label{ValironConditionslimit}
\lim_{n\to\infty}\frac{\log n}{\lambda_n}=0 \qquad \text{and} \qquad \lim_{n\to\infty}\frac{\mu_n}{\lambda_n}=0.
\end{equation}
Associate to $\Lambda$ the class of Taylor-Dirichlet series
\[
h(z)=\sum_{n=1}^{\infty} \left(\sum_{k=0}^{\mu_n-1}a_{n,k}
z^k\right) e^{\lambda_n z},\quad a_{n,k}\in \mathbb{C},
\]
and to each $h(z)$ consider the Dirichlet series
\[
h^*(z)=\sum_{n=1}^{\infty} A_n e^{\lambda_n z}\qquad
\text{where}\qquad
A_n=\max\{|a_{n,k}|:\, k=0,1,2,\dots,\mu_n-1\}.
\]
From the first condition in $(\ref{ValironConditionslimit})$, it follows by the results of
E. Hille \cite[Theorems 1 and 3]{Hille} that the open region of absolute convergence of $h^*$
is convex and it coincides with its open region of pointwise convergence.
Assuming both conditions in $(\ref{ValironConditionslimit})$,
it then follows from the results  of G. Valiron \cite[p. 29]{Valiron}, that
the open regions of pointwise and absolute convergence of $h(z)$ coincide with the open region of convergence
of $h^*(z)$, that is, all open regions are identical and convex.
Both series $h$ and $h^*$, converge uniformly on every compact subset of the open region, call it $D$,
thus defining analytic functions in $D$.

\begin{remark}
We note however that it is possible for Taylor-Dirichlet series to converge pointwise on a $\bf set$ of points,
call it $\Delta$, which lies outside the region $D$. We will discuss more on this phenomenon in subsection 10.2.
\end{remark}

And now our own contribution to the topic of Taylor-Dirichlet series.
\begin{lemma}\label{TDSeries}
Let the multiplicity sequence $\Lambda=\{\lambda_n,\mu_n\}_{n=1}^{\infty}$ be an interpolating variety for the space $A^0_{|z|}$
and satisfying the condition  $(\ref{lessthan})$. Then a Taylor-Dirichlet series
\[
g(z)=\sum_{n=1}^{\infty}\left(\sum_{k=0}^{\mu_n-1} c_{n,k}z^k\right)e^{\lambda_n z},\quad c_{n,k}\in\mathbb{C},
\]
defines an analytic function  in the open sector $\Theta_{\eta ,\beta}$ $(\ref{opensector})$,
converging uniformly on its compact subsets,
\[
\text{if\quad and\quad only\quad if}
\]
the coefficients $c_{n,k}$ satisfy the upper bound $(\ref{cnkbound})$.

Such a series is either an entire function or an analytic function in a convex region $D\supset\Theta_{\eta ,\beta}$.
In the latter case the boundary of $D$ is a natural boundary for $g$.
\end{lemma}

\begin{proof}

Firstly, suppose that the coefficients $c_{n,k}$ satisfy the upper bound $(\ref{cnkbound})$.
We will show that $g$ is analytic in $\Theta_{\eta,\beta}$ by proving that $g$ converges uniformly on its compact subsets.

Consider such a compact set $K$ so that $\sup_{z\in K}|z|=M>1$.
Clearly we can shift $\Theta_{\eta,\beta}$ to the left by
$\alpha$ units for some $\alpha>0$, so that $K$ remains in the interior of this shifted sector, call it
$\Theta_{\eta,\beta,\alpha}$:
\[
\Theta_{\eta,\beta,\alpha}:=\left\{z:\, |\Im z| \le |\Re (z-\beta+\alpha)|\left( \frac{1}{\tan\eta}\right),
\quad \Re z\le \beta-\alpha\right\}.
\]

Since $\Lambda$ satisfies condition  $(\ref{lessthan})$, then $|\Im \lambda_n|\le (\Re\lambda_n)\cdot \tan\eta$.
Also, since $z\in \Theta_{\eta,\beta,\alpha}$ then $|\Re (z-\beta+\alpha)|= -\Re z+\beta-\alpha$. Then for all $z\in K$ one has
\begin{eqnarray*}
|\exp \{\lambda_n z\}| & = & \exp\left\{\Re\lambda_n\cdot \Re z-\Im\lambda_n\cdot \Im z\right\}\\
& \le & \exp\left\{\Re\lambda_n\cdot \Re z+|\Im\lambda_n|\cdot |\Im z|\right\}\\
& \le & \exp\left\{\Re\lambda_n\cdot \Re z+(\Re\lambda_n\cdot \tan\eta)\cdot |\Re (z-\beta+\alpha)|\cdot \left(\frac{1}{\tan\eta}\right)\right\}\\
& = &
\exp\{\Re\lambda_n\cdot [\Re z+|\Re (z-\beta+\alpha)|]\}\\
& = &
\exp\{\Re\lambda_n\cdot (\beta-\alpha)\}.
\end{eqnarray*}

Since the coefficients satisfy $(\ref{cnkbound})$, combined with the above upper bound
and $(\ref{munlambdan})$, shows that for every $\epsilon>0$  there are $m^*_{\epsilon}>0$ and $m_{\epsilon}>0$,
so that for all $z\in K$ we have
\begin{eqnarray*}
\sum_{k=0}^{\mu_n-1}|c_{n,k}|\cdot |z|^k\cdot |e^{\lambda_n z}| & \le &
\mu_n\cdot m^*_{\epsilon}\cdot \exp\{(-\beta+\epsilon)\cdot \Re\lambda_n\}\cdot M^{\mu_n-1}\cdot \exp\{(\beta-\alpha)\cdot \Re\lambda_n\}\\
& \le & m_{\epsilon}\cdot
\exp\{\epsilon\Re\lambda_n\}\cdot \exp\{(-\beta+\epsilon)\cdot\Re\lambda_n\}\cdot\exp\{\epsilon\Re\lambda_n\}
\cdot\exp\{(\beta-\alpha)\cdot\Re\lambda_n\}\\
& = & m_{\epsilon}\cdot \exp\{(3\epsilon - \alpha)\cdot\Re\lambda_n\}.
\end{eqnarray*}

Let us now choose  $\epsilon$ to be equal to $\alpha/4$. Then one gets
\[
\sum_{n=1}^{\infty}\left(\sum_{k=0}^{\mu_n-1}|c_{n,k}|\cdot |z|^k\right)\cdot |e^{\lambda_n z}|
\le\sum_{n=1}^{\infty}m_{\alpha/4}\cdot e^{(-\alpha/4)\cdot\Re\lambda_n}<\infty.
\]
This means that $g$ converges uniformly on the set $K$. The arbitrary choice of $K$ shows that
$g$ is an analytic function in the sector $\Theta_{\eta,\beta}$.\\

Secondly, suppose that $g$ is an analytic function in the sector $\Theta_{\eta,\beta}$,
thus all real points $\rho<\beta$ belong to the sector.
We will show that the coefficients $c_{n,k}$ satisfy $(\ref{cnkbound})$ by
utilizing the results obtained by G. Valiron and E. Hille.

Since in our case $\Lambda$ is an interpolating variety for the space $A^0_{|z|}$, then it has zero $Density$, in other words
\[
\frac{\sum_{|\lambda_n|\le t}\mu_n}{t}\to 0,\qquad t\to\infty.
\]
Therefore,
\[
\frac{\mu_1+\mu_2+\mu_3+\dots+\mu_n}{|\lambda_n|}\to 0,\qquad n\to\infty.
\]
Then it easily follows that both relations in $(\ref{ValironConditionslimit})$ hold.
Therefore if we compare the series $g(z)$ with the series
\[
g^*(z)=\sum_{n=1}^{\infty} C_n e^{\lambda_n z}\qquad
\text{where}\qquad
C_n=\max\{|c_{n,k}|:\, k=0,1,2,\dots,\mu_n-1\},
\]
the open regions of pointwise and absolute convergence of the series $g(z)$ coincide with the open region of
convergence of the series $g^* (z)$ which is convex. If we denote this open region by $D$ and since in $D$ the series $g$ is analytic,
then obviously we get $D\supset \Theta_{\eta,\beta}$.

Now, if the coefficients $c_{n,k}$ of the series $g$ do not satisfy the upper bound $(\ref{cnkbound})$, then we will have
\[
\limsup_{n\to\infty}\frac{\log C_n}{\Re\lambda_n}=a>-\beta.
\]
If $a\in\mathbb{R}$, and similarly if $a=\infty$, it is then easy to see that the series
\[
\sum_{n=1}^{\infty}C_n\cdot e^{x\Re\lambda_n}
\]
does not converge for $x_o=(\beta-a)/2$. However, $x_o<\beta$ hence $x_o\in \Theta_{\eta,\beta}$, a contradiction.
Therefore the coefficients $c_{n,k}$ must satisfy the upper bound $(\ref{cnkbound})$.\\

And finally, we note that if the open region $D$ is not the complex plane $\mathbb{C}$,
it follows from our work \cite[Theorem 4.1]{Z2015JMAA}
that its boundary $\partial D$  is a Natural Boundary,
in other words, the series cannot be analytically continued across any part of $\partial D$.
The proof of our lemma is now complete.
\end{proof}

\section{The Closed Span of $E_{\Lambda}$ in $L^p (\gamma,\beta)$: Proof of Theorems $\ref{theorem1}$ and $\ref{converse}$}
\setcounter{equation}{0}

In this section we characterize the closed span of the system $E_{\Lambda}$ in $L^p (\gamma,\beta)$
by proving Theorem $\ref{theorem1}$ as well as its converse result Theorem $\ref{converse}$.

\subsection{Proof of Theorem $\ref{theorem1}$}

Suppose that $f\in\overline{\text{span}} (E_{\Lambda})$  in $L^p (\gamma,\beta)$ for some $p\ge 1$.
Then there exists a sequence $\{P_j(x)\}_{j=1}^{\infty}$ in  $\text{span} (E_{\Lambda})$, where
\[
P_j(x)=\sum_{n=1}^{r(j)}\left(\sum_{k=0}^{\mu_n-1} c_{n,k,j}x^{k}\right) e^{\lambda_n x}
\]
such that $||f-P_j||_{L^p (\gamma,\beta)}\to 0$ as $j\to\infty$. Without loss of generality,
suppose that $r(j)$ is unbounded, thus we may assume that $r(j)$ is strictly increasing.

It follows from Corollary $\ref{corbound}$ that for every $\epsilon>0$ there is a positive
constant $m_{\epsilon}$ independent of $\{P_j\}_{j=1}^{\infty}$, $n\in\mathbb{N}$ and $k=0,1,\dots,\mu_n-1$, so that
\begin{equation}\label{Cauchy}
|c_{n,k,j}|\le m_{\epsilon}e^{(-\beta+\epsilon)\Re\lambda_n}||P_j||_{L^p (\gamma,\beta)},\qquad n=1,\dots ,r(j)\quad
and\quad k=0,1,\dots,\mu_n-1,
\end{equation}
and for $i\ge j$, one has
\begin{equation}\label{nkjnki}
|c_{n,k,j}-c_{n,k,i}|\le m_{\epsilon} e^{(-\beta+\epsilon)\Re\lambda_n}||P_j-P_{i}||_{L^p (\gamma,\beta)},\qquad
n=1,\dots ,r(j) \quad and \quad k=0,1,\dots,\mu_n-1.
\end{equation}

Fixing $n,k$ in $(\ref{nkjnki})$ and
since $||P_j||_{L^p (\gamma,\beta)}\to ||f||_{L^p (\gamma,\beta)}$ as $j\to\infty$,
shows that $\{c_{n,k,j}\}_{j=1}^{\infty}$ is a Cauchy sequence,
hence converging to some complex number, call it $c_{n,k}$. Furthermore, from $(\ref{Cauchy})$ we get
\begin{equation}\label{Ankbound}
|c_{n,k}|\le  m_{\epsilon} e^{(-\beta+\epsilon)\Re\lambda_n}||f||_{L^p (\gamma,\beta)},\qquad\forall\,\,n\in\mathbb{N},\quad \forall\,\, k=0,1,\dots,\mu_n-1.
\end{equation}
It now follows from Lemma $\ref{TDSeries}$ that
\begin{equation}\label{TDg}
g(z):=\sum_{n=1}^{\infty}\left(\sum_{k=0}^{\mu_n-1} c_{n,k}z^{k}\right) e^{\lambda_n z}
\end{equation}
converges uniformly on compact subsets of the sector $\Theta_{\eta,\beta}$, therefore $g$ is analytic in $\Theta_{\eta,\beta}$.
Thus $g$ converges uniformly  on intervals of the form $[\gamma,\beta-\rho]$ for any small $\rho>0$.\\

We now claim that $\{P_j\}_{j=1}^{\infty}$ also converges to $g$ uniformly on such subintervals.
Fix some small $\rho>0$ and consider the interval $[\gamma,\beta-\rho]$. Choose also
\begin{equation}\label{epsilonrho}
\epsilon=\frac{\rho}{6},\qquad\text{and}\quad T=\max\{1, |\gamma|,|\beta-\rho|\}.
\end{equation}
For this $\epsilon>0$ it follows from $(\ref{munlambdan})$ that there is a positive constant $m_{\epsilon}$,
independent of $n\in\mathbb{N}$, so that
\begin{equation}\label{various}
|x|^{\mu_n}\le T^{\mu_n}\le m_{\epsilon}e^{\epsilon\Re\lambda_n}\qquad \forall\,\, x\in [\gamma,\beta-\rho].
\end{equation}

Then, for all $x\in [\gamma,\beta-\rho]$ we write
\begin{equation}\label{Pjg}
|P_j(x)-g(x)|\le |I(x)|+|II(x)|
\end{equation}
where
\[
I(x):=\sum_{n=1}^{r(j)}\left(\sum_{k=0}^{\mu_n-1} (c_{n,k,j}-c_{n,k})x^{k}\right) e^{\lambda_n x}
\]
and
\[
II(x):=\sum_{n=r(j)+1}^{\infty}\left(\sum_{k=0}^{\mu_n-1} c_{n,k}x^{k}\right) e^{\lambda_n x}.
\]

Now, in relation $(\ref{nkjnki})$ if we keep $j$ fixed and let $i\to\infty$, shows that
\begin{equation}\label{Ankbound1}
|c_{n,k,j}-c_{n,k}|\le  m_{\epsilon} e^{(-\beta+\epsilon)\Re\lambda_n}||P_j-f||_{L^p (\gamma,\beta)},\quad \forall\,\, n\in\mathbb{N},\quad\forall\,\, k=0,1,\dots,\mu_n-1.
\end{equation}
Thus from relations $(\ref{Ankbound})$ and $(\ref{Ankbound1})$, we get
\begin{equation}\label{firstinequality}
|I(x)|\le
||P_j-f||_{L^p (\gamma, \beta)} \cdot \sum_{n=1}^{r(j)}\left(\sum_{k=0}^{\mu_n-1}  m_{\epsilon} e^{(-\beta+\epsilon)\Re\lambda_n}|x|^{k}\right) e^{x\Re\lambda_n}
\end{equation}
and
\begin{equation}\label{secondinequality}
|II(x)|\le
||f||_{L^p (\gamma, \beta)}\cdot\sum_{n=r(j)+1}^{\infty}\left(\sum_{k=0}^{\mu_n-1} m_{\epsilon} e^{(-\beta+\epsilon)\Re\lambda_n} |x|^{k}\right) e^{x\Re\lambda_n}.
\end{equation}

From $(\ref{munlambdan})$, $(\ref {various})$, and since $\epsilon=\rho/6$, then for all $x\in [\gamma,\beta-\rho]$
there exists an $m^*_{\epsilon}>0$, so that
\begin{eqnarray*}
\sum_{n=1}^{\infty}\left(\sum_{k=0}^{\mu_n-1}e^{(-\beta+\epsilon)\Re\lambda_n} |x|^{k}\right) e^{x\Re\lambda_n}
& \le &
\sum_{n=1}^{\infty}\mu_n \cdot e^{(-\beta+\epsilon)\Re\lambda_n}\cdot T^{\mu_n}\cdot e^{x\Re\lambda_n}\\
& \le &
\sum_{n=1}^{\infty} m^*_{\epsilon}\cdot e^{\epsilon\Re\lambda_n}\cdot e^{(-\beta+\epsilon)\Re\lambda_n}
\cdot e^{\epsilon\Re\lambda_n}\cdot e^{(\beta-\rho)\Re\lambda_n}\\
& = &
\sum_{n=1}^{\infty}m^*_{\epsilon}\cdot e^{(3\epsilon-\rho)\Re\lambda_n}\\
& = &
\sum_{n=1}^{\infty}m^*_{\epsilon}\cdot e^{(-\frac{\rho}{2})\Re\lambda_n}\\
& < & \infty.
\end{eqnarray*}
This result together with the fact that $||P_j-f||_{L^p (\gamma,\beta)}\to 0$ as $j\to\infty$ shows
that the right hand-sides of the inequalities $(\ref{firstinequality})$
and $(\ref{secondinequality})$ converge to zero uniformly on $[\gamma,\beta-\rho]$ as $j\to\infty$. Finally,
by substituting in $(\ref{Pjg})$ we see that $\{P_j\}_{j=1}^{\infty}$ converges to $g$ uniformly on $[\gamma,\beta-\rho]$
as well, thus our claim is verified.\\

Therefore $||P_j-g||_{L^p (\gamma,\beta-\rho)}\to 0$ as $j\to\infty$ for any small $\rho>0$.
But $||P_j-f||_{L^p (\gamma,\beta-\rho)}\to 0$  as $j\to \infty$ as well for any small $\rho>0$
since $||P_j-f||_{L^p (\gamma,\beta)}\to 0$  as $j\to \infty$.
These facts show that $f(x)=g(x)$ almost everywhere on $(\gamma,\beta-\rho)$.
The arbitrary choice of $\rho>0$,
means that $f(x)=g(x)$ almost everywhere on $(\gamma,\beta)$.\\

And finally, the uniqueness of the Taylor-Dirichlet series
follows from Lemma $\ref{uniquenessDirichlet}$ (see the Appendix).
The proof of Theorem $\ref{theorem1}$ is now complete.

\subsection{The converse result: Proof of  Theorem $\ref{converse}$}

The proof is a refinement of \cite[Theorem 8.1]{Z2011JAT} and it is inspired by the work of Korevaar \cite{Korevaar1947}.

First note that since $f\in L^p(\gamma,\beta)$ and $f$ is continuous on $(-\infty, \beta)$,
then $f\in L^p(c,\beta)$ for any $c\in (-\infty, \beta)$. Second,
let $\{\delta_n\}_{n=1}^{\infty}$ be an arbitrary sequence of positive real numbers so that $\delta_n\to 0$ as $n\to\infty$.
The continuity of $f$ implies that for all $x\in [\gamma-1,\beta)$ we have $f(x-\delta_n)\to f(x)$ as $n\to \infty$.
By applying Fatou's Lemma and changing variables we get
\begin{eqnarray*}
\int_{\gamma}^{\beta}|f(x)|^p\,dx\le \liminf_{\delta_n\to 0}\int_{\gamma}^{\beta}|f(x-\delta_n)|^p\,dx\le
\limsup_{\delta_n\to 0}\int_{\gamma}^{\beta}|f(x-\delta_n)|^p\,dx & = &  \limsup_{\delta_n\to 0}\int_{\gamma-\delta_n}^{\beta-\delta_n}|f(x)|^p\,dx\\
& \le &\limsup_{\delta_n\to 0}\int_{\gamma-\delta_n}^{\beta}|f(x)|^p\,dx.
\end{eqnarray*}
Then we write
\begin{eqnarray*}
\limsup_{\delta_n\to 0}\int_{\gamma-\delta_n}^{\beta}|f(x)|^p\,dx  & = &
\limsup_{\delta_n\to 0}\left(\int_{\gamma-\delta_n}^{\gamma}|f(x)|^p\,dx + \int_{\gamma}^{\beta}|f(x)|^p\,dx \right)\\ & = &
\int_{\gamma}^{\beta}|f(x)|^p\,dx + \limsup_{\delta_n\to 0}\int_{\gamma-\delta_n}^{\gamma}|f(x)|^p\,dx\\ & = &
\int_{\gamma}^{\beta}|f(x)|^p\,dx,
\end{eqnarray*}
with the last step valid since $f$ is uniformly continuous on intervals $[\rho,\gamma]$ for every $\rho\in (-\infty,\gamma)$.

Combining the above, we have
\[
\int_{\gamma}^{\beta}|f(x)|^p\,dx\le \liminf_{\delta_n\to 0}\int_{\gamma}^{\beta}|f(x-\delta_n)|^p\,dx\le
\limsup_{\delta_n\to 0}\int_{\gamma}^{\beta}|f(x-\delta_n)|^p\,dx\le\int_{\gamma}^{\beta}|f(x)|^p\,dx.
\]

Clearly this means that  $\lim_{\delta_n\to 0}\int_{\gamma}^{\beta}|f(x-\delta_n)|^p\,dx$ exists and it is equal to $\int_{\gamma}^{\beta}|f(x)|^p\,dx$.
The arbitrary choice of $\{\delta_n\}_{n=1}^{\infty}$ implies that
$\lim_{\delta\to 0^+}\int_{\gamma}^{\beta}|f(x-\delta)|^p\,dx=\int_{\gamma}^{\beta}|f(x)|^p\,dx$.
Therefore
\[
\lim_{\delta\to 0^+}\int_{\gamma}^{\beta}|f(x)-f(x-\delta)|^p\, dx=0.
\]
Thus, for every $\epsilon>0$ there is $\delta>0$ such that
\begin{equation}\label{uniformgammabeta}
\left(\int_{\gamma}^{\beta}|f(x)-f(x-\delta)|^p\, dx\right)^{1/p}\le\epsilon.
\end{equation}

Fix such an $\epsilon$ and its $\delta$. We then have
\begin{eqnarray*}
f(x-\delta) & = & \sum_{n=1}^{\infty}\left(\sum_{k=0}^{\mu_n-1} c_{n,k} (x-\delta)^k\right) e^{\lambda_n (x-\delta)},\qquad (x-\delta)<\beta\\
& = & \sum_{n=1}^{\infty}\left[e^{-\lambda_n \delta}\sum_{k=0}^{\mu_n-1}c_{n,k}\left(\sum_{j=k}^{\mu_n-1}\frac{j!}{k!(j-k)!}(-\delta)^{j-k}\right) x^k\right] e^{\lambda_n x},
\qquad x<\beta+\delta.
\end{eqnarray*}
Since the series of $f(z)$ converges uniformly on compact subsets of the sector $\Theta_{\eta,\beta}$, then the series of $f(z-\delta)$ converges
uniformly on compact subsets of the sector
\[
\Theta_{\eta,\beta+\delta}:=
\left\{z:\left|\frac{\Im z}{\Re (z-\beta-\delta)}\right|\le \frac{1}{\tan\eta},\,\,\Re z<\beta+\delta\right\},
\]
hence on the interval $[\gamma,\beta]$ as well. Therefore,
for the fixed $\epsilon,\, \delta$ there is a positive integer $m_{\epsilon}$ so that
\[
\left|f(x-\delta)- \sum_{n=1}^{m_{\epsilon}}\left[e^{-\lambda_n \delta}\sum_{k=0}^{\mu_n-1}c_{n,k}\left(\sum_{j=k}^{\mu_n-1}\frac{j!}{k!(j-k)!}(-\delta)^{j-k}\right)
x^k\right] e^{\lambda_n x}\right|<\epsilon,\quad \forall\,\, x\in [\gamma,\beta].
\]

This relation together with $(\ref{uniformgammabeta})$ and the Minkowksi inequality,
yield that $f\in\overline{\text{span}}(E_{\Lambda})$ in the space $L^p(\gamma,\beta)$.

\section{The Biorthogonal Family $r_{\Lambda}$ to $E_{\Lambda}$: Proof of Theorem $\ref{biorthogonalsystem}$}
\setcounter{equation}{0}

In this section we prove Theorem $\ref{biorthogonalsystem}$ which is on the properties of a biorthogonal family $r_{\Lambda}$
to the system $E_{\Lambda}$ in $L^2 (\gamma, \beta)$.

\subsection{Constructing the Biorthogonal family and deriving the upper bound $(\ref{rnkbound})$}

Consider a multiplicity sequence $\Lambda$ in the $ABC$ class.
As before, let $e_{n,k}(x)=x^ke^{\lambda_n x}$ and $E_{\Lambda_{n,k}}=E_{\Lambda}\setminus e_{n,k}$.
From Theorem $\ref{Distances}$ we know that
for every $\epsilon>0$ there is a positive constant $u_{\epsilon}$ which depends only on $\Lambda$ and $(\beta-\gamma)$,
but not on $n\in\mathbb{N}$ and neither on $k=0,\dots,\mu_n-1$, such that the distance $D_{\gamma,\beta,2,n,k}$
in $L^2(\gamma,\beta)$, satisfies
\begin{equation}\label{L2casebound}
D_{\gamma,\beta,2,n,k}\ge u_{\epsilon}e^{(\beta-\epsilon)\Re\lambda_n}.
\end{equation}
Since $L^2(\gamma,\beta)$ is a Hilbert space,
it then follows that there exists a unique element in $\overline{\text{span}}(E_{\Lambda_{n,k}})$ in $L^2(\gamma,\beta)$,
that we denote by $\phi_{n,k}$,
so that
\[
||e_{n,k}-\phi_{n,k}||_{L^2 (\gamma,\beta)}=
\inf_{g\in \overline{\text{span}}(E_{\Lambda_{n,k}})}||e_{n,k}-g||_{L^2 (\gamma,\beta)}=D_{\gamma,\beta,2,n,k}.
\]
The function $e_{n,k}-\phi_{n,k}$ is orthogonal to all the elements of the closed span of $E_{\Lambda_{n,k}}$
in $L^2 (\gamma,\beta)$, hence to $\phi_{n,k}$ itself.
Therefore
\[
\langle e_{n,k}-\phi_{n,k}, e_{n,k}-\phi_{n,k} \rangle=\langle e_{n,k}-\phi_{n,k}, e_{n,k}\rangle.
\]
Hence
\[
(D_{\gamma,\beta,2,n,k})^2=\langle e_{n,k}-\phi_{n,k}, e_{n,k}\rangle.
\]
Next, we define
\[
r_{n,k}(x):=\frac{e_{n,k}(x)-\phi_{n,k}(x)}{(D_{\gamma,\beta,2,n,k}) ^2}.
\]
It then follows that $\langle r_{n,k}, e_{n,k}\rangle=1$ and $r_{n,k}$ is orthogonal to all the
elements of the system $E_{\Lambda_{n,k}}$.
Thus $\{r_{n,k}:\,\, n\in\mathbb{N},\,\, k=0,1,\dots,\mu_n-1\}$ is biorthogonal  to the system $E_{\Lambda}$.
Since $\phi_{n,k}\in\overline{\text{span}}(E_{\Lambda_{n,k}})$ in $L^2 (\gamma,\beta)$ then
$r_{n,k}\in\overline{\text{span}}(E_{\Lambda})$ in $L^2 (\gamma,\beta)$.

\begin{remark}
Clearly $||r_{n,k}||_{L^2 (\gamma,\beta)}=\frac{1}{D_{\gamma,\beta,2,n,k}}$,
hence from $(\ref{L2casebound})$ we obtain $(\ref{rnkbound})$.
\end{remark}

\subsection{Uniqueness and Optimality}

Next we show that $\{r_{n,k}\}$ is the unique biorthogonal sequence to the system $E_{\Lambda}$,
which belongs to its closed span in $L^2 (\gamma,\beta)$. Indeed, if there is another such biorthogonal sequence,
call it $\{q_{n,k}\}$, then for all $n\in\mathbb{N}$ and $k\in\{0,1,\dots,\mu_n-1\}$ we have
\[
\langle r_{n,k}-q_{n,k}, e_{m,l}\rangle=0, \qquad \forall\,\, m\in\mathbb{N}\quad\text{and}\quad l=0,1,\dots ,\mu_m-1.
\]
But this in turn implies that $r_{n,k}-q_{n,k}=0$ almost everywhere on $(\gamma,\beta)$ since the system $E_{\Lambda}$
is complete in its closed span in $L^2 (\gamma, \beta)$.

We also claim that if $\{v_{n,k}\}$ is any other sequence biorthogonal to the system $E_{\Lambda}$, then
\[
||r_{n,k}||_{L^2 (\gamma,\beta)}\le ||v_{n,k}||_{L^2 (\gamma,\beta)}.
\]
In other words, $r_{\Lambda}$ is optimal.

To justify this, choose an element $v_{n,k}$ and write $v_{n,k}=r_{n,k}+(v_{n,k}-r_{n,k})$.
Then $\langle v_{n,k}-r_{n,k}, e_{m,l}\rangle=0$ for all $e_{m,l}\in E_{\Lambda}$, thus
$\langle v_{n,k}-r_{n,k}, f\rangle=0$ for every $f$ which belongs to the closed span of $E_{\Lambda}$.
Hence $v_{n,k}-r_{n,k}$ belongs to the orthogonal complement of the closed span of $E_{\Lambda}$ in $L^2 (\gamma,\beta)$.
Therefore
\[
||v_{n,k}||^2_{L^2 (\gamma,\beta)}=||r_{n,k}||^2_{L^2 (\gamma,\beta)}+||v_{n,k}-r_{n,k}||^2_{L^2 (\gamma,\beta)}
\ge ||r_{n,k}||^2_{L^2 (\gamma,\beta)}.
\]

\subsection{The Fourier-type representation $(\ref{representationf})$}

Relations $(\ref{representationf})$ and $(\ref{representation})$ follow directly from Lemma $\ref{important}$
since by Theorem $\ref{theorem1}$ every function in the closed span of the system
$E_{\Lambda}$ in $L^2(\gamma, \beta)$ extends analytically as a Taylor-Dirichlet series.

\begin{lemma}\label{important}
Let the multiplicity sequence $\Lambda=\{\lambda_n,\mu_n\}_{n=1}^{\infty}$ belong to the $ABC$ class.
Suppose that a Taylor-Dirichlet series
\[
f(z)=\sum_{m=1}^{\infty}\left(\sum_{j=0}^{\mu_m-1} c_{m,j} z^j\right) e^{\lambda_m z}
\]
is analytic in the sector $\Theta_{\eta,\beta}$ $(\ref{opensector})$ and $f\in L^2 (\gamma, \beta)$. Then
\begin{equation}\label{coefcnk}
c_{n,k}=\langle f, r_{n,k} \rangle \qquad \forall\,\, n\in\mathbb{N}\quad \text{and}
\quad k=0,1,2,\dots ,\mu_n-1.
\end{equation}
\end{lemma}

\begin{proof}

We have
\begin{eqnarray}
\langle f, r_{n,k} \rangle =\int_{\gamma}^{\beta} f(x)\cdot \overline{r_{n,k}(x)}\, dx & = & \int_{\gamma}^{\beta} \overline{r_{n,k}(x)}\cdot
\sum_{m=1}^{\infty}\left(\sum_{j=0}^{\mu_m-1} c_{m,j}x^j\right)e^{\lambda_m x}
\, dx\nonumber\\
& = &
\int_{\gamma}^{\beta}\overline{r_{n,k}(x)} \cdot \sum_{m=1}^{n}\left(\sum_{j=0}^{\mu_m-1} c_{m,j}x^j\right)e^{\lambda_m x} \, dx\nonumber\\ & + &
\int_{\gamma}^{\beta}\overline{r_{n,k}(x)} \cdot \sum_{m=n+1}^{\infty}\left(\sum_{j=0}^{\mu_m-1} c_{m,j}x^j\right)e^{\lambda_m x}\, dx\nonumber\\
& = & c_{n,k}+
\int_{\gamma}^{\beta} \overline{r_{n,k}(x)}\cdot \,\sum_{m=n+1}^{\infty}\left(\sum_{j=0}^{\mu_m-1} c_{m,j}x^j\right)e^{\lambda_m x} dx,\label{cnk}
\end{eqnarray}
with the last step valid due to the biorthogonality.

We show below that
\begin{equation}\label{zerozero1}
\int_{\gamma}^{\beta} \overline{r_{n,k}(x)}\cdot \sum_{m=n+1}^{\infty}\left(\sum_{j=0}^{\mu_m-1} c_{m,j}x^j\right)e^{\lambda_m x}\, dx=0.
\end{equation}

Since $f\in L^2(\gamma,\beta)$, then the Taylor-Dirichlet series
\[
Q_n(z):=\sum_{m=n+1}^{\infty}\left(\sum_{j=0}^{\mu_m-1} c_{m,j}z^j\right)e^{\lambda_m z}
\]
also belongs to $L^2(\gamma,\beta)$. It then follows from Theorem $\ref{converse}$ that $Q_n$ belongs
to the closed span of the exponential system
\[
E_{\Lambda, n+1}:=\{x^ke^{\lambda_m x}:\, m\ge n+1,\,\, k=0,\dots,\mu_m-1\}
\]
in $L^2(\gamma,\beta)$. Hence, for every $\epsilon>0$, there is a function $f_{\epsilon}$ in the span of $E_{\Lambda, n+1}$
so that $||Q_n-f_{\epsilon}||_{L^2(\gamma,\beta)}<\epsilon$. Due to the biorthogonality we have
\[
\int_{\gamma}^{\beta} \overline{r_{n,k}(x)}\cdot f_{\epsilon}(x)\, dx=0.
\]
Combining with the Cauchy-Schwartz inequality we get
\begin{eqnarray*}
\left|\int_{\gamma}^{\beta} \overline{r_{n,k}(x)}\cdot Q_n(x)\, dx\right| & = & \left|\int_{\gamma}^{\beta} \overline{r_{n,k}(x)}\cdot (Q_n (x)-f_{\epsilon}(x))\, dx\right|\\
& \le & \epsilon\cdot ||r_{n,k}||_{L^2(\gamma,\beta)}.
\end{eqnarray*}
The arbitrary choice of $\epsilon$ implies that $(\ref{zerozero1})$ is true. Together with $(\ref{cnk})$
shows that $(\ref{coefcnk})$ holds.

\end{proof}

\subsection{Equal closures: $\overline{\text{span}}(r_{\Lambda})=\overline{\text{span}}(E_{\Lambda})\quad \text{in}\quad L^2(\gamma, \beta)$}

Clearly $\overline{\text{span}}(E_{\Lambda})$ in $L^2 (\gamma, \beta)$ is a separable Hilbert space.
So let us denote this space by $H_{\Lambda}$ and let
$S_{\Lambda}$ be the closed span of $r_{\Lambda}$ in $L^2 (\gamma, \beta)$.
Obviously $S_{\Lambda}$ is a subspace of $H_{\Lambda}$.
Let $S^{\perp}_{\Lambda}$ be the orthogonal complement  of $S_{\Lambda}$ in $H_{\Lambda}$, that is
\[
S^{\perp}_{\Lambda}=\{f\in H_{\Lambda}:\,\, \langle f, g \rangle =0\quad for\,\, all\,\, g\in S_{\Lambda}\}.
\]
Now, if $f\in H_{\Lambda}$ then as shown above
\[
f(t)=\sum_{n=1}^{\infty}\left(\sum_{k=0}^{\mu_n-1} \langle f, r_{n,k} \rangle \cdot t^k\right)e^{\lambda_n t},
\quad \text{almost everywhere on}\,\, (\gamma,\beta).
\]
But if $f\in S^{\perp}_{\Lambda}$ then $\langle f, r_{n,k} \rangle =0$ for all $r_{n,k}\in r_{\Lambda}$.
Thus, $f=0$ almost everywhere on $(\gamma, \beta)$, hence $S^{\perp}_{\Lambda}$ contains just the zero function.
Therefore $S_{\Lambda}=H_{\Lambda}$.

\subsection{The exponential system $E_{\Lambda}$ is hereditarily complete in its closure in $L^2(\gamma, \beta)$}

As above, let $H_{\Lambda}$ be $\overline{\text{span}}(E_{\Lambda})$ in $L^2 (\gamma, \beta)$.
Write the set
\[
\{(n,k):\,\, n\in\mathbb{N},\,\, k=0,1,\dots,\mu_n-1\}
\]
as an $\bf arbitrary$ disjoint union of two sets $N_1$ and $N_2$, that is,
\[
\{(n,k):\,\, n\in\mathbb{N},\,\, k=0,1,\dots,\mu_n-1\}=N_1\cup N_2, \qquad N_1\cap N_2=\emptyset.
\]
We will show that the closed span of the mixed system
\[
E_{1,2}:=\{e_{n,k}:\,\, (n,k)\in N_1\}\cup\{r_{n,k}:\,\, (n,k)\in N_2\},\qquad (e_{n,k}=x^ke^{\lambda_n x})
\]
is equal to $H_{\Lambda}$.

Denote by $W_{\Lambda_{1,2}}$ the closed span of $E_{1,2}$ in $L^2 (\gamma, \beta)$.
Obviously $W_{\Lambda_{1,2}}$ is a subspace of $H_{\Lambda}$.
Let $W^{\perp}_{\Lambda_{1,2}}$ be the orthogonal complement of $W_{\Lambda_{1,2}}$ in $H_{\Lambda}$, that is
\[
W^{\perp}_{\Lambda_{1,2}}=\{f\in H_{\Lambda}:\,\, \langle f, g \rangle =0\quad for\,\, all\,\, g\in W_{\Lambda_{1,2}}\}.
\]
Now, if $f\in H_{\Lambda}$ then by $(\ref{representationf})$ we have
\[
f(t)=\sum_{n=1}^{\infty}\left(\sum_{k=0}^{\mu_n-1} \langle f, r_{n,k} \rangle \cdot t^k\right)e^{\lambda_n t},
\quad \text{almost everywhere on}\,\, (\gamma,\beta).
\]
But if $f\in W^{\perp}_{\Lambda_{1,2}}$ then $\langle f, r_{n,k} \rangle =0$ for all $(n,k)\in N_2$.
Thus,
\[
f(t)=\sum_{(n,k)\in N_1}\left(\sum \langle f, r_{n,k} \rangle \cdot t^k\right)e^{\lambda_n t},
\quad \text{almost everywhere on}\,\, (\gamma,\beta).
\]
Since this Taylor-Dirichlet series $f$ is in $L^2 (\gamma, \beta)$, it then follows from Theorem $\ref{converse}$
that $f$ belongs to the closed span of the exponential system
\[
E_1:=\{e_{n,k}:\,\, (n,k)\in N_1\}
\]
in $L^2 (\gamma, \beta)$. Thus, for every $\epsilon>0$ there is a function $g_{\epsilon}$ in $\text{span}(E_1)$ so that
$||f-g_{\epsilon}||_{L^2(\gamma,\beta)}<\epsilon$. Then write
\[
\langle f, f \rangle = \langle f, f-g_{\epsilon} \rangle + \langle f, g_{\epsilon} \rangle.
\]
Since $f\in W^{\perp}_{\Lambda_{1,2}}$ then $\langle f, e_{n,k} \rangle =0$ for all $(n,k)\in N_1$, thus
$\langle f, g_{\epsilon} \rangle=0$. Therefore,
\[
\langle f, f \rangle = \langle f, f-g_{\epsilon}\rangle \le ||f||_{L^2(\gamma,\beta)}\cdot ||f-g_{\epsilon}||_{L^2(\gamma,\beta)}
\le ||f||_{L^2(\gamma,\beta)}\cdot \epsilon,
\qquad \text{hence}\qquad ||f||_{L^2(\gamma,\beta)}<\epsilon.
\]
Clearly this means that $f=0$ almost everywhere on $(\gamma, \beta)$, hence $W^{\perp}_{\Lambda_{1,2}}=\{0\}$.
Thus $W_{\Lambda_{1,2}}=H_{\Lambda}$, meaning that the exponential system $E_{\Lambda}$
is hereditarily complete in its closed span in $L^2(\gamma,\beta)$.

The proof of Theorem $\ref{biorthogonalsystem}$ is now complete.

\section{The Moment Problem: Proof of Theorem $\ref{MomentProblem}$}
\setcounter{equation}{0}

In this  section we prove Theorem $\ref{MomentProblem}$:
there exists a unique function $f\in\overline{\text{span}}(E_{\Lambda})$ in $L^2 (\gamma, \beta)$,
so that $f$ is a solution of the moment problem $(\ref{momentproblemquestion})$,
where the set $\{d_{n,k}:\, n\in\mathbb{N},\, k=0,1,\dots,\mu_n-1\}$ is
under the growth condition $(\ref{interpolationrnkcomplex})$.

First we deal with Uniqueness. Suppose that $f,\,\, g\in\overline{\text{span}}(E_{\Lambda})$ in $L^2 (\gamma, \beta)$
are both solutions. Then
\[
\int_{\gamma}^{\beta}(f(t)-g(t))\cdot t^k e^{\overline{\lambda_n} t}\, dt=0,\qquad \forall\,\, n\in\mathbb{N}\quad \text{and}
\quad k=0,1,2,\dots ,\mu_n-1.
\]
Completeness of $E_{\Lambda}$ in its closed span implies that $f=g$ a.e on $(\gamma, \beta)$.

Moreover, by $(\ref{representationf})$ in Theorem $(\ref{biorthogonalsystem})$,
this unique $f$ extends analytically in the sector $\Theta_{\eta, \beta}$ as a Taylor-Dirichlet series.
\[
\sum_{n=1}^{\infty}\left(\sum_{k=0}^{\mu_n-1} \langle f, r_{n,k} \rangle \cdot z^k\right)e^{\lambda_n z}.
\]

In addition, if $h\in L^2(\gamma,\beta)$ is a Taylor-Dirichlet series analytic in the same sector
and $h$ is a solution of the moment problem as well, it then follows from Theorem $\ref{converse}$
that $h\in\overline{\text{span}}(E_{\Lambda})$ in $L^2 (\gamma, \beta)$.
By the preceding remark, $h=f$ a.e on $(\gamma, \beta)$.
In other words, there can exist just one Taylor-Dirichlet series $f$ analytic in $\Theta_{\eta, \beta}$
with $f\in L^2 (\gamma, \beta)$, solving the moment problem.\\

\smallskip

Next we deal with the Existence of the solution $f$. We give two proofs
and both of them depend on the $r_{\Lambda}$ family of functions from Theorem $\ref{biorthogonalsystem}$.
We note that each $r_{n,k}\in r_{\Lambda}$ is defined almost everywhere on $(\gamma, \beta)$
on some subset call it $B_{n,k}$. Let
\[
(\gamma, \beta)^*:=\bigcap_{n\in\mathbb{N},\, k=0,1,\dots,\mu_n-1}B_{n,k}=\left(\bigcup_{n\in\mathbb{N},\,
k=0,1,\dots,\mu_n-1}B_{n,k}^c\right)^c
\]
where $B_{n,k}^c$ is the complement of $B_{n,k}$ in $(\gamma, \beta)$. Since the measure of $B_{n,k}^c$ is zero,
so is the measure of the set
\[
\bigcup_{n\in\mathbb{N},\, k=0,1,\dots,\mu_n-1}B_{n,k}^c
\]
since it a countable union of such sets. Therefore the measure of $(\gamma, \beta)^*$ is equal to the measure of $(\gamma, \beta)$.\\

In the proofs below we assume that $a$ in $(\ref{interpolationrnkcomplex})$ is a real number less than $\beta$.
The case $a=-\infty$ is treated in a similar way.

\subsection{Method I: A first proof to the Moment Problem}

Choose $\epsilon>0$ such that $\epsilon=(\beta-a)/6$. Then one has
\begin{equation}\label{boundforseriescomplex}
\sum_{n=1}^{\infty}e^{(-\beta+a+3\epsilon)\Re\lambda_n}<\infty.
\end{equation}
Next, for every $n\in\mathbb{N}$ let
\begin{equation}\label{Unt}
U_n(t):=\sum_{k=0}^{\mu_n-1}d_{n,k} r_{n,k}(t),\qquad t\in (\gamma,\beta)^*,
\end{equation}
It then follows from  relations $(\ref{rnkbound})$, $(\ref{interpolationrnkcomplex})$, $(\ref{munlambdan})$,
and $(\ref{boundforseriescomplex})$, that for every $\epsilon>0$, there are positive constants $M_{\epsilon}$ and $M^*_{\epsilon}$,
so that
\begin{eqnarray*}
\sum_{n=1}^{\infty}\int_{\gamma}^{\beta} |U_n(t)|\, dt &\le &
\sum_{n=1}^{\infty}\sum_{k=0}^{\mu_n-1}|d_{n,k}|\int_{\gamma}^{\beta} |r_{n,k}(t)|\, dt\\
& \le &\sum_{n=1}^{\infty}\sum_{k=0}^{\mu_n-1}|d_{n,k}|\cdot ||r_{n,k}(t)||_{L^2 (\gamma,\beta)}\cdot \sqrt{(\beta-\gamma)}\\
& \le & \sum_{n=1}^{\infty}\mu_n\cdot M_{\epsilon}e^{(a+\epsilon)\Re\lambda_n}\cdot e^{(-\beta+\epsilon)\Re\lambda_n}\cdot \sqrt{(\beta-\gamma)}\\
& \le & \sum_{n=1}^{\infty}M^*_{\epsilon}e^{(-\beta+a+3\epsilon)\Re\lambda_n}\\
& < & \infty.
\end{eqnarray*}
By \cite[Theorem 1.38]{Rudin} this implies that the infinite series
\[
U(t):=\sum_{n=1}^{\infty}U_n(t)=\sum_{n=1}^{\infty}\left(\sum_{k=0}^{\mu_n-1}d_{n,k} r_{n,k}(t)\right),\qquad t\in (\gamma, \beta)^*
\]
converges pointwise almost everywhere on $(\gamma,\beta)^*$ and
belongs to $L^1(\gamma,\beta)$.

In fact, $U\in L^2 (\gamma, \beta)$ as well: by the Minkowski inequality, and once again using
relations $(\ref{rnkbound})$, $(\ref{interpolationrnkcomplex})$, $(\ref{munlambdan})$, and $(\ref{boundforseriescomplex})$, we get
\[
\int_{\gamma}^{\beta} |U(t)| ^2\, dt  \le
\left(\sum_{n=1}^{\infty}\sum_{k=0}^{\mu_n-1}|d_{n,k}|\cdot ||r_{n,k}||_{L^2 (\gamma,\beta)}\right)^2
<  \infty.
\]
Needless to say that the function
\[
h(t):=\sum_{n=1}^{\infty}\left(\sum_{k=0}^{\mu_n-1}|d_{n,k}|\cdot |r_{n,k}(t)|\right)\qquad \forall\,\, t\in (\gamma, \beta)^*
\]
also converges pointwise almost every everywhere on $(\gamma, \beta)$ and $h\in L^1 (\gamma, \beta)\cap L^2 (\gamma, \beta)$.

We also claim that $U$ converges in $L^2 (\gamma, \beta)$. Indeed, we have
$|U(t)-\sum_{n=1}^{M} U_n(t)|^2\le h^2(t)$ and $|U(t)-\sum_{n=1}^{M} U_n(t)|^2\to 0$ as $M\to\infty$
for almost all $t\in (\gamma, \beta)^*$.
Then by the Lebesgue Convergence Theorem we have
\begin{equation}\label{tozeroU}
||U-\sum_{n=1}^{M}U_n||_{L^2 (\gamma,\beta)}\to 0\qquad\text{as}\qquad M\to\infty.
\end{equation}

Next, since each $U_n$ is a finite sum of $r_{n,k}$'s, then each $U_n$ belongs to the closed span of $E_{\Lambda}$ in $L^2 (\gamma,\beta)$.
Together with  $(\ref{tozeroU})$ yields that
$U$ also belongs to the closed span of $E_{\Lambda}$ in $L^2 (\gamma,\beta)$.

It remains to show that $U$ is a solution of the moment problem. If we fix some $m\in\mathbb{N}$ and some
$l=0,1,\dots,\mu_m-1$, then
$\sum_{n=1}^{M}U_n(t)\cdot t^l e^{\overline{\lambda_m} t}\to U(t)\cdot t^l e^{\overline{\lambda_m} t}$ almost everywhere on $(\gamma,\beta)$ as $M\to\infty$.
Also there is a positive constant $T_{m,l}$ which depends on $m$ and $l$, so that
\[
|t^{l}e^{\overline{\lambda_m} t}\cdot \sum_{n=1}^{M} U_n(t)|\le T_{m,l}\cdot h(t),\quad \forall\,\, t\in (\gamma,\beta),\quad \forall\,\, M\in\mathbb{N}.
\]
Using again the Lebesgue Convergence Theorem gives
\[
\lim_{M\to\infty}\int_{\gamma}^{\beta}t^{l}e^{\overline{\lambda_m} t}\cdot \sum_{n=1}^{M} U_n (t)\, dt=
\int_{\gamma}^{\beta}U(t)\cdot t^l e^{\overline{\lambda_m} t}\, dt.
\]
By the biorthogonality of the $\{r_{n,k}\}$ family to the system $E_{\Lambda}$ and $(\ref{Unt})$, we finally get
\[
\int_{\gamma}^{\beta} U(t)\cdot t^{l}e^{\overline{\lambda_m} t}\, dt=d_{m,l}.
\]
Thus $U$ is a solution to the Moment Problem $(\ref{momentproblemquestion})$.
The first proof is now complete.

\subsection{Method II: A second proof to the Moment Problem via Nonharmonic Fourier Series}

We shall make use of the following notions from Non-Harmonic Fourier Series.

Let $\cal{H}$ be a separable Hilbert space endowed with an inner product $\langle \,\cdot \, \rangle$, and consider a sequence $\{f_n\}_{n=1}^{\infty}\subset \cal{H}$. We say that (see \cite[p. 128 Definition]{Young}):

(i) $\{f_n\}_{n=1}^{\infty}$ is a Bessel sequence if there exists a constant $B>0$ such that
$\sum_{n=1}^{\infty}|\langle f,f_n\rangle|^2<B||f||^2$ for all $f\in \cal{H}$.

(ii) $\{f_n\}_{n=1}^{\infty}$ is a Riesz-Fischer sequence
if the moment problem $\langle f, f_n\rangle =c_n$ has at least one solution in $\cal{H}$
for every sequence $\{c_n\}$ in the space $l^2(\mathbb{N})$.

The following result stated by Casazza et al.
is an interesting connection between Bessel and Riesz-Fischer sequences.

\begin{propa}\label{casazza}\cite[Proposition 2.3, (ii)]{Casazza}

The Riesz-Fischer sequences in $\cal{H}$ are precisely the families for which a biorthogonal Bessel sequence exists.
In other words

(a) Suppose that two sequences $\{f_n\}_{n=1}^{\infty}$ and $\{g_n\}_{n=1}^{\infty}$ in $\cal{H}$
are biorthogonal. Suppose also that $\{f_n\}_{n=1}^{\infty}$ is a Bessel sequence.
Then $\{g_n\}_{n=1}^{\infty}$ is a Riesz-Fischer sequence.

(b) If $\{f_n\}_{n=1}^{\infty}$ in $\cal{H}$ is a Riesz-Fischer sequence,
then there exists a biorthogonal Bessel sequence $\{g_n\}_{n=1}^{\infty}$.

\end{propa}

And now a sufficient condition so that two biorthogonal families in $\cal{H}$ are Bessel and Riesz-Fischer sequences.
The result follows from \cite[Proposition 3.5.4]{Christensen} and Proposition $\bf A$.
\begin{lemma}\label{besselriesz}
Consider two biorthogonal sequences $\{u_n\}_{n=1}^{\infty}$ and $\{v_n\}_{n=1}^{\infty}$ in $\cal{H}$ and
suppose there is some $M>0$ so that
\[
\sum_{n=1}^{\infty} |\langle v_n, v_m\rangle  |<M\qquad \text{for\,\, all}\quad m=1,2,3,\dots.
\]
Then $\{v_n\}_{n=1}^{\infty}$ is a Bessel sequence in $\cal{H}$ and
$\{u_n\}_{n=1}^{\infty}$ is a Riesz-Fischer sequence in $\cal{H}$.
\end{lemma}

Let us now proceed with the second proof of Theorem $\ref{MomentProblem}$.
As before, let $H_{\Lambda}$ be $\overline{\text{span}}(E_{\Lambda})$ in $L^2 (\gamma, \beta)$,
hence $H_{\Lambda}$ is a separable Hilbert space.

Let $\{d_{n,k}:\,\, n\in\mathbb{N},\,\, k=0,1,\dots,\mu_n-1\}$ be the sequence of non-zero complex numbers that satisfies
$(\ref{interpolationrnkcomplex})$. For every $n\in\mathbb{N}$ and $k=0,1,\dots,\mu_n-1$, define
\[
U_{n,k}(t):=\lambda_nd_{n,k}r_{n,k}(t)\qquad\text{and}\qquad V_{n,k}(t):=\frac{t^ke^{\lambda_n t}}{\overline{\lambda_n}\cdot\overline{d_{n,k}}}.
\]
It then easily follows that the sets
\[
\{U_{n,k}:\,\, n\in\mathbb{N},\,\, k=0,1,\dots,\mu_n-1\}\quad\text{and}\quad\{V_{n,k}:\,\, n\in\mathbb{N},\,\, k=0,1,\dots,\mu_n-1\}
\]
are biorthogonal in $H_{\Lambda}$.\\

We claim that $\{U_{n,k}\}$ is a Bessel sequence and $\{V_{n,k}\}$ is a Riesz-Fischer sequence in $H_{\Lambda}$.
Indeed, by $(\ref{interpolationrnkcomplex})$ and $(\ref{rnkbound})$ it follows that for $\epsilon =(\beta -a)/6$
there is a positive constant $m_{\epsilon}$ so that
\[
||U_{n,k}||_{L^2(\gamma,\beta)}\le m_{\epsilon}e^{(-\beta+a+3\epsilon)\Re\lambda_n}.
\]
By the Cauchy-Schwartz inequality we get
\begin{equation}\label{operatorbound}
|\langle U_{n,k},U_{m,j}\rangle|\le m_{\epsilon}e^{(-\beta+a+3\epsilon)\Re\lambda_n}\cdot e^{(-\beta+a+3\epsilon)\Re\lambda_m}.
\end{equation}
If we denote by $C_{n,k,m,j}$ the values of $\langle U_{n,k},U_{m,j}\rangle$ and by $C$ the infinite dimensional hermitian matrix with entries the $C_{n,k,m,j}$,
then $C$ is the Gram matrix associated with $\{U_{n,k}\}$. From the choice of $\epsilon$ and relations $(\ref{munlambdan})$ and $(\ref{operatorbound})$, we get
\[
\sum_{n=1}^{\infty}\sum_{k=0}^{\mu_n-1}\sum_{m=1}^{\infty}\sum_{j=0}^{\mu_m-1}|C_{n,k,m,j}|<\infty.
\]
It then follows from Lemma $\ref{besselriesz}$ that $\{U_{n,k}\}$  is a Bessel sequence in $H_{\Lambda}$ and
its biorthogonal sequence $\{V_{n,k}\}$ is a Riesz-Fischer sequence in $H_{\Lambda}$.\\

Therefore, the moment problem
\[
\int_{\gamma}^{\beta}f(t)\cdot \overline{V_{n,k}(t)}\, dt=c_{n,k} \qquad \forall\,\, n\in\mathbb{N}\quad \text{and}\quad k=0,1,2,\dots ,\mu_n-1,
\]
has a solution $f$ in $H_{\Lambda}$ whenever $\sum_{n=1}^{\infty}\sum_{k=0}^{\mu_n-1}|c_{n,k}|^2<\infty$.
Since $\Lambda$ satisfies the condition $(\ref{convergencecondition})$, then we can take $c_{n,k}=1/\lambda_n$ for all $n\in\mathbb{N}$ and $k=0,1,\dots,\mu_n-1$.
Hence, recalling the definition of $V_{n,k}$, there is some function $f\in H_{\Lambda}$ so that
\[
\int_{\gamma}^{\beta}f(t)\cdot \left( \frac{t^ke^{\overline{\lambda_n} t}}{d_{n,k}\lambda_n}\right) \, dt=\frac{1}{\lambda_n}\qquad \forall\,\, n\in\mathbb{N}\quad \text{and}\quad k=0,1,2,\dots ,\mu_n-1.
\]
Thus
\[
\int_{\gamma}^{\beta}f(t)\cdot  t^ke^{\overline{\lambda_n} t} \, dt=d_{n,k}\qquad \forall\,\, n\in\mathbb{N}\quad
\text{and}\quad k=0,1,2,\dots ,\mu_n-1,
\]
hence obtaining a solution $f\in H_{\Lambda}$ to the moment problem. The second proof is now complete.

\section{The solution space of the Carleson differential equation: Proof of Theorem $\ref{CarlesonTheorem}$ and a counterexample}
\setcounter{equation}{0}

In this section we prove Theorem $\ref{CarlesonTheorem}$ which is on the solution space
of the differential equation $(\ref{CarlesonLeontevEquation})$.
We first prove Part $(A)$ and then Part $(B)$.
We also present an example where the groupings in $(\ref{solutionTDS})$ cannot be dropped
in case $\Lambda$ is not an interpolating variety for the space $A^0_{|z|}$.

\subsection{Proof of Theorem $\ref{CarlesonTheorem}$}

\subsubsection{Proof of Part A}

Let $f$ be a function in the class $C(\gamma,\beta,\{G_n\})$ (Definition $\ref{The Class}$)
such that it is a solution of the equation $(\ref{CarlesonLeontevEquation})$.
Carleson  and  Leont' ev  proved that $f$
extends analytically in the sector $\Theta_{\eta, \beta}$ $(\ref{opensector})$ as a series with grouping of terms,
(see $(\ref{solutionTDS})$), converging uniformly on compact subsets of the sector $\Theta_{\eta, \beta}$, thus
on closed bounded intervals $[\gamma+\epsilon, \beta-\epsilon]$
for every small $\epsilon>0$. Clearly this means that  for every $j\in \mathbb{N}$,
$f$ belongs to the closed span of the system $E_{\Lambda}$ in the space
$L^2 (\gamma+\epsilon_j, \beta-\epsilon_j)$ where $\epsilon_j=(\beta-\gamma)/4j$.
From Theorem $\ref{theorem1}$ we see that for every fixed $j\in\mathbb{N}$,
$f$ extends as an analytic function in the open sector
\[
\Theta_{\eta ,\beta-\epsilon_j}:=\left\{z:\left|\frac{\Im z}{\Re (z-\beta+\epsilon_j)}\right|\le \frac{1}{\tan\eta},\,\,\Re z<\beta-\epsilon_j\right\}
\]
and admits a unique Taylor-Dirichlet series representation of the form
\[
f_j(z)=\sum_{n=1}^{\infty}\left(\sum_{k=0}^{\mu_n-1}c_{j,n,k}
z^k\right) e^{\lambda_n z},\qquad \forall\,\, z\in \Theta_{\eta,\beta-\epsilon_j}
\]
converging uniformly on every compact set of $\Theta_{\eta,\beta-\epsilon_j}$.

Now, observe that $\Theta_{\eta,\beta-\epsilon_1}\subset\Theta_{\eta,\beta-\epsilon_{j}}$ for all $j\in\mathbb{N}$.
This means that for all $j\in\mathbb{N}$, $f_1(z)=f_j(z)$ for all $z$ in the sector $\Theta_{\eta,\beta-\epsilon_1}$.
Hence, by the Uniqueness result of Lemma $\ref{uniquenessDirichlet}$ (see the Appendix) the respective coefficients of the
$f_j$ series are identical. In other words $c_{1,n,k,}=c_{j,n,k}$ for all $j\in\mathbb{N}$.
Clearly this means that $f$ extends as an analytic function in the open sector $\Theta_{\eta,\beta}$
and admits in the sector the Taylor-Dirichlet series representation
\[
\sum_{n=1}^{\infty}\left(\sum_{k=0}^{\mu_n-1}c_{1,n,k} z^k\right) e^{\lambda_n z}.
\]

Alternatively, since $f(z)=f_1(z)$ on the sector $\Theta_{\eta,\beta-\epsilon_1}$ and $f$ is analytic
in the larger sector $\Theta_{\eta,\beta}$, by \cite[Theorem 4.1]{Z2015JMAA} we know that
$f$ is either an entire function or it has a Natural Boundary. This implies that
$f$ admits the $f_1$ series representation not only in $\Theta_{\eta,\beta-\epsilon_1}$ but in $\Theta_{\eta,\beta}$ as well.

\subsubsection{Proof of Part B}

The proof is essentially the one given in \cite[Theorem 4.1]{Z2012CMFT}
but we will rewrite it in a more clear and coherent way.
First we prove two results, Lemmas $\ref{classC}$ and $\ref{equivalentsystems1}$.
\begin{lemma}\label{classC}
Let $\Omega$ be an open region in the complex plane such that $(\gamma, \beta)\subset\Omega$.
Let a function $p$ be analytic on $\Omega$ and let $F, \,\, G$ be the infinite products as in $(\ref{F})$.
Then $p$ belongs to the class $C(\gamma, \beta, \{G_n\})$ and both $F(D)f(z)$ and $G(D)f(z)$ define analytic fucntions on $\Omega$.
\end{lemma}

\begin{proof}
By \cite[Proposition 6.4.2]{Berenstein}, $F(D)p(z)$ and $G(D)p(z)$ define analytic fucntions on $\Omega$.
In order to show that $p\in C(\gamma, \beta, \{G_n\})$ as well, we have to prove that
\[
\sum_{n=1}^{\infty}\frac{G^{(n)}(0)}{n!}\cdot |p^{(n)}(x)|
\]
converges uniformly on the interval $[\gamma+\delta, \beta-\delta]$ for every small $\delta>0$.

Let us first consider the Taylor series expansion of $G$ about zero
\[
G(z)=\sum_{n=0}^{\infty}\frac{G^{(n)}(0)}{n!}z^n.
\]
We recall from Remark $\ref{allpositive}$ that $G^{(n)}(0)$ is positive for all $n\ge 0$.

Now, since $G$ is an entire function of exponential type $0$,
the relation between the type and the coefficients of its power series is given by
\[
\limsup_{n\to\infty}n\left(\frac{G^{(n)}(0)}{n!}\right)^{1/n}=0.
\]
Thus for every $\epsilon>0$, there is a $m\in\mathbb{N}$ so that
\begin{equation}\label{FandG}
G^{(n)}(0)\le\frac{\epsilon^n\cdot n!}{n^n}\qquad \forall\,\, n\ge m.
\end{equation}

Next, since $\Omega$ contains the interval $(\gamma, \beta)$, then it contains the intervals
$[\gamma+\delta, \beta-\delta]$ for every small $\delta>0$. Fix such an interval.
Clearly it is properly contained in some compact set $K\subset \Omega$. In fact there is some $R>0$, so that
for any point $\zeta\in [\gamma+\delta, \beta-\delta]$, the disk $D_{j,R}:=\{z:\,\, |z-\zeta|\le R$ is a subset of
$K\subset \Omega$. We let $\partial D_{j,R}$ to be the boundary of the disk $D_{j,R}$.

Now, since $p$ is analytic on $\Omega$, then $p$ is bounded on $K$,
thus there is some $M_K>0$, so that $|p(z)|\le M_K$ for all $z\in K$.
Then by the Cauchy Integral formula
\[
\frac{p^{(n)}(\zeta)}{n!}=\frac{1}{2\pi i}\int_{\partial D_{j,R}}\frac{p(z)}{(z-\zeta)^{n+1}}\, dz,\qquad n=1,2,3,\cdots.
\]
Thus for $n=1,2,3,\dots$, one gets
\[
\frac{|p^{(n)}(\zeta)|}{n!}\le \frac{M_K}{R^n}\qquad \text{for all}\qquad \zeta\in [\gamma+\delta, \beta-\delta].
\]
Combining with $(\ref{FandG})$ and taking $\epsilon<R$, gives for $n=1,2,3,\dots$,
\[
\frac{G^{(n)}(0)}{n!}\cdot |p^{(n)}(\zeta)|\le \frac{M_K\cdot n!}{n^n}\qquad \text{for all}
\qquad \zeta\in [\gamma+\delta, \beta-\delta].
\]
Obviously the series $\sum_{n=1}^{\infty}\frac{n!}{n^n}$ converges, hence the series
\[
\sum_{n=1}^{\infty}\frac{G^{(n)}(0)}{n!}\cdot |p^{(n)}(\zeta)|
\]
converges uniformly on the interval $[\gamma+\delta, \beta-\delta]$.
The arbitrary choice of $\delta$ means that $p$ belongs to the class $C(\gamma,\beta,\{G_n\})$.
\end{proof}
The following result is known. A short proof follows.
\begin{lemma}\label{equivalentsystems1}
Let $f$ be a function analytic on $\Omega$. Let $F(z)$ be the infinite product as in $(\ref{F})$
and let $\gamma(z)$ be its Borel Transform. Let $\zeta\in \Omega$ be an arbitrary point. Then there is a
circle $\partial B_{\rho_{\zeta}/2}$ so that
\[
\frac{1}{2\pi i}\int_{\partial B_{\rho_{\zeta}/2}}\gamma(z)f(z+\zeta)\,
dz=F(D)f(\zeta)\quad \text{where}\quad D=\frac{d}{dz}.
\]
\end{lemma}

\begin{proof}
First we note that since the entire function $F$ is of exponential type zero,
then by the P\'{o}lya representation theorem for entire functions of exponential type \cite[Theorem 5.3.5]{Boas} we have
\begin{equation}\label{Polyarepresentationtheorem}
F(z)=\frac{1}{2\pi i}\int_{l} \gamma(w) e^{zw} dw
\end{equation}
where $\gamma$ is the Borel Transform of $F$ and $l$ is any simple closed rectifiable curve enclosing the Origin.

By differentiation one gets
\begin{equation}\label{Polyarepresentationtheorem1}
F^{(n)}(0)=\frac{1}{2\pi i}\int_{l} \gamma(w) w^n dw,
\end{equation}
and since $F$ vanishes on $\Lambda=\{\lambda_n,\mu_n\}_{n=1}^{\infty}$ we get
\begin{equation}\label{Polyarepresentationtheorem2}
\int_{l} \gamma(w) w^k e^{\lambda_n w} dw=0,\qquad n\in\mathbb{N},\qquad k=0,1,2,\dots,\mu_n-1.
\end{equation}
Consider now a function $f$ analytic on $\Omega$ and a point $\zeta\in \Omega$.
The Taylor series expansion of $f$ about $\zeta$ has radius of convergence equal to some positive number,
call it $\rho_{\zeta}>0$. Therefore the series
\[
\sum_{n=1}^{\infty}\frac{f^{(n)}(\zeta)}{n!}\cdot (z-\zeta)^n
\]
converges uniformly with respect to $z$ to $f(z)$ on the closed disk $\{z:\,\, |z-\zeta|\le \rho_{\zeta}/2\}$.
Combined with the continuity of $\gamma (z)$ on curves enclosing the origin, yields that the series
\[
\gamma(z)\cdot f(z+\zeta)=\sum_{n=1}^{\infty}\gamma(z)\cdot \frac{f^{(n)}(\zeta)}{n!}\cdot z^n
\]
converges uniformly with respect to $z$ on the closed disk
\[
B_{\rho_{\zeta}/2}:=\{z:\,\, |z|\le \rho_{\zeta}/2\}.
\]
Applying $(\ref{Polyarepresentationtheorem1})$ as well, one has
\begin{eqnarray*}
\frac{1}{2\pi i}\int_{\partial B_{\rho_{\zeta}/2}} \gamma(z)f(z+\zeta)\, dz
& = &
\frac{1}{2\pi i}\int_{\partial B_{\rho_{\zeta}/2}} \gamma(z)\cdot \sum_{n=0}^{\infty}\frac{f^{(n)}(\zeta)}{n!} z^{n}\, dz\\
& = &
\sum_{n=0}^{\infty}\frac{f^{(n)}(\zeta)}{n!}\left(\frac{1}{2\pi i}\int_{\partial B_{\rho_{\zeta}/2}} \gamma(z) z^{n}\,
dz\right)\\
& = &
\sum_{n=0}^{\infty}\frac{F^{(n)}(0)}{n!}f^{(n)}(\zeta)\\
& = & F(D)f(\zeta).
\end{eqnarray*}
\end{proof}

Finally, consider a Taylor-Dirichlet series
\[
f(z)=\sum_{n=1}^{\infty}\left(\sum_{k=0}^{\mu_n-1}c_{n,k} z^k\right) e^{\lambda_n z},
\]
which is analytic in the sector $\Theta_{\eta,\beta}$, thus converging uniformly on its compact subsets.
Let $\zeta$ be an arbitrary point in $(\gamma,\beta)$ and as in the previous lemma, consider the
disk $\{z:\,\, |z-\zeta|< \rho_{\zeta}\}$ where $\rho_{\zeta}$ is the radius of convergence of the Taylor series of $f$
about the point $\zeta$. Clearly the Taylor-Dirichlet series converges uniformly on the closed disk
$\{z:\,\, |z-\zeta|\le \rho_{\zeta}/2\}$. Thus, the series
\[
\gamma(z)\cdot f(z+\zeta)= \sum_{n=1}^{\infty}\left(\sum_{k=0}^{\mu_n-1}\gamma(z)\cdot c_{n,k}
(z+\zeta)^k\right) e^{\lambda_n (z+\zeta)}
\]
converges uniformly on the circle $\partial B_{\rho_{\zeta}/2}=\{z:\,\, |z|= \rho_{\zeta}/2\}$.
Together with Lemma $\ref{equivalentsystems1}$, we get
\begin{eqnarray*}
2\pi i \cdot F(D)f(\zeta) & = & \int_{\partial B_{\rho_{\zeta}/2}}\gamma(z)\cdot f(z+\zeta)\, dz\\
& = & \int_{\partial B_{\rho_{\zeta}/2}}
\gamma(z) \cdot \sum_{n=1}^{\infty}\left(\sum_{k=0}^{\mu_n-1}c_{n,k}
(z+\zeta)^k\right) e^{\lambda_n (z+\zeta)}\, dz\\
& = & \int_{\partial B_{\rho_{\zeta}/2}}\gamma(z)\cdot \sum_{n=1}^{\infty}\left(\sum_{k=0}^{\mu_n-1}c_{n,k}
\sum_{j=0}^{k}\frac{k!}{j!(k-j)!}\zeta^{k-j}z^j\right) e^{\lambda_n (z+\zeta)}\, dz\\\
& = & \sum_{n=1}^{\infty}\left(\sum_{k=0}^{\mu_n-1}
\left(\sum_{j=0}^{k}\frac{k!}{j!(k-j)!}\zeta^{k-j}
\cdot \int_{\partial B_{\rho_{\zeta}/2}}\gamma(z)z^j e^{\lambda_n z}\, dz\right)
c_{n,k}\right) e^{\lambda_n \zeta}
\end{eqnarray*}
It then follows from the integrals in $(\ref{Polyarepresentationtheorem2})$ that
\[
F(D)f(\zeta)=0.
\]
The arbitrary choice of $\zeta\in (\gamma, \beta)$ means that equation $(\ref{CarlesonLeontevEquation})$ is true.
The second part of the proof of Theorem $\ref{CarlesonTheorem}$ is now complete.

\subsection{A Counterexample where the groupings in $(\ref{solutionTDS})$ cannot be dropped}

Consider the sequence $\Lambda=\{\lambda_n\}_{n=1}^{\infty}$ where
\begin{equation}\label{speciall}
\lambda_{2n-1}=n^2\quad \text{and}\quad \lambda_{2n}=n^2+e^{-n^4}.
\end{equation}
Clearly the conditions $A$ $(\ref{convergencecondition})$ and $B$ $(\ref{lessthan})$ hold. However,
it follows from the Geometric Conditions $(II)$ in subsection 2.1 and  relation $(\ref{mun=1})$
that $\Lambda=\{\lambda_n\}_{n=1}^{\infty}$ is not an interpolating variety for the space of entire functions of exponential type zero.
Hence this $\Lambda$ is not in the $ABC$ class.

Consider then the entire functions of exponential type zero,
\[
F(z)=\prod_{n=1}^{\infty}\left(1-\frac{z}{\lambda_n}\right)\quad
\text{and}
\quad
G(z)=\prod_{n=1}^{\infty}\left(1+\frac{z}{|\lambda_n|}\right).
\]
We have
\[
F(z)=\frac{1}{2\pi i}\int_{l} \gamma(w) e^{zw} dw
\]
where $\gamma$ is the Borel Transform of $F$ and $l$ is any simple closed rectifiable curve enclosing the Origin.

Since $F$ vanishes on $\Lambda=\{\lambda_n\}_{n=1}^{\infty}$ we get
\begin{equation}\label{Polyarepresentationtheorem2alone}
\int_{l} \gamma(w) e^{\lambda_n w} dw=0,\qquad n\in\mathbb{N}.
\end{equation}
We will now give an example of a series
$\sum_{n=1}^{\infty}\left( a_{2n-1}\cdot e^{\lambda_{2n-1} z} - a_{2n}\cdot e^{\lambda_{2n} z}\right)$
which is analytic in the half-plane $\Re z<0$ and
satisfies  the Carleson equation $(\ref{CarlesonLeontevEquation})$ on any interval $(\gamma, \beta)$ on the half-line $\Re x<0$,
whereas the series $\sum_{n=1}^{\infty} a_{n}\cdot e^{\lambda_{n} z}$ itself diverges everywhere in $\mathbb{C}$.

Clearly the following two series diverge everywhere in the complex plane
\[
\sum_{n=1}^{\infty} e^{n ^3}\cdot e^{z n^2}\qquad \text{and}\qquad
\sum_{n=1}^{\infty} e^{n ^3}\cdot e^{z\cdot (n^2+e^{-n^4})}.
\]
The same is true for the series $\sum_{n=1}^{\infty}a_n\cdot e^{\lambda_n z}$ where $\lambda_n$ are as in $(\ref{speciall})$
and
\[
a_{2n-1}= e^{n ^3}\quad \text{and}\quad a_{2n}= -e^{n ^3}.
\]
However, we claim that if we regroup the $a_n\cdot e^{\lambda_n z}$ terms, as
$(a_{2n-1}\cdot e^{z\lambda_{2n-1}}-a_{2n}\cdot e^{z\lambda_{2n}})$, hence having
\begin{equation}\label{functioncounter}
f(z):=\sum_{n=1}^{\infty}\left( e^{n ^3}\cdot e^{zn^2} - e^{n ^3}\cdot e^{z\cdot (n^2+e^{-n^4})}\right)=
\sum_{n=1}^{\infty} e^{n ^3}\cdot\left(e^{zn^2}-e^{z\cdot (n^2+e^{-n^4}) }\right),
\end{equation}
then this series is analytic in the left half-plane.
Indeed, first we write
\[
e^{n ^3}\cdot\left(e^{zn^2}-e^{z\cdot (n^2+e^{-n^4}) }\right) =
e^{n ^3}\cdot e^{z n^2}\cdot \left(1-e^{z\cdot e^{-n^4}}\right).
\]

Then, clearly for all $z$ in the closed left half-plane $\Re z\le 0$, we have $|e^{z\cdot n^2}|\le 1$.
Using the Maclaurin series of the exponential function $e^{z\cdot e^{-n^4}}$
gives
\begin{eqnarray}
|e^{z n^2}|\cdot e^{n ^3}\cdot
|1-e^{z\cdot e^{-n^4}}| & = & |e^{z n^2}|\cdot e^{n ^3}\cdot |z\cdot e^{-n^4} + \frac{z^2\cdot e^{-2n^4}}{2} + \frac{z^3\cdot e^{-3n^4}}{3!} +
\frac{z^4\cdot e^{-4n^4}}{4!} + \dots | \nonumber\\
& < & e^{n ^3}\cdot e^{-n^4} \left(|z| + \frac{|z|^2}{2} + \frac{|z|^3\cdot }{3!} + \frac{|z|^4}{4!} + \dots |\right) \nonumber \\
& < & e^{n ^3-n^4}\cdot e^{|z|}.\nonumber
\end{eqnarray}

Replacing in $(\ref{functioncounter})$ shows that there is some positive constant $A$, so that $|f(z)|<Ae^{|z|}$ for all $z$
in the closed half-plane $\Re z\le 0$, hence
the series $f$ converges uniformly on compact subsets of $\Re z\le 0$,
thus it is analytic in $\Re z< 0$.\\

Now, by Lemma $\ref{classC}$, $f$ belongs to the class $C(\gamma, \beta, \{G_n\})$
for any interval $(\gamma, \beta)$ which is on the semi-axis $x<0$.
Choose an arbitrary $\zeta\in (\gamma, \beta)$ and consider the circle $\partial B_{\rho_{\zeta}/2}$
as in the previous subsection. Then, like before, we get

\begin{eqnarray*}
2\pi i \cdot F(D)f(\zeta) & = & \int_{\partial B_{\rho_{\zeta}/2}}\gamma(z)\cdot f(z+\zeta)\, dz\\
& = & \int_{\partial B_{\rho_{\zeta}/2}}
\gamma(z) \cdot \sum_{n=1}^{\infty} e^{n ^3}\cdot\left(e^{(z+\zeta)\lambda_{2n-1}}-e^{(z+\zeta)\lambda_{2n}}\right)\, dz\\
& = & \sum_{n=1}^{\infty} e^{n ^3}\cdot \int_{\partial B_{\rho_{\zeta}/2}}
\gamma(z) \cdot \left(e^{(z+\zeta)\lambda_{2n-1}}-e^{(z+\zeta)\lambda_{2n}}\right)\, dz\\
& = & \sum_{n=1}^{\infty} e^{n ^3}\cdot
\left(
e^{\zeta\lambda_{2n-1}}\cdot \int_{\partial B_{\rho_{\zeta}/2}} \gamma(z) \cdot e^{z\lambda_{2n-1}}\, dz
-
e^{\zeta\lambda_{2n}}\cdot \int_{\partial B_{\rho_{\zeta}/2}} \gamma(z) \cdot e^{z\lambda_{2n}}\, dz
\right)
\end{eqnarray*}
It follows from $(\ref{Polyarepresentationtheorem2alone})$ that $F(D)f(\zeta)=0$.
Since $\zeta\in (\gamma, \beta)$ was chosen arbitrarily, then $(\ref{CarlesonLeontevEquation})$ is true.

\section{Revisiting Theorems $\ref{converse}$ and $\ref{CarlesonTheorem}$}
\setcounter{equation}{0}

Given $\Lambda\in ABC$, by Lemma $\ref{TDSeries}$  a Taylor-Dirichlet series
\[
g(z)=\sum_{n=1}^{\infty}\left(\sum_{k=0}^{\mu_n-1} c_{n,k}z^k\right)e^{\lambda_n z},\quad c_{n,k}\in\mathbb{C},
\]
defines an analytic function  in the open sector $\Theta_{\eta ,\beta}$ $(\ref{opensector})$,
converging uniformly on its compact subsets, if and only if
the coefficients $c_{n,k}$ satisfy the upper bound $(\ref{cnkbound})$.

For $\Lambda\in ABC$, suppose now that a Taylor-Dirichlet series converges $\bf pointwise$ on an interval $(\gamma, \beta)$.
Can we draw any conclusions from this? Does pointwise convergence on $(\gamma, \beta)$ imply uniform convergence
on subintervals of $(\gamma, \beta)$? Does it also imply analytic continuation to the sector $\Theta_{\eta ,\beta}$ $(\ref{opensector})$?
The answer is affirmative (see Lemma $\ref{pointwiselemma}$) assuming that $\Lambda$ satisfies the additional condition

$(D)$\footnote{The author was not able to prove whether this condition is redundant or not assuming that $\Lambda\in ABC$}
\begin{equation}\label{extracondition}
\qquad \frac{n_{\Lambda}(t)\log t}{t}=O(1),\qquad\text{where}\qquad
n_{\Lambda}(t):=\sum_{|\lambda_n|\le t}\mu_n.
\end{equation}
We shall then say that $\Lambda$ belongs to the $ABCD$ class of multiplicity sequences.

\begin{lemma}\label{pointwiselemma}
Let the multiplicity sequence $\Lambda=\{\lambda_n,\mu_n\}_{n=1}^{\infty}$ belong to the $ABCD$ class.
Suppose that the series
\begin{equation}\label{taylor-dirichlet-representationpointlemma}
f(x)=\sum_{n=1}^{\infty}\left(\sum_{k=0}^{\mu_n-1} c_{n,k} x^k\right) e^{\lambda_n x}
\end{equation}
converges $\bf pointwise$ on a bounded interval $(\gamma,\beta)$. Then the series extends analytically to
the sector $\Theta_{\eta,\beta}$ $(\ref{opensector})$ with the same series representation, $x$ replaced by $z$,
converging uniformly on compact subsets of $\Theta_{\eta,\beta}$.
\end{lemma}

The proof of this lemma depends on an old result of G. Valiron \cite{Valiron}, on
a Brudnyi inequality recalled below (Theorem $\bf E$), and the following $\bf{Distance\,\, result}$,
whose proof follows along the lines of proving Theorem $\ref{Distances}$.

\begin{theorem}\label{ContinuousDistance}
Denote by $C[\gamma, \beta]$ the space of complex-valued continuous functions on the interval
$[\gamma, \beta]$ equipped with the supremum norm $||f||_{C[\gamma,\beta]}$.
Let the multiplicity sequence $\Lambda=\{\lambda_n,\mu_n\}_{n=1}^{\infty}$ belong to the $ABC$ class.
Let $D_{\gamma,\beta,n,k}$ be the $\bf Distance$ between the function $e_{n,k}=x^k e^{\lambda_n x}$
and the closed span of the system $E_{\Lambda_{n,k}}=E_{\Lambda}\setminus e_{n,k}$ in $C[\gamma, \beta]$, that is
\[
D_{\gamma,\beta,n,k}:=\inf_{g\in \overline{\text{span}} (E_{\Lambda_{n,k}})} ||e_{n,k}-g||_{C[\gamma,\beta]}.
\]
Then, for every $\epsilon>0$ there is a constant $u_{\epsilon}>0$, independent of $n\in\mathbb{N}$ and $k=0,1,\dots,\mu_n-1$,
but depending on $\Lambda$ and $(\beta-\gamma)$, so that
\begin{equation}\label{distancelowerboundsc}
D_{\gamma,\beta,n,k}\ge u_{\epsilon}e^{(\beta-\epsilon)\Re\lambda_n}.
\end{equation}
\end{theorem}

As a result of Lemma $\ref{pointwiselemma}$, we revisit Theorems $\ref{converse}$ and $\ref{CarlesonTheorem}$
and we obtain the following.

\begin{theorem}\label{CarlesonTheorem2}
Let the multiplicity sequence $\Lambda$ belongs to the $ABCD$ class.
Let $W(\Lambda , \gamma, \beta)$ be the space of functions defined on the interval $(\gamma, \beta)$
that admit a Taylor-Dirichlet series representation
\[
\sum_{n=1}^{\infty}\left(\sum_{k=0}^{\mu_n-1} c_{n,k} x^k\right) e^{\lambda_n x},
\]
converging $\bf pointwise$ on $(\gamma, \beta)$.

(i) Let $V(\Lambda,\gamma,\beta)$ be the solution space of the differential equation $(\ref{CarlesonLeontevEquation})$. Then
\[
V(\Lambda,\gamma,\beta)=W(\Lambda,\beta).
\]

(ii) Let $f\in W(\Lambda , \beta)$ and suppose that $f\in L^p (\gamma, \beta)$. Then $f$ belongs to the closed span
of the exponential system $E_{\Lambda}$ in $L^p (\gamma, \beta)$.
\end{theorem}

\subsection{A Bernstein-type inequality for exponential polynomials by Alexander Brudnyi}

Following A. Brudnyi \cite{Brudnyi}, for $n=1,2,\dots, M$, let $\lambda_n$ be complex numbers and consider the exponential polynomial
\[
q(x):=\sum_{n=1}^{M}\left(\sum_{k=0}^{\mu_n-1} c_{n,k} x^k\right) e^{\lambda_n x},\quad c_{n,k}\in\mathbb{C}.
\]
The degree of the exponential polynomial $q$, denoted by $m(q)$, is defined to be
\[
m(q):=\mu_1+\mu_2+\mu_3+\dots+\mu_M.
\]
The exponential type of $q$, denoted by $\epsilon (q)$, is defined to be
\[
\epsilon (q):=\max_{1\le n\le M}\max_{\{z:\,\, |z|\le 1\}}|\lambda_n z|.
\]

Brudnyi proved the following Bernstein-type inequality.

\begin{thme}\cite[Theorem 1.5]{Brudnyi}
Let $q$ be an exponential polynomial of degree m(q). Let $I$ be an interval on the real line, $I=\{x:\,\, |x-x_0|\le r\}$ for some $x_0\in\mathbb{R}$ and some $r>0$.
Let $w$ be a measurable subset of $I$. Then there exist positive constants $c_1,\, c_2,\, c_3$, which do not depend on $q,I,w$,
\[
c_1<15e^3,\qquad c_2<4e+1,\qquad c_3<4e+1
\]
and another positive constant $C$ which depends on the exponential polynomial $q$, satisfying the upper bound
\[
C<\left(M\cdot m(q)\right)^{m(q)}\cdot e^{2M},
\]
such that
\[
\sup_{x\in I} |q(x)|\le \left(\frac{c_1\cdot |I|}{|w|}\right)^l\cdot \sup_{x\in w}|q(x)|
\]
where
\[
l=\log C+(m(q)-1)\cdot\log (c_2\max\{1,\epsilon (q)\})+c_3\epsilon(q)\cdot r.
\]
\end{thme}

\subsection{Proof of Lemma $\ref{pointwiselemma}$}

Consider the Taylor-Dirichlet series $(\ref{taylor-dirichlet-representationpointlemma})$ converging pointwise on $(\gamma, \beta)$.
We will first prove that there exist positive constants $T$ and $d$, so that for every $\epsilon>0$, a positive constant
$m_{\epsilon}$ exists, independent of $n$ and $k$, with the coefficients $c_{n,k}$ having the upper bound
\begin{equation}\label{upperboundoncoefficients}
|c_{n,k}|\le m_{\epsilon}\cdot T\cdot e^{(-\beta+\epsilon+7d)\Re\lambda_n},\quad \forall\,\, n\in\mathbb{N},\quad\forall\,\, k=0,1,\dots,\mu_m-1.
\end{equation}
Then by utilizing an old result of G. Valiron, we will reach our conclusion.

\subsubsection{Obtaining the upper bound $(\ref{upperboundoncoefficients})$}

Consider the sequence of continuous functions $\{f_M\}_{M=1}^{\infty}$ on $(\gamma, \beta)$,
\[
f_M(z):=\sum_{n=1}^{M}\left(\sum_{k=0}^{\mu_n-1} c_{n,k} x^k\right) e^{\lambda_n x},\qquad M=1,2,\dots,
\]
which converges pointwise to $f$ on $(\gamma, \beta)$.

Let us choose some small $\epsilon>0$: by Egoroff's theorem there exists a measurable subset of $(\gamma,\beta)$, call it $E$,
such that $\beta-\gamma>\mu (E)>\beta-\gamma-\epsilon$ and
such that $\{f_M\}_{M=1}^{\infty}$ converges uniformly to $f$ on $E$.

Since $E\subset (\gamma, \beta)$ is measurable there is a closed and hence compact set $K\subset E$
such that $\beta-\gamma>\mu (K)>\beta-\gamma-2\epsilon$.
The sequence $\{f_M\}_{M=1}^{\infty}$ converges uniformly to $f$ on $K$ and this implies that $f$ is a continuous function on $K$.
Clearly now there is a positive constant $T$ such that

\begin{equation}\label{boundcompact}
||f_M||_{C[K]}\le T \qquad \forall\,\, M\in\mathbb{N}\qquad\text{where}\quad  ||f_M||_{C[K]}=\sup_{x\in K}|f_M (x)|.
\end{equation}
We now use the result of Brudnyi. First observe that the degree $m(f_M)$ of  the exponential polynomial $f_M$ is
\[
m(f_M)=\mu_1+\mu_2+\mu_3+\dots+\mu_M\qquad\text{that\,\, is}\qquad m(f_M)=n_{\Lambda}(|\lambda_M|)
\]
where $n_{\Lambda}(t)=\sum_{|\lambda_n|\le t}\mu_n$ is the counting function of $\Lambda$.
The exponential type of $f_M$ satisfies
\[
\epsilon (f_M)\le 2\Re\lambda_M.
\]
Consider the interval $I$, $I=\{x:\, x\in [\gamma,\beta]\}$ with radius $r=(\beta-\gamma)/2$ and the compact subset $K\subset I$.
By Theorem $\bf E$, there exist positive constants $c_1,\, c_2,\, c_3$,
\[
c_1<15e^3,\qquad c_2<4e+1,\qquad c_3<4e+1
\]
and another positive constant $C$ satisfying the upper bound
\begin{equation}\label{ConstantC}
C<\left(M\cdot n_{\Lambda}(|\lambda_M|)\right)^{n_{\Lambda}(|\lambda_M|)}\cdot e^{2M},
\end{equation}
such that
\begin{equation}\label{remez}
||f_M||_{C[\gamma,\beta]}\le \left(\frac{c_1\cdot (\beta-\gamma)}{|K|}\right)^l\cdot ||f_M||_{C[K]}
\end{equation}
where
\[
l=\log C+(n_{\Lambda}(|\lambda_M|)-1)\cdot\log (2c_2\Re\lambda_M)+c_3\Re\lambda_M\cdot (\beta-\gamma).
\]
We will show below that $l$ satisfies the upper bound $(\ref{upperboundforl})$ for some $d>0$. Observe that
from $(\ref{ConstantC})$ we get
\[
l<  2M+n_{\Lambda}(|\lambda_M|)\cdot\log(M\cdot n_{\Lambda}(|\lambda_M|))+(n_{\Lambda}(|\lambda_M|)-1)\cdot\log (2c_2\Re\lambda_M)+c_3\Re\lambda_M\cdot (\beta-\gamma).
\]
Since $\sum_{n=1}^{\infty}\mu_n/|\lambda_n|<\infty$ and $\sup |\arg\lambda_n|<\pi/2$,  then
\[
\frac{M}{\Re\lambda_M}\to 0,\,\, M\to\infty,\qquad\text{and}\qquad \frac{n_{\Lambda}(|\lambda_M|)}{\Re\lambda_M}\to 0,\,\, n\to\infty.
\]
Thus $M<\Re\lambda_M$ and  $n_{\Lambda}(|\lambda_M|)< \Re\lambda_M$ as well. Substituting above gives
\begin{eqnarray*}
l & < &  2\Re\lambda_M+2n_{\Lambda}(|\lambda_M|)\cdot\log(\Re\lambda_M)+(n_{\Lambda}(|\lambda_M|)-1)\cdot\log (2c_2\Re\lambda_M)+c_3\Re\lambda_M\cdot (\beta-\gamma)\nonumber\\
& < & A\cdot \Re\lambda_M +B\cdot n_{\Lambda}(|\lambda_M|)\cdot\log(\Re\lambda_M)
\end{eqnarray*}
for some positive constants $A$ and $B$.
But $\Lambda$ satisfies the condition $(\ref{extracondition})$ as well which implies that
\[
n_{\Lambda}(|\lambda_M|)\cdot \log(\Re\lambda_M)\le D\Re\lambda_M
\]
for some $D>0$. All combined shows that there is some positive constant $d$ so that
\begin{equation}\label{upperboundforl}
l\le d\Re\lambda_M.
\end{equation}

Let us go back to relation $(\ref{remez})$.
From $(\ref{boundcompact})$, $(\ref{upperboundforl})$, the upper bound on $c_1$,
and the fact that $\beta-\gamma>\mu (K)>\beta-\gamma-2\epsilon$, we get
\begin{equation}\label{boundforfM}
||f_M||_{C[\gamma,\beta]}\le e^{7l}\cdot T\le Te^{7d\Re\lambda_M}.
\end{equation}
Then write $f_M(x)$ as
\[
f_M(x)=c_{m,k}\cdot [e_{m,k}(x) + Q_{M,m,k}(x)], \qquad e_{m,k}(x)=x^ke^{\lambda_m x},
\]
where for fixed $M\in\mathbb{N}$, $m=1,\dots, M$ and $k=0,1,\dots,\mu_M-1$,
\[
Q_{M,m,k}(x):=\left(\sum_{j=0,j\not=k}^{\mu_m-1} \frac{c_{m,j}}{c_{m,k}}x^{j}\right) e^{\lambda_m x}+\sum_{n=1, n\not=m}^{M}\left(\sum_{j=0}^{\mu_n-1} \frac{c_{n,j}}{c_{m,k}}x^{j}\right)
e^{\lambda_n x}.
\]
Then from  $(\ref{boundforfM})$ we get
\[
T\cdot e^{7d\Re\lambda_M}\ge |c_{m,k}|\cdot ||e_{m,k}+Q_{M,m,k}||_{C[\gamma,\beta]},\quad \forall\,\,
m=1,\dots, M,\quad\forall\,\, k=0,1,\dots,\mu_M-1.
\]
Clearly $Q_{M,m,k}$ belongs to the span of the exponential system $E_{\Lambda}\setminus e_{m,k}$.
It then follows from the distance result $(\ref{distancelowerboundsc})$, that for every $\epsilon>0$ there is a positive constant
$m_{\epsilon}$ which depends only on $\Lambda$ and $(\beta-\gamma)$, such that
$(\ref{upperboundoncoefficients})$ is true.

\subsubsection{Ending the proof with the aid of Valiron}

With $(\ref{upperboundoncoefficients})$ verified, Lemma $\ref{TDSeries}$ says that the Taylor-Dirichlet series
\[
g(z)=\sum_{n=1}^{\infty}\left(\sum_{k=0}^{\mu_n-1} c_{n,k}z^k\right) e^{\lambda_n z}
\]
defines an analytic function in the sector
\begin{equation}\label{opensectorD}
\Theta_{\eta ,\beta-7d}:=\left\{z:\left|\frac{\Im z}{\Re (z-\beta+7d)}\right|\le \frac{1}{\tan\eta},\,\,\Re z<\beta-7d\right\},
\end{equation}
converging uniformly on its compact subsets. By the Valiron-Hille results, we know that the open regions
of pointwise convergence and absolute convergence of $g(z)$, coincide with the open region $\cal D$ of convergence of the series
\[
g^* (z)=\sum_{n=1}^{\infty}C_n e^{\lambda_n z},\qquad C_n=\max\{|c_{n,k}|:\,\, k=0,1,\dots,\mu_n-1\}
\]
and this open region $\cal D$ is convex with $\Theta_{\eta,\beta-7d}\subset \cal D$.

In fact we will show below that the interval $(\gamma,\beta)$ is a subset of $\cal D$. So
suppose otherwise: then there is a point $x_0\in (\gamma,\beta)$ such that $x_0$ is not in $\cal D$, thus $x_o\notin\Theta_{\eta ,\beta-7d}$ either, which means that $x_0>\beta-7d$.
If this is true, we then claim that $(x_0,\beta)\cap \cal{D}=\emptyset$.
Indeed, suppose that there is a point $x_1 \in (x_0,\beta)$ such that $x_1\in \cal D$ and choose the point $x_2=\beta-7d-1$ which clearly is in the sector $\Theta_{\eta ,\beta-7d}$
hence in $\cal D$ as well. Then the convexity of the region $\cal D$ implies that the whole segment $[x_2,x_1]$ lies in $\cal D$ as well. But this is a contradiction
since $x_0 \in (x_2,x_1)$. Thus if some point $x_0\in (\gamma,\beta)$ is not in $\cal D$, then $[x_0,\beta)\cap \cal{D}=\emptyset$. Then let $x_3=(x_0+\beta)/2$ thus
$(x_3,\beta)\subset (x_0,\beta)$. Due to the convexity of the region $\cal D$, we can then show that there is $\delta>0$ so that the
rectangle
\[
K_{x_3,\beta,\delta}:=\{z:\,\, x_3\le\Re z\le\beta,\,\, |\Im z|\le\delta\}
\]
does not intersect the closure of $\cal D$.
This means that
\[
\inf\{|z_1-z_2|:\,\, z_1\in [x_3,\beta],\,\, z_2\in \overline{\cal{D}}\}>0,
\]
that is, the $\bf Distance$ of the segment $[x_3,\beta]$ from the closure of $\cal D$ is $\bf positive$.
However, we will see below, based on a Valiron result, that this cannot take place: thus $(\gamma,\beta)\subset\cal D$.

First we note that if a multiplicity sequence $\Lambda$ satisfies  $(\ref{ValironConditionslimit})$ as well as the condition
\begin{equation}\label{extraValiron}
\lim_{n\to\infty}\frac{\mu_n\log n}{\lambda_n}=0,
\end{equation}
then Valiron (\cite[page 30]{Valiron} as well as  Lepson \cite[pages 42-43]{Lepson} proved that the $\bf set$ of points at which a
Taylor-Dirichlet series converges pointwise and which lies at a $\bf positive$ distance from the convex region of convergence must have
Carath\'{e}odory Linear Measure \footnote{ see \cite[Chapter 6]{Pommerenke} for Linear Measure} equal to zero.
We recall now that since $\Lambda\in ABC$, then condition $(\ref{conditionn})$ is a necessary one. If we write
\[
\frac{\mu_n\log n}{\lambda_n}=\frac{\mu_n\log |\lambda_n|}{|\lambda_n|}\cdot\frac{\log n}{\log |\lambda_n|}\cdot \frac{|\lambda_n|}{\lambda_n}
\]
and use the fact that the convergence of the series $\sum_{n=1}^{\infty}\mu_n/|\lambda_n|$ implies that $n/\lambda_n\to 0$ as $n\to \infty$,
we then see that conditions $(\ref{ValironConditionslimit})$ and $(\ref{extraValiron})$ definitely hold when $\Lambda\in ABC$.
Thus for a Taylor-Dirichlet series associated to such a multiplicity sequence, it is not possible to have pointwise convergence
on a set of points which has a positive linear measure and lies at a positive distance from the convex region of convergence.

As shown above, if the interval $(\gamma,\beta)$ is not a subset of the convex region $\cal D$, this implies that there is a segment $[x_3,\beta)$ on which the series $f(z)$
converges pointwise and which is at a positive distance from the closure of $\cal D$. But the linear measure of the interval $[x_3,\beta)$ is equal to its length $(\beta-x_3)$ (see \cite[Proposition 6.2]{Pommerenke}), hence it is positive.
By Valiron's result we reach a contradiction.
In other words, we proved that the whole interval $(\gamma,\beta)$ is a subset of $\cal D$.
But then the convexity of the region $\cal D$ implies that the sector $\Theta_{\eta,\beta}$ $(\ref{opensector})$ is also a subset of $\cal D$.
The proof of this lemma is complete.

\section{Acknowledgements}
The author expresses his deep appreciation to Professor Alekos Vidras for his useful suggestions and remarks.

\appendix

\section{Matrix representation of relation $(\ref{representation}$)}

If we identify $r_{n,k}$ with $R_{n,k}$, let $\xi_n=\mu_n-1$ for all $n\in\mathbb{N}$, let $c_{n,k,j,l}=\langle r_{n,k}, r_{j,l}\rangle$,
and write the elements of the exponential system $E_{\Lambda}$ and its biorthogonal sequence $r_{\Lambda}$ in column matrix form,
then we have the following relationship:
\[
r_{\Lambda} = G_{r_{\Lambda}}\cdot E_{\Lambda}\quad \text{for\,\, all}\quad x<\beta
\]
where $G_{r_{\Lambda}}$ is the Gram matrix associated with $r_{\Lambda}$ with entries
the inner products $\langle r_{n,k}, r_{j,l}\rangle$.

 \[
  \begin{pmatrix}
  r_{1,0}(x)  \\
  r_{1,1}(x)  \\
  \vdots \\
  r_{1,\xi_1}(x) \\

 \vdots  \\

 r_{2,0}(x)  \\
  r_{2,1}(x)  \\
  \vdots \\
  r_{2,\xi_2}(x) \\

  \vdots \\

  r_{n,0}(x)  \\
  r_{n,1}(x)  \\
  \vdots \\
  r_{n,\xi_n}(x) \\

 \vdots
 \end{pmatrix}
 =
 \begin{pmatrix}
  c_{1,0,1,0} & c_{1,0,1,1} & \cdots & c_{1,0,1,\xi_1} & c_{1,0,2,0} & c_{1,0,2,1} & \cdots & c_{1,0,2,\xi_2} & \cdots \\
  c_{1,1,1,0} & c_{1,1,1,1} & \cdots & c_{1,1,1,\xi_1} & c_{1,1,2,0} & c_{1,1,2,1} & \cdots & c_{1,1,2,\xi_2} & \cdots \\
 \vdots  & \vdots  & \ddots & \vdots & \vdots & \vdots & \ddots & \vdots  & \cdots \\
  c_{1,\xi_1,1,0} & c_{1,\xi_1,1,1} & \cdots & c_{1,\xi_1,1,\xi_1} & c_{1,\xi_1,2,0} & c_{1,\xi_1,2,1} & \cdots & c_{1,\xi_1,2,\xi_2} & \cdots \\

  c_{2,0,1,0} & c_{2,0,1,1} & \cdots & c_{2,0,1,\xi_1} & c_{2,0,2,0} & c_{2,0,2,1} & \cdots & c_{2,0,2,\xi_2} & \cdots \\
  c_{2,1,1,0} & c_{2,1,1,1} & \cdots & c_{2,1,1,\xi_1} & c_{2,1,2,0} & c_{2,1,2,1} & \cdots & c_{2,1,2,\xi_2} & \cdots \\
 \vdots  & \vdots  & \ddots & \vdots & \vdots & \vdots & \ddots & \vdots  & \cdots \\
  c_{2,\xi_2,1,0} & c_{2,\xi_2,1,1} & \cdots & c_{2,\xi_2,1,\xi_1} & c_{2,\xi_2,2,0} & c_{2,\xi_2,2,1} & \cdots & c_{2,\xi_2,2,\xi_2} & \cdots \\
  \vdots  & \vdots  & \vdots & \vdots & \vdots & \vdots & \vdots & \vdots  & \cdots \\
  c_{n,0,1,0} & c_{n,0,1,1} & \cdots & c_{n,0,1,\xi_1} & c_{n,0,2,0} & c_{n,0,2,1} & \cdots & c_{n,0,2,\xi_2} & \cdots \\
  c_{n,1,1,0} & c_{n,1,1,1} & \cdots & c_{n,1,1,\xi_1} & c_{n,1,2,0} & c_{n,1,2,1} & \cdots & c_{n,1,2,\xi_2} & \cdots \\
 \vdots  & \vdots  & \ddots & \vdots & \vdots & \vdots & \ddots & \vdots  & \cdots \\
  c_{n,\xi_n,1,0} & c_{n,\xi_n,1,1} & \cdots & c_{n,\xi_n,1,\xi_1} & c_{n,\xi_n,2,0} & c_{n,\xi_n,2,1} & \cdots & c_{n,\xi_n,2,\xi_2} & \cdots \\
  \vdots  & \vdots  & \vdots & \vdots & \vdots & \vdots & \vdots & \vdots  & \cdots
 \end{pmatrix}
 \cdot
 \begin{pmatrix}
  e^{\lambda_1 x}  \\
  x e^{\lambda_1 x}  \\
  \vdots \\
  x^{\xi_1} e^{\lambda_1 x} \\

  e^{\lambda_2 x}  \\
  x e^{\lambda_2 x} \\
  \vdots \\
  x^{\xi_2} e^{\lambda_2 x} \\
  \vdots \\
  e^{\lambda_n x}  \\
  x e^{\lambda_n x} \\
  \vdots \\
  x^{\xi_n} e^{\lambda_n x}\\
  \vdots
 \end{pmatrix}
 \]

\section{A uniqueness result concerning the coefficients of Taylor-Dirichlet series}
\setcounter{equation}{0}

The following result is known, however for the sake of completeness we present below a proof.
\begin{lemma}\label{uniquenessDirichlet}
Suppose that $\Lambda=\{\lambda_n,\mu_n\}_{n=1}^{\infty}$ satisfies $\sum_{n=1}^{\infty}\mu_n/|\lambda_n|<\infty$.
Suppose that two Taylor-Dirichlet series
\[
f(z)=\sum_{n=1}^{\infty}\left(\sum_{k=0}^{\mu_n-1}c_{n,k} z^{k}\right) e^{\lambda_n z}\quad\text{and}\quad
g(z)=\sum_{n=1}^{\infty}\left(\sum_{k=0}^{\mu_n-1}d_{n,k} z^{k}\right) e^{\lambda_n z},
\]
are analytic in some region $\Omega$ and $g=f$ on some closed disk in $\Omega$.
Then $c_{n,k}=d_{n,k}$ for all $n\in\mathbb{N}$ and $k=0,1,\dots,\mu_n-1$.
\end{lemma}
\begin{proof}

Consider the entire function $F$ of exponential type zero, vanishing exactly on $\Lambda$, with
\[
F(z)=\prod_{n=1}^{\infty}\left(1-\frac{z}{\lambda_n}\right)^{\mu_n}.
\]
Since $1/F$ has a pole of order $\mu_n$ at the point $\lambda_n$, we write down the Laurent series
\[
\frac{1}{F(z)}=\sum_{k=1}^{\mu_n}\frac{A_{n,k}}{(z-\lambda_n)^k}+f_n(z)
\]
which holds in some closed punctured disk $D_n$ so that $f_n$ is the regular part and
\[
A_{n,k}=\frac{1}{2\pi i}\int_{\partial D_n} \frac{(z-\lambda_n)^{k-1}}{F(z)}\, dz.
\]
As in Lemma $\ref{FirstLemma}$, we can construct entire functions
$\{F_{n,k}\}_{n=1, k=0,\dots,\mu_n-1}^{n=\infty}$, and in fact of exponential type zero, that satisfy
\[
F_{n,k}^{(l)}(\lambda_j)
=\begin{cases} 1, & j=n,\,\,  l=k, \\ 0,  & j=n,\,\,
l\in\{0,1,\dots,\mu_n-1\}\setminus\{k\}, \\ 0, &
j\not=n,\,\, l\in\{0,1,\dots,\mu_j-1\},\end{cases}
\]
by defining
\[
F_{n,k}(z) := \frac{F(z)}{k!}\sum_{l=1}^{\mu_n-k}\frac{A_{n,k+l}}{(z-\lambda_n)^l}.
\]
Therefore, if we denote by $\gamma_{n,k}(z)$ the Borel transform of $F_{n,k}(z)$,
by the P\'{o}lya representation theorem of entire functions of exponential type, we have
\[
F_{n,k}(z)=\frac{1}{2\pi i}\int_{\partial B_{\epsilon}} \gamma_{n,k}(w) e^{zw}\, dw,
\]
where $\partial B_{\epsilon}$ is the boundary of the closed disk $B_{\epsilon}=\{z:|z|\le\epsilon\}$ for $\epsilon>0$.
Differentiation yields
\begin{equation}\label{orthogonalcc}
\frac{1}{2\pi i}\int_{\partial B_{\epsilon}}\gamma_{n,k}(w)\cdot  w^l e^{\lambda_j w}\,
dw=\begin{cases} 1 & j=n,\,\,  l=k, \\ 0,  & j=n,\,\,
l\in\{0,1,\dots,\mu_n-1\}\setminus\{k\}, \\ 0, &
j\not=n,\,\, l\in\{0,1,\dots,\mu_j-1\}.\end{cases}
\end{equation}
Without loss of generality, suppose that the two Taylor-Dirichlet series $f,g$ are equal on the closed disk $B_{\epsilon}$
for some $\epsilon>0$. The continuity of the Borel Transforms on  $\partial B_{\epsilon}$, the uniform convergence
of the two series on $B_{\epsilon}$ and relation $(\ref{orthogonalcc})$, imply that
for every fixed $n\in\mathbb{N}$ and $k\in\{0,1,\dots,\mu_n-1\}$, one has
\[
\frac{1}{2\pi i}\int_{\partial B_{\epsilon}}\gamma_{n,k}(z)\cdot f(z)\, dz=c_{n,k}\quad\text{and}\quad
\frac{1}{2\pi i}\int_{\partial B_{\epsilon}}\gamma_{n,k}(z)\cdot g(z)\, dz=d_{n,k}.
\]
Since $f=g$ on $B_{\epsilon}$ then $c_{n,k}=d_{n,k}$.
\end{proof}

\section{Revisiting the Distance result $(\ref{distancehalflinezero})$}
\setcounter{equation}{0}

In this section we provide a second proof of the Distance result $(\ref{distancehalflinezero})$. The tool is the meromorphic function $f$ $(\ref{meromorphicfunctionfz})$
and we begin with a result regarding this function. We then prove two helpful lemmas.
As usual, we assume that the multiplicity sequence $\Lambda$ belongs to the class $ABC$.

\subsection{The Meromorphic function $f(z)$}

\begin{lemma}\label{meromorphic}
Consider the meromorphic function $f(z)$ $(\ref{meromorphicfunctionfz})$.
Then this function is analytic in the right half-plane $\Re z>-4$ and there is some $M>0$ so that
\begin{equation}\label{upperboundmer}
|f(z)|\le\frac{M}{1+y^2}\qquad \forall\,\, z:\,\, \Re z\ge -2.
\end{equation}
Moreover, for fixed $\epsilon>0$ let $C_{n,\epsilon}$ for $n=1,2,\dots$ be the disks as in  Remark $\ref{disjointdisks}$ and let
$\partial C_{n,\epsilon}$ be their respective circles. Then
there is a positive constant $m_{\epsilon ,2}$, independent of $n\in\mathbb{N}$, so that
\begin{equation}\label{lowerboundmer}
|f(z)|\ge m_{\epsilon ,2}e^{-\epsilon |\lambda_n|}\qquad \forall\,\, z\in \partial C_{n,\epsilon},\quad n=1,2,\dots.
\end{equation}
\end{lemma}
\begin{proof}
We can write
\[
|f(z)|=\frac{1}{|4+z|^2}\cdot \prod_{n=1}^{\infty}\left(\frac{|\lambda_n-z|}{|\overline{\lambda_n}+4+z|}\right)^{\mu_n}
\cdot\prod_{n=1}^{\infty}\left|\frac{\overline{\lambda_n}+4}{\lambda_n}\right|^{\mu_n}
\]
since the infinite product
\[
\prod_{n=1}^{\infty}\left|\frac{\overline{\lambda_n}+4}{\lambda_n}\right|^{\mu_n}
\]
converges due to condition $(\ref{convergencecondition})$. Then clearly
relation $(\ref{upperboundmer})$ holds for all $z$ such that $\Re z\ge -2$.

Next we prove $(\ref{lowerboundmer})$: we can write $(4+z)^2f(z)$ as
\[
\frac{\prod_{n=1}^{\infty}\left(1-z/\lambda_n\right)^{\mu_n}}
{\prod_{n=1}^{\infty}\left(1+z/\overline{\lambda_n+4}\right)^{\mu_n}}
\]
since both infinite products define entire functions. Both of them are of exponential type zero since $\Lambda\in ABC$,
thus for every $\epsilon>0$ there is a positive constant $t_{\epsilon}$, such that the modulus of the function in the denominator is bounded from above by
$t_{\epsilon}e^{\epsilon|z|}$. Regarding the numerator, we claim that it satisfies the lower bound $(\ref{lowerboundforF})$. The proof is identical as in the case
of the even function in Lemma $\ref{anevenfunction}$. Combining all together yields $(\ref{lowerboundmer})$.
\end{proof}

\subsection{Auxiliary results}
\begin{lemma}\label{meromorphicfunctions}
There exist analytic functions $\{f_{n,k}(z):\, n\in\mathbb{N},\, k=0,1,\dots,\mu_n-1\}$ in the half-plane $\Re z>-2$, so that
\begin{equation}\label{G1a}
f_{n,k}^{(l)}(\lambda_j)
=\begin{cases} 1 & j=n,\,\,  l=k, \\ 0,  & j=n,\,\,
l\in\{0,1,\dots,\mu_n-1\}\setminus\{k\}, \\ 0, &
j\not=n,\,\, l\in\{0,1,\dots,\mu_j-1\}.\end{cases}
\end{equation}
For fixed $\epsilon>0$ let $C_{n,\epsilon}$ for $n=1,2,\dots$ be the disks as in  Remark $\ref{disjointdisks}$ and let
$\partial C_{n,\epsilon}$ be their respective circles.
Then there is a positive constant $m_{\epsilon}$, independent of $n$ and $k$,
so that for every fixed $n\in\mathbb{N}$ and $k\in\{0,1,\dots,\mu_n-1\}$, one has
\begin{equation}\label{upperboundfnk}
|f_{n,k}(z)|\le \frac{m_{\epsilon} e^{\epsilon \Re \lambda_n}}{1+y^2},\quad \forall\,\, z\in\left\{z:\,\, \Re z>-2,\, |z-\lambda_n|\ge \frac{1}{|\lambda_n|}\right\},
\end{equation}
and
\begin{equation}\label{diskbound}
|f_{n,k}(z)|\le m_{\epsilon} e^{\epsilon \Re \lambda_n},\quad \forall\,\, z\in\left\{z:\,\, |z-\lambda_n|< \frac{1}{|\lambda_n|}\right\}.
\end{equation}
\end{lemma}
\begin{proof}
Consider the meromorphic function $f$ of Lemma  $\ref{meromorphic}$. Since $1/f(z)$ has a pole of order $\mu_n$ at the point $\lambda_n$,
we write down its Laurent series representation
\[
\frac{1}{f(z)}=\sum_{j=1}^{\mu_n}\frac{A_{n,j}}{(z-\lambda_n)^j}+g_n(z)
\]
which is valid in the open punctured disk $C_{n,\epsilon}\setminus\lambda_n$ such that $g_n(z)$ is the regular part and
\[
A_{n,j}=\frac{1}{2\pi i}\int_{\partial C_{n,\epsilon}} \frac{(z-\lambda_n)^{j-1}}{f(z)}\, dz.
\]
Then from $(\ref{lowerboundmer})$ we see that for the fixed $\epsilon>0$ there is a positive constant $m^*_{\epsilon}$ so that
\begin{equation}\label{upperboundforAnjmer}
|A_{n,j}|\le m^*_{\epsilon}e^{\epsilon \Re\lambda_n}.
\end{equation}

We now construct the functions that satisfy $(\ref{G1a})$.
Fix some positive integer $n$ and some $k\in\{0,1,2,\dots,\mu_n-1\}$ and define
\begin{equation}\label{formulaforfnk}
f_{n,k}(z): = \frac{f(z)}{k!}\sum_{l=1}^{\mu_n-k}\frac{A_{n,k+l}}{(z-\lambda_n)^l}.
\end{equation}
First suppose that $k=0$, thus
\[
f_{n,0}(z) = f(z)\sum_{l=1}^{\mu_n}\frac{A_{n,l}}{(z-\lambda_n)^l}.
\]
Then we get $f_{n,0} ^{(l)}(\lambda_j)=0$ for $j\not= n$ and $l=0,1,\dots, \mu_j -1$.
Since $f_{n,0}(z)$ is continuous at $z=\lambda_n$, then
\[
f_{n,0}(z)=f(z)\left[ \frac{1}{f(z)}-g_{n}(z)\right]=1-f(z)g_{n}(z)\qquad\forall\,\, z\in C_{n,\epsilon}.
\]
Hence, $f_{n,0}(\lambda_n)=1$ and  $f_{n,0}^{(l)}(\lambda_n)=0$ for $l\in\{1,\dots, \mu_n -1\}$.
Thus, $f_{n,0}(z)$ satisfies $(\ref{G1a})$.

Next, suppose that $k\in\{1,2,\dots ,\mu_n-1\}$. Since $f_{n,k}(z)$ is continuous at $z=\lambda_n$, we rewrite $f_{n,k}(z)$ for all $z$
in $C_{n,\epsilon}$ as
\begin{align}
f_{n,k}(z) & = \frac{f(z) (z-\lambda_n)^k}{k!}\sum_{l=k+1}^{\mu_n}\frac{A_{n,l}}{(z-\lambda_n)^l}\nonumber\\
& = \frac{f(z) (z-\lambda_n)^k}{k!}\left[\frac{1}{f(z)}-g_n(z)- \sum_{j=1}^{k}\frac{A_{n,j}}{(z-\lambda_n)^j}\right]\nonumber\\
& = \frac{(z-\lambda_n)^k}{k!}-\frac{f(z) (z-\lambda_n)^k g_n(z)}{k!}-\frac{f(z)}{k!}\sum_{j=1}^{k}A_{n,j}(z-\lambda_n)^{k-j}.
\label{secondformulaforGnkmer}
\end{align}
From $(\ref{formulaforfnk})$ we get $f_{n,k} ^{(l)}(\lambda_j)=0$ for $j\not= n$ and $l=0,1,\dots, \mu_j -1$.
From $(\ref{secondformulaforGnkmer})$ we get  $f_{n,k} ^{(k)}(\lambda_n)=1$ and $f_{n,k} ^{(l)}(\lambda_n)=0$ for $l\in\{0,1,\dots, \mu_n -1\}\setminus
\{k\}$. Thus, $f_{n,k}(z)$ satisfies $(\ref{G1a})$ for $k\not= 0$.

Next, by combining $(\ref{munlambdan})$, $(\ref{upperboundmer})$, $(\ref{upperboundforAnjmer})$ and $(\ref{formulaforfnk})$,
shows that for the fixed $\epsilon>0$ there is a positive constant $m_{\epsilon}$ such that
the upper bound $(\ref{upperboundfnk})$ holds for all $z:\, \Re z>-2$ which are
outside the open disk $|z-\lambda_n|<1/|\lambda_n|$. Finally, $(\ref{diskbound})$ follows by the Maximum Modulus Theorem.
\end{proof}

\begin{lemma}\label{SecondLemmamer}
There exist continuous functions $\{h_{n,k}(t):\,\, n\in\mathbb{N},\,\, k=0,1,\dots,\mu_n-1\}$
on the interval $(-\infty,0]$, with $h_{n,k}\in L^2 (-\infty,0]$,
so that
\begin{equation}\label{orthogonalmer}
\frac{1}{\sqrt{2\pi}}\int_{-\infty}^{0}h_{n,k}(t)e^{2t}t^{l}e^{\lambda_j t}\,
dt=\begin{cases} 1, & j=n,\,\,  l=k, \\ 0,  & j=n,\,\,
l\in\{0,1,\dots,\mu_n-1\}\setminus\{k\}, \\ 0, &
j\not=n,\,\, l\in\{0,1,\dots,\mu_j-1\}.\end{cases}
\end{equation}
Furthermore, for every $\epsilon>0$ there is a constant $M_{\epsilon ,3}>0$
independent of $n$ and $k$ but depending on $\Lambda$, so that
\begin{equation}\label{upperboundforhnkmer}
|h_{n,k}(t)e^{t}|\le M_{\epsilon ,3} e^{\epsilon\Re\lambda_n}\qquad \forall\,\,
t\in (-\infty,0],\quad  n\in\mathbb{N}, \quad k\in\{0,1,\dots,\mu_n-1\}.
\end{equation}
\end{lemma}
\begin{proof}
It follows from $(\ref{upperboundfnk})$ and $(\ref{diskbound})$ that
\[
\sup_{x>-2}\int_{-\infty}^{\infty} |f_{n,k}(x+iy)|^p\, dy<\infty.
\]
Thus
\[
\sup_{x>0}\int_{-\infty}^{\infty} |f_{n,k}(x-2+iy)|^p\, dy<\infty.
\]
Therefore, the function $f_{n,k}(z-2)$ belongs to all the $H^p (\mathbb{C}_+)$ spaces for $p\ge 1$,
where $\mathbb{C}_+$ is the right half-plane $\Re z>0$.
It follows from the Paley-Wiener theorem \cite[Theorems 11.9 and 11.10]{Duren}
that there exists a continuous and square-integrable function $h_{n,k}$ on $(-\infty,0]$ so that
\[
f_{n,k}(z-2)=\frac{1}{\sqrt{2\pi}}\int_{-\infty}^{0}h_{n,k}(t) e^{tz}\, dt,\qquad \forall\,\, z:\,\, \Re z>0.
\]
Hence
\begin{equation}\label{paley}
f_{n,k}(z)=\frac{1}{\sqrt{2\pi}}\int_{-\infty}^{0}h_{n,k}(t)e^{2t} e^{tz}\, dt,\qquad \forall\,\, z:\,\, \Re z>-2.
\end{equation}
Differentiating with respect to $z$ and applying $(\ref{G1a})$ gives $(\ref{orthogonalmer})$.

Next, letting $z=-1+iy$ in $(\ref{paley})$ gives
\[
f_{n,k}(-1+iy)=\frac{1}{\sqrt{2\pi}}\int_{-\infty}^{0}h_{n,k}(t)e^{t} e^{ity}\, dt.
\]
Since $h_{n,k}$ is continuous and square-integrable on $(-\infty,0]$,
then the function $h_{n,k}(t) e^t$ is also continuous on $(-\infty,0]$. Moreover, it belongs to the space
$L^1 (-\infty,0)$ by simply applying the Cauchy-Schwartz inequality.
But $f_{n,k}(-1+iy)$ is in $L^p (-\infty, \infty)$ for all $p\ge 1$. Then by Fourier Inversion we have
\[
h_{n,k}(t)e^t=\frac{1}{\sqrt{2\pi}}\int_{-\infty}^{\infty} f_{n,k}(-1+iy)e^{-iyt}\, dy,\qquad \forall\,\, t\in (-\infty,0].
\]
Finally, since $(\ref{upperboundfnk})$ holds for the function $f_{n,k}(-1+iy)$, then the upper bound $(\ref{upperboundforhnkmer})$ follows easily.
\end{proof}

\subsection{A second proof of the Distance result $(\ref{distancehalflinezero})$}

Suppose that $f\in\overline{\text{span}}(E_{\Lambda_{n,k}})$ in the space $L^p(-\infty,0)$ for some $p>1$.
Hence for every $\epsilon>0$ there is an exponential polynomial $P_{\epsilon}\in \text{span}(E_{\Lambda_{n,k}})$ such that  $||f-P_{\epsilon}||_{L^p(-\infty,0)}<\epsilon$.
Let $h_{n,k}$ be the function as in Lemma $\ref{SecondLemmamer}$. Due to the upper bound $(\ref{upperboundforhnkmer})$ regarding the function $h_{n,k}(t)e^t$, we see that for every
$\epsilon>0$ there is a positive constant $m_{\epsilon}$, independent of $n,k$, so that
\begin{equation}\label{last}
|h_{n,k}(t)e^{2t}|\le m_{\epsilon} e^{\epsilon\Re\lambda_n}\cdot e^t \qquad \forall\,\, t\in (-\infty,0],\quad  n\in\mathbb{N}, \quad k\in\{0,1,\dots,\mu_n-1\}.
\end{equation}
Clearly now $h_{n,k}e^{2t}\in L^q (-\infty,0)$ where $1/p +1/q =1$.

Next, let $g_{n,k}(t)=h_{n,k}(t)e^{2t}$. Then
from $(\ref{orthogonalmer})$ and the H\"{o}lder inequality we get
\begin{eqnarray*}
\left|\int_{-\infty}^{0}g_{n,k}(t) f(t)\, dt\right| & = & \left|\int_{-\infty}^{0} g_{n,k}(t) \left(f(t)-P_{\epsilon}(t)\right)\, dt\right|\\
& \le & ||g_{n,k}||_{L^q(-\infty,0)}\cdot\epsilon.
\end{eqnarray*}
Since $\epsilon$ is arbitrary we then have
\[
\int_{-\infty}^{0}g_{n,k}(t) f(t)\, dt=0.
\]
Together with $(\ref{orthogonalmer})$ gives
\[
\sqrt{2\pi}=\int_{-\infty}^{0}g_{n,k}(t)t^k e^{\lambda_n t}\, dt=\int_{-\infty}^{0}g_{n,k}(t)\left(t^k e^{\lambda_n t}-f(t)\right)\, dt.
\]
Hence
\begin{equation}\label{allmer}
\sqrt{2\pi}\le ||g_{n,k}(t)||_{L^q (-\infty,0)} \cdot ||e_{n,k}-f||_{L^p (-\infty,0)}, \qquad e_{n,k}(t)=t^k e^{\lambda_n t}.
\end{equation}
Now, it follows from $(\ref{last})$ that for every $\epsilon>0$ there is some $M_{\epsilon}>0$, independent of $q$, $n\in\mathbb{N}$ and $k=0,1,\dots,\mu_n-1$,
so that
\[
||g_{n,k}(t)||_{L^q (-\infty,0)}=\left(\int_{-\infty}^{0}|h_{n,k}(t) e^{2t}|^q\, dt\right)^{1/q}\le M_{\epsilon} e^{\epsilon \Re\lambda_n}.
\]
Then, since $(\ref{allmer})$ is true for all $f\in\overline{\text{span}} (E_{\Lambda_{n,k}})$ in $L^p (-\infty,0)$, we get
\[
\sqrt{2\pi}\le \inf_{f\in\overline{\text{span}}(E_{\Lambda_{n,k}})}||e_{n,k}-f||_{L^p (-\infty,0)} \cdot
M_{\epsilon} e^{\epsilon\Re\lambda_n}.
\]
Letting $u_{\epsilon}=\sqrt{2\pi}/M_{\epsilon}$, we get $D_{-\infty,0,p,n,k}\ge u_{\epsilon}e^{-\epsilon\Re\lambda_n}$.
Similarly we can prove that $D_{-\infty,0,1,n,k}\ge u_{\epsilon}e^{-\epsilon\Re\lambda_n}$.


\begin{thebibliography}{99}


\bibitem{2018Allonsius} D. Allonsius, F. Boyer, M. Morancey,
Spectral analysis of discrete elliptic operators and applications in control theory,
Numer. Math., $\bf 140$ (2018), 857-911.

\bibitem{2021Allonsius} D. Allonsius, F. Boyer, M. Morancey,
Analysis of the null controllability of degenerate parabolic systems
of Grushin type via the moments method,
J. Evol. Equ., $\bf 21$ (2021), 4799-4843.

\bibitem{Almira} J. M. Almira,
M\"{u}ntz Type Theorems I.
Surveys in Approximation Theory $\bf 3$ (2007), 152-194.

\bibitem{2011JPA} F. Ammar-Khodja, A. Benabdallah, M. Gonz\'{a}lez-Burgos, L. de Teresa,
The Kalman condition for the boundary controllability of coupled parabolic systems.
Bounds on biorthogonal families to complex matrix exponentials,
J. Math. Pures Appl. (9) $\bf 96$ no. 6 (2011), 555-590.

\bibitem{2014JFA} F. Ammar-Khodja, A. Benabdallah, M. Gonz\'{a}lez-Burgos, L. de Teresa,
Minimal time for the null controllability of parabolic systems: the effect of the
condensation index of complex sequences,
J. Funct. Anal. $\bf 267$ no. 7 (2014), 2077-2151.

\bibitem{2016JMAA} F. Ammar-Khodja, A. Benabdallah, M. Gonz\'{a}lez-Burgos, L. de Teresa,
New phenomena for the null controllability of parabolic systems:
Minimal time and geometrical dependence,
J. Math. Anal. Appl. $\bf 444$ (2016), 1071-1113.

\bibitem{2019JPA} F. Ammar-Khodja, A. Benabdallah, M. Gonz\'{a}lez-Burgos, M. Morancey,
Quantitative Fattorini-Hautus test and minimal null control time for parabolic problems,
J. Math. Pures Appl. (9) $\bf 122$ (2019), 198-234.

\bibitem{Baranov2013} A. Baranov, Y. Belov, A. Borichev,
Hereditary completeness for systems of exponentials and reproducing kernels,
Adv. Math. $\bf 235$ (2013), 525-554.

\bibitem{Baranov2015} A. Baranov, Y. Belov, A. Borichev,
Spectral Synthesis in de Branges Spaces,
Geom. Funct. Anal. $\bf 25$ (2015) 417-452.

\bibitem{Baranov2022}  A. Baranov, Y. Belov, A. Kulikov,
Spectral synthesis for exponentials and logarithmic length,
Isr. J. Math. (2022), 1-25, DOI: 10.1007/s11856-022-2341-3.

\bibitem{2014Control} A. Benabdallah, F. Boyer, M. Gonz\'{a}lez-Burgos, G. Olive,
Sharp estimates of the one-dimensional boundary control cost
for parabolic systems and application to the
N-dimensional boundary null controllability in cylindrical domains,
SIAM J. Control Optim. $\bf 52$ no. 5 (2014), 2970-3001.

\bibitem{2020AHL} A. Benabdallah, F. Boyer, M. Morancey,
A block moment method to handle spectral condensation phenomenon in parabolic control problems,
Annales Henri Lebesgue $\bf 3$ (2020), 717-793.

\bibitem{Berenstein} C.A. Berenstein, R. Gay,
Complex Analysis and Special Topics in Harmonic Analysis, Springer-Verlag, New York, ISBN 0-387-94411-7 (1995), x+482 pp.

\bibitem{BerBaoVidras}  C. A. Berenstein, B. Q. Li, A. Vidras,
Geometric characterization of interpolating varieties for the (FN)-space $A^0 _p$ of entire functions,
Canad. J. Math. $\bf 47$ no. 1 (1995), 28-43.

\bibitem{2021Bhandari} K. Bhandari, F. Boyer,
Boundary null-controllability of coupled parabolic systems with Robin conditions,
Evol. Equ. Control Theory, $\bf 10$ no. 1 (2021), 61-102.

\bibitem{2021Bhandarib} K. Bhandari, F. Boyer,  V. Hernández-Santamaría,
Boundary null-controllability of 1-D coupled parabolic systems with Kirchhoff-type conditions,
Math. Control. Signals, Syst., $\bf 33$ (2021), 413-471.

\bibitem{Boas} R. P. Boas. Jr,
Entire Functions, Academic Press, New York, 1954.

\bibitem{BE} P. Borwein, T. Erd\'{e}lyi,
Polynomials and polynomial inequalities, Springer-Verlag, New-York, 1995. x+480 pp. ISBN: 0-387-94509-1.

\bibitem{BE1} P. Borwein, T. Erd\'{e}lyi,
Generalizations of M\"{u}ntz Theorem via a Remez-type inequality for M\"{u}ntz spaces,
J. Amer. Math. Soc. $\bf 10$ (1997), 327-349.

\bibitem{Brudnyi} A. Brudnyi,
Bernstein type inequalities for quasipolynomials,
J. Approx. Theory $\bf 112$ no. 1 (2001), 28-43.

\bibitem{2014Bugariua} I. F. Bugariu, I. Roven\c{t}a,
Small Time Uniform Controllability of the Linear One-Dimensional Schrödinger Equation with Vanishing Viscosity,
J. Optim. Theory Appl., $\bf 160$ no. 3 (2014), 949-965.


\bibitem{2014Bugariu} I. F. Bugariu,
Uniform controllability for the beam equation with vanishing structural damping,
Czechoslov. Math. J., $\bf 64$ no. 4 (2014), 869-881.

\bibitem{2014Bugariub} I. F. Bugariu, S. Micu,
A singular controllability problem with vanishing viscosity,
ESAIM Control Optim. Calc. Var., $\bf 20$ no. 1 (2014), 116-140.

\bibitem{2014Bugariuc} I. F. Bugariu, S. Micu,
A numerical method for the controls of the heat equation,
Math. Model. Nat. Phenom., $\bf 9$ no. 4 (2014), 65-87.

\bibitem{2016Bugariu} I. F. Bugariu, S. Micu, I. Roven\c{t}a,
Approximation of the controls for the beam equation with vanishing viscosity,
Math. Comp. $\bf 85$ no. 301 (2016), 2259-2303.

\bibitem{Cannarsa2017} P. Cannarsa, P. Martinez, J. Vancostenoble,
The cost of controlling weakly degenerate parabolic equations by boundary controls,
Math. Control Relat. Fields $\bf 7$ no. 2 (2017), 171-211.

\bibitem{Cannarsa2020a} P. Cannarsa, P. Martinez, J. Vancostenoble,
Precise estimates for biorthogonal families under asymptotic gap conditions,
Discrete Contin. Dyn. Syst. - S, $\bf 13$ no. 5 (2020), 1441-1472.

\bibitem{Cannarsa2020b} P. Cannarsa, P. Martinez, J. Vancostenoble,
The cost of controlling strongly degenarate parabolic equations,
ESAIM Control Optim. Calc. Var. $\bf 26$ no. 2 (2020).

\bibitem{Carleson} L. Carleson,
On infinite differential equations with constant coefficients, I,
Math. Scand. $\bf 1$ (1953), 31-38.

\bibitem{Casazza}  P. Casazza, O. Christensen, S. Li, A. Lindner,
Riesz-Fischer sequences and lower frame bounds,
Z. Anal. Anwendungen $\bf 21$ no. 2 (2002), 305-314.

\bibitem{Christensen} O. Christensen,
An introduction to Frames and Riesz Bases,
Applied and Numerical Harmonic Analysis. Birkh\"{a}user Boston, Inc., Boston, MA, 2003. xxii+440 pp. ISBN: 0-8176-4295-1.

\bibitem{CE} J. A. Clarkson, P. Erd\H{o}s,
Approximation by polynomials,
Duke Math. J. $\bf 10$ (1943), 5-11.

\bibitem{Duren} P. L. Duren,
Theory of $H^p$ spaces,
Pure and Applied Mathematics, $\bf 38$ (Academic Press, New York, 1970), p. xii+258.

\bibitem{Dzh} M. M. Dzhrbashyan,
A characterization of closed linear spans of two families of incomplete systems of analytic functions,
Math. USSR Sbornik $\bf 42$ no. 1 (1982), 1-70.

\bibitem{EJ} T. Erd\'{e}lyi, W. B. Johnson,
The   Full M\"{u}ntz Theorem in $L^p(0,1)$  for $0<p<\infty$,
J. Anal. Math. $\bf 84$ (2001), 145-172.

\bibitem{E1} T. Erd\'{e}lyi,
The Full Clarkson-Erd\H{o}s-Schwartz Theorem on the closure of non-dense M\"{u}ntz spaces,
Studia Math. $\bf 155$ (2003), 145-152.

\bibitem{E2} T. Erd\'{e}lyi,
The Full M\"{u}ntz Theorem revisited,
Constr. Approx. $\bf 21$ (2005), 319-335.

\bibitem{1971FR} H.O. Fattorini, D.L. Russell,
Exact controllability theorems for linear parabolic equations in one space dimension,
Arch. Ration. Mech. Anal. $\bf 43$ (1971), 272-292.

\bibitem{2010JFA} E. Fern\'{a}ndez-Cara, M. Gonz\'{a}lez-Burgos, L. de Teresa,
Boundary controllability of parabolic coupled equations,
J. Funct. Anal. $\bf 259$ no. 7 (2010), 1720-1758.

\bibitem{Glass2010} O. Glass,
A complex-analytic approach to the problem of uniform controllability of a transport equation in the vanishing viscosity limit,
J. Funct. Anal., $\bf 258$ no. 3 (2010), 852-868.

\bibitem{2022Gonzalez-Burgos} M. González-Burgos, L. Ouaili,
Sharp estimates for biorthogonal families to exponential
functions associated to complex sequences without gap conditions,
https://hal.archives-ouvertes.fr/hal-03115544.

\bibitem{Hille}  E. Hille,
Note on Dirichlet's series with complex exponents,
Annals of Math. $\bf 25$ no. 3 (1924), 261-278.

\bibitem{Jaming} P. Jaming, I. Simon,
M\"{u}ntz-Sz\'{a}sz type theorems for the density of the span of powers of functions,
Bull. des Sci. Math. $\bf 166$ (2021), article number 102933.

\bibitem{Khabibullin} B. N. Khabibullin, G. R. Talipova, F. B. Khabibullin,
Zero subsequences for Bernstein spaces and completeness of exponential systems in spaces of functions on an interval,
(Russian) Algebra i Analiz $\bf 26$ no. 2 (2014),  185-215; translation in St. Petersburg Math. J. $\bf 26$ no. 2 (2014), 319-340.

\bibitem{Korevaar1947} J. Korevaar,
A characterization of the submanifold , of $C[a,b]$ spanned by the sequence $\{x^{n_k}\}$,
Nederl. Akad. Wetensch. Pro. Ser. A $\bf 50$ (1947), 750-758=Indag. Math. $\bf 9$ (1947), 360-368.

\bibitem{Kriv} A. S. Krivosheev,
A fundamental principle for invariant subspaces in convex domains,
Izv. Ross. Acad. Nauk Ser. Mat. $\bf 68$ no. 2 (2004), 71-136;
English transl., Izv. Math. $\bf 68$ no. 2 (2004), 291-353.

\bibitem{Lefevre} P. Lef\`{e}vre,
M\"{u}ntz spaces and special Bloch type inequalities,
Complex Var. Elliptic Equ. $\bf 63$ no. 7–8, (2018), 1082-1099.

\bibitem{Leontev} A. F. Leont'ev,
Equations of infinite order with analytic solutions,
Math. Notes $\bf 10$ (1971), 585-590.

\bibitem{Lepson} B. Lepson,
On Hyperdirichlet series and on related questions of the general theory of functions,
Trans. Amer. Math. Soc., $\bf 72$ (1952), 18-45.

\bibitem{Levin} B. Ya. Levin,
Distribution of zeros of entire functions,
Amer. Math. Soc., Providence, R. I., 1964 viii+493 pp.

\bibitem{Lissy2014} P. Lissy,
On the Cost of Fast Controls for Some Families of Dispersive or Parabolic Equations in One Space Dimension,
SIAM J. Control. Optim., $\bf 52$ no. 4 (2014), 2651-2676.

\bibitem{Lissy2017} P. Lissy,
The cost of the control in the case of a minimal time of control: The example of the one-dimensional heat equation,
J. Math. Anal. Appl., $\bf 451$ (2017), 497-507.

\bibitem{Lissy2017b} P. Lissy,
Construction of Gevrey functions with compact support using the Bray-Mandelbrojt iterative process
and applications to the moment method in control theory.
Math. Control Relat. Fields, $\bf 7$ no. 1 (2017), 21-40.

\bibitem{Lissy2019} P. Lissy, I. Roven\c{t}a,
Optimal filtration for the approximation of boundary controls for the one-dimensional wave equation
using a finite-difference method,
Math. Comp., $\bf 88$ no. 315 (2019),  273-291.

\bibitem{LK} W. A. J. Luxemburg, J. Korevaar,
Entire functions and M\"{u}ntz-Sz\'{a}sz type approximation,
Trans. Amer. Math. Soc. $\bf 157$ (1971), 23-37.

\bibitem{2011MicuZuazua} S. Micu, E. Zuazua,
Regularity issues for the null-controllability of the linear 1-d heat equation,
Systems Control Lett. $\bf 60$ no. 6 (2011),  406-413.

\bibitem{2011Micu} S. Micu, J. H. Ortega, A. F. Pazoto,
Null-controllability of a Hyperbolic Equation as Singular Limit of Parabolic Ones,
J. Fourier Anal. Appl. $\bf 17$ no. 5 (2011), 991-1007.

\bibitem{2012Micu} S. Micu, I. Roven\c{t}a,
Uniform controllability of the linear one dimensional Schrödinger equation with vanishing viscosity,
ESAIM Control Optim. Calc. Var., $\bf 18$ no.1 (2012), 277-293.

\bibitem{2016Micu} S. Micu, I. Roven\c{t}a, L. E. Temereanc\u{a},
Approximation of the controls for the linear beam equation,
Math. Control Signals Syst. $\bf 28$ no. 2 (2016), article number 12.

\bibitem{2018Micu} S. Micu, T. Takahashi,
Local controllability to stationary trajectories of a Burgers equation with nonlocal viscosity,
J. Differential Equations, $\bf 264$ no. 5 (2018), 3664-3703.

\bibitem{Pinkus} A. Pinkus,
Density in Approximation Theory,
Surveys in Approximation Theory $\bf 1$ (2005), 1-45.

\bibitem{Pommerenke} Ch. Pommerenke,
Boundary Behaviour of Conformal Maps,
Berlin/New York, Springer-Verlag, 1992.

\bibitem{Redheffer}  R. M. Redheffer,
Completeness of Sets of Complex Exponentials,
Advances in Math. $\bf 24$ (1977), 1-62.

\bibitem{Rudin} W. Rudin,
Real and complex analysis,
Third edition, McGraw-Hill Book Co., New York, (1987) xiv+416 pp. ISBN 0-07-054234-1.

\bibitem{Sed1} A. M. Sedletskii,
Analytic Fourier Transforms and exponential approximations I,
Journal of Mathematical Sciences $\bf 129$ no. 6 (2005), 4251-4408.


\bibitem{Valiron} M. G. Valiron,
Sur les solutions des \'{e}quations diff\'{e}rentielles lin\'{e}aires d'ordre infini et a coefficients constants,
Ann. Ecole Norm. (3) $\bf 46$ (1929),  25-53.

\bibitem{Vidras} A. Vidras,
On a theorem of P\'{o}lya and Levinson,
J. Math. Anal. Appl. $\bf 183$ no. 1 (1994), 216-232.

\bibitem{Young} R. M. Young,
An introduction to Nonharmonic Fourier Series,
Revised first edition. Academic Press, Inc., San Diego, CA, (2001) xiv+234 pp. ISBN: 0-12-772955-0.

\bibitem{Z2011JAT} E. Zikkos,
The closed span of an exponential system in the Banach spaces $L^p(\gamma,\beta)$ and $C[\gamma,\beta]$,
J. Approx. Theory $\bf 163$ no. 9 (2011), 1317-1347.

\bibitem{Z2012CMFT} E. Zikkos,
An addendum to theorems of A. F. Leont'ev and L. Carleson on an infinite order differential equation on a real interval,
Comput. Methods Funct. Theory $\bf 12$ no. 2 (2012), 403-417.

\bibitem{Z2015JMAA} E. Zikkos,
Solutions of infinite order differential equations without the grouping
phenomenon and a generalization of the Fabry-P\'{o}lya theorem,
J. Math. Anal. Appl. $\bf 423$ no. 2 (2015), 1825-1837.

\bibitem{Zuazua} E. Zuazua,
Controllability of Partial Differential Equations,
3rd cycle. Castro Urdiales (Espagne), (2006), pp.311. ffcel-00392196f.

\end{thebibliography}
\end{document}